\documentclass[10pt]{amsart}
\usepackage{graphicx,amssymb,amsfonts,amsmath,amsthm,newlfont}
\usepackage{epsfig,url}
\usepackage{color}
\usepackage{enumitem}

\usepackage[all,2cell]{xy} \UseAllTwocells \SilentMatrices

\newtheorem{thm}{Theorem}[section]
\newtheorem{cor}[thm]{Corollary}
\newtheorem{lem}[thm]{Lemma}
\newtheorem{prop}[thm]{Proposition}
\theoremstyle{definition}
\newtheorem{defn}[thm]{Definition}
\newtheorem{conj}{Conjecture} 

\newtheorem{ex}[thm]{Examples}

\newtheorem{example}[thm]{Example}

\theoremstyle{remark}
\newtheorem{rem}[thm]{Remark}
\numberwithin{equation}{section}

\newcommand{\Z}{\mathbb Z}
\newcommand{\C}{\mathbb C}

\newcommand{\R}{\mathbb R}
\newcommand{\N}{\mathbb N}
\newcommand{\Pro}{\mathbb P}

\newcommand{\gr}{\mathrm{gr}}
\newcommand{\Ext}{\mathrm{Ext}}
\newcommand{\csmb}{\mathrm{cmb}}

\font \rus= wncyr10
\newcommand{\sha}{\, \hbox{\rus x} \,}

\newcommand{\MT}{\mathcal{MT}}

\newcommand{\an}{\mathrm{an}}

\newcommand{\Isom}{\mathrm{Isom}}

\newcommand{\ev}{\mathrm{ev}}

\newcommand{\MM}{\mathcal{MM}}

\newcommand{\zetam}{\zeta^{ \mathfrak{m}}}

\newcommand{\Q}{\mathbb Q}
\newcommand{\Li}{\mathrm{Li}}
\newcommand{\Lo}{\mathcal{L}}
\newcommand{\U}{\mathcal{U}}

\newcommand{\s}{\mathbf{s}}

\newcommand{\To}{\longrightarrow}

\newcommand{\A}{\mathbb{A}}
\newcommand{\G}{\mathcal{G}}

\newcommand{\tone}{\overset{\rightarrow}{1}\!}

\newcommand{\Aut}{Aut}

\newcommand{\smb}{\mathrm{smb}}

\newcommand{\dR}{\mathfrak{dr}}

\newcommand{\Or}{\mathcal{O}}

\newcommand{\V}{\mathcal{V}}

\newcommand{\BB}{\mathbb{B}}

\newcommand{\dec}{\mathrm{dec}}

\newcommand{\mm}{\mathfrak{m} }

\newcommand{\HH}{\mathcal{H} }
\newcommand{\Lef}{\mathbb{L} }

\newcommand{\Ao}{\mathcal{A} }
\newcommand{\GG}{\mathbb{G} }

\newcommand{\id}{\mathrm{id} }

\newcommand{\T}{\mathcal{T} }

\newcommand{\X}{ X}
\newcommand{\Y}{ Y}

\newcommand{\Lie}{\mathrm{Lie}\,}

\newcommand{\LL}{\mathbb{L}}

\newcommand{\tr}{\mathrm{tr}}

\newcommand{\per}{\mathrm{per}}

\newcommand{\uu}{\mathfrak{u}}

\newcommand{\Spec}{\mathrm{Spec} \,}

\newcommand{\Pe}{\mathcal{P}}

\newcommand{\comp}{\mathrm{comp}}

\newcommand{\Hom}{\mathrm{Hom}}

\begin{document}
\author{Francis Brown}
\begin{title}[Notes on motivic periods]{Notes on motivic periods}\end{title}
\maketitle

\begin{abstract} The second part of a set of notes based on lectures given at the IHES in 2015 on Feynman amplitudes and motivic periods. 
\end{abstract} 

These notes started out as an appendix to \cite{Cosmic}.
The  aim is merely to provide some basic definitions and  tools to describe a
 certain class of numbers and functions defined  by integrals of algebraic forms over algebraic domains.

\section{Introduction}
A period, according to an elementary definition of Kontsevich and Zagier, is a complex number whose real and imaginary parts are  given by an
 integral of a rational function over a domain defined by polynomial inequalities \cite{KoZa}.
 
 An example is the number 
 $$\pi = \int_{x^2+y^2\leq 1 } dx dy  \ .$$
 The number $\pi$ is ubiquitous and clearly deserves a name of its own. In these notes, we seek to address the problem of how to   describe general periods,  and families of periods depending on parameters. 
 \vspace{0.05in}
 
 Following Grothendieck, periods can  be viewed as the coefficients of a comparison isomorphism between two cohomology theories. 
For simplicity, consider a period 
 $I=\int_{\sigma} \omega$, where $\sigma$ is a closed cycle representing  an element in the Betti homology $H_n(X(\C);\Q)$   of a smooth affine algebraic variety $X$ over $\Q$, 
 and $\omega$  is a regular differential form over $\Q$ representing an  element in the  algebraic de Rham  cohomology $H^n_{dR}(X;\Q)$. The Grothendieck-de Rham comparison isomorphism:
 $$\mathrm{comp}: H^n_{dR}(X;\Q) \otimes_{\Q} \C \overset{\sim}{\To} H^n(X(\C);\Q)\otimes_{\Q} \C  $$
 is induced by integration of algebraic  differential forms over closed cycles. 
   This data can be represented  in a  category $\T$ of triples $(V_B, V_{dR}, c)$ where $V_B$, $V_{dR}$ are finite-dimensional vector spaces over $\Q$, and
$c:V_{dR}\otimes_{\Q} \C \overset{\sim}{\rightarrow} V_{B} \otimes_{\Q} \C$ is  an isomorphism.  The  period integral can be encoded by algebraic data:   an object
$( H^n(X(\C);\Q), H^n_{dR}(X;\Q),\mathrm{comp})$ in the category $\T$, together with a class $[\sigma]$, an element of the dual of the first vector space,  and $[\omega]$,
an element of the second. 
Define a  space of periods $\Pe^{\mm}_{\T}$ of $\T$ to be  the $\Q$-vector space spanned by symbols 
\begin{equation} \label{introtriples} 
((V_B, V_{dR}, c) , \sigma, \omega)\quad  \hbox{ where } \sigma\in V_B^{\vee}\ , \  \omega \in V_{dR} 
\end{equation} 
modulo a certain equivalence relation (linearity in $\sigma$, $\omega$, and functoriality with respect to morphisms in $\T$), which reflects  the fact that  periods can have different integral representations. The space $\Pe^{\mm}_{\T}$ forms  a  ring, and   is  equipped with   a homomorphism called the period map
 $$\per : \Pe^{\mm}_{\T} \To \C\ ,$$
 which sends the class of $(\ref{introtriples})$ to $\sigma(c(\omega))$. In this way, we obtain an element $I^{\mm}$ in the ring $\Pe^{\mm}_{\T}$ whose
 period $\per(I^{\mm})=I$ is the  integral we started  off with. 
 
 The crucial point   is that $\T$ is a Tannakian category, which automatically endows $\Pe_{\T}^{\mm}$ with the action of  a group\footnote{in fact, two groups, one for each fiber functor Betti or de Rham.}, whose action on elements such as $I^{\mm}$ is a prototype for  a `Galois group of periods'. 
 In more classical language, $\Pe^{\mm}_{\T}$ is simply the affine ring of tensor isomorphisms from the de Rham to Betti fiber functors on $\T$. 
  The idea of a `Galois theory of periods' has its origins in Grothendieck's Tannakian philosophy of mixed motives,  and has been developed by Nori, Kontsevich, Andr\'e, and most recently Ayoub, Huber and M\"uller-Stach.
  The more common, and most sophisticated, approach  to this subject  involves replacing $\T$ with a suitable category of mixed motives. Several different approaches are possible. In this context, Grothendieck's period conjecture states that the period map is injective.

 The naive category $\T$ defined above is the simplest possible framework in which  one can  set up a working Galois theory of periods. However, much is gained  by adding just a little more; namely the requirement that $V_B, V_{dR}$ are equipped with filtrations  forming a mixed Hodge structure.  This leads to a  category $\HH$ of triples $(V_B, V_{dR}, c)$ carrying some extra data  (Hodge  and weight filtrations, and a  real Frobenius involution, which encodes the action of complex conjugation).  This category  was first introduced by Deligne who proved that it is Tannakian. 
The upshot is that one obtains a `ring of $\HH$-periods' $\Pe^{\mm}_{\HH}$ defined entirely  in terms of linear algebra which encapsulates many fundamental features of periods. A `motivic' period, for us, is then an element of $\Pe^{\mm}_{\HH}$  that comes from the cohomology of an algebraic variety in a specific way. 
We shall use the adjective `motivic'  for such a period, although much of this paper is in fact Hodge-theoretic.  The ring $\Pe^{\mm}_{\HH}$ has the universal property that the periods of any reasonable category of mixed motives  admitting Betti and de Rham realisations will factor through it. \vspace{0.05in}

These notes  explore some simple consequences of this  general notion of $\HH$-period, and explain how to compute using these objects.
For this, one is obliged to use the language of Hopf algebras and matrix coefficients,  since the fundamental objects are not the (motivic) Galois groups themselves
but their  affine rings.  This  elementary formalism already  enables one to  attach a   panoply of invariants to an $\HH$-period, such as its weight, rank, dimension, Hodge  polynomial, and more  elaborate notions such as its single-valued versions, unipotency filtration and  Galois groups. A glossary of non-standard terms is given in  \S\ref{sectGlossary}. Much of this work is motivated by applications to physics, where some of these concepts  (such as the notion of `transcendental weight') have  already taken root and have several applications, and we felt there was a need to place these notions in a rigorous context.
\vspace{0.05in}

One key point, that must be mentioned from the outset,  is that all the concepts in these notes  translate immediately into a suitable Tannakian  subcategory of mixed motives over the rationals $\MM_{\Q}$, whenever it is  defined. Any reasonable candidate for such a category has Betti and de Rham realisations, so we obtain a map of rings of periods (rings  $\Pe^{\mm}_{\bullet}=\Or(\mathrm{Isom}^{\otimes}_{ \bullet} (\omega_{dR}, \omega_B))$, with $\bullet = \MM(\Q), \HH$)
\begin{equation} \label{introPemtoH}
\Pe^{\mm}_{\MM_{\Q}} \To \Pe^{\mm}_{\HH}\ .
\end{equation} 
All the constructions in this paper can be pulled  back to the ring on the left   without  any difficulty. Possible choices of categories include  Nori's category of motives, or  the abelian category of mixed Tate motives over number fields  \cite{LevineTate}.
If one wants to prove  \emph{independence} of periods, then  it is enough to work 
 in the elementary category $\HH$, and  defining invariants of  $\Pe^{\mm}_{\HH}$  provides tools to do precisely that.\footnote{the relations should come through the back door as a consequence of bounds on algebraic $K$-theory  and the Tannakian formalism.  The theory of multiple zeta values provides  many examples of relations between  periods which can be proved
 by analytic methods but  for which a proof using algebraic correspondences are not presently known. }

Another important point is that if the Hodge realisation $\MM_{\Q} \rightarrow \HH$ is fully faithful, as one  hopes, then $(\ref{introPemtoH})$
is injective and the action of the motivic Galois group of $\MM_{\Q}$ is already  correctly calculated by the action of the `elementary' Galois group of $\HH$. This is the case for mixed Tate motives over number fields, so we can identify motivic periods in $\Pe^{\mm}_{\MT(
\Q)}$, for example,  with their image in $\Pe^{\mm}_{\HH}$ with impunity.
For  these reasons, we choose to work unconditionally  in $\HH$,  whilst waiting for the dust to settle on the final definition of $\MM_{\Q}$. 

\subsection{Contents}  In \S\ref{sectGeneralities}  we gather some properties of Tannakian categories, matrix coefficients and unipotent algebraic groups 
needed for the rest of the notes. This section, which is one of the most technical,  can be referred to when needed. If the reader is only interested in periods over $\Q$, then \S\ref{sectTannakianCase}  can be  simplified by taking the rings  $B_1,B_2,k$ to be equal to $\Q$. In \S\ref{sectMotperQ}  we attach some basic invariants to elements of $\Pe^{\mm}_{\HH}$. Section \ref{sectFurther} defines further concepts including single-valued versions of de Rham $\HH$-periods generalising the single-valued multiple zeta values of \cite{SVMP}, and a certain projection map, which can be used, for example, to   infer results about $p$-adic periods of mixed Tate motives from their complex periods.
Section \ref{sectExamples} offers some  basic examples and can be read in parallel with the previous sections for illustration. In \S\ref{sectClassif}, we study  a decomposition map which enables us to break up an arbitrary $\HH$-period into elementary pieces.  
It takes the form of a canonical isomorphism
$$\Phi: \gr^{C}_{\bullet }\Pe^{\mm}_{\HH} \overset{\sim}{\To} \Pe^{\mm}_{\HH^{ss}} \otimes_{\Q} T^c(H)$$
where $\Pe^{\mm}_{\HH^{ss}}$ is the ring of pure or semi-simple periods;  $T^c$ is  the tensor  coalgebra (or shuffle algebra) graded by length of tensors;  $H$ is a certain  explicitly-defined 
vector space which is a direct sum of pure Hodge structures; and $C$ (for coradical) is a certain filtration on $\HH$-periods by unipotency degree.   

\begin{example} The map $\Phi$  is a  generalisation to all periods of the `highest length' part of the map $\phi$ of \cite{MZVdecomp},  which assigns to any motivic multiple zeta value an element of a shuffle algebra on certain symbols. 
  For example,
   $$\Phi (\zetam(2n+1)) = 1\otimes f_{2n+1}  \qquad \hbox{ for all } n\geq 1$$  where $f_{2n+1} \in H$ are certain elements which span a copy of $\Q(-2n-1)$.  Let $\Lef^{\mm}$ denote the motivic period corresponding to $2\pi i$. Since it is  the period of a pure object, it satisfies $\Phi( \Lef^{\mm}) = \Lef^{\mm} \otimes 1$.      From these ingredients,  one can then decompose the   $\HH$-periods corresponding to multiple zeta values.
For example, one finds   
  $$\Phi(\zetam(2,3)) = 3  (\Lef^{\mm})^2 \otimes f_3    \qquad \hbox{ and } \qquad \Phi(\zetam(3,2)) = -2 (\Lef^{\mm})^2 \otimes  f_3 \ ,$$
As a further example, $\zetam(3,5) \in C_2 \Pe^{\mm}_{\HH}$  satisfies $\Phi(\zetam(3,5)) = -5 \otimes (f_5 \otimes f_3)$.
  \end{example}
 The map $\Phi$ may provide a  useful model with which to think about  the structure of periods. 
 In section \ref{sectFamilies}, we very briefly describe how one might  set up similar notions for period integrals depending on parameters.
 The notion of `families of $\HH$-periods' is  very rich, and we barely  do it justice.
 In section \S\ref{sectSymbols}   we associate  various notions of symbol associated to a family of  $\HH$-periods. For the symbol
to exist, the family must underly a (globally) unipotent vector bundle with integrable connection, which always holds in the mixed Tate case.
 The symbol is a tensor product of differential $1$-forms modulo some relations. In the words of one of the referees, this should clarify many clumsy approaches in the recent physics literature.  One can think of the decomposition map as a kind of  analogue of a symbol for constant periods. 
 We also define single-valued versions of families of $\HH$-periods, again with applications to physics in mind. 
  Finally, in \S\ref{sect:  geometric examples} we provide some  geometric  examples and prove some technical results  required for \cite{Cosmic}.  Lastly, we discuss some examples in  the case of the projective line minus three points
  for illustrative purposes. 
\vspace{0.02in}

Further background about periods can be found in  the book project \cite{PeriodsBook}, the surveys \cite{KoZa} and \cite{An4} and references therein.
An obvious omission from this paper is a discussion of Tannakian ideas relating to limiting mixed Hodge structures and the regularisation of divergent (motivic) periods. 

Since these notes are based on lectures, there is not much in the way of technical argumentation. In order to keep the length of the paper down, we provide essential technical arguments only where constructions are new or absent from the literature, and omit proofs which are straightforward or  well-known.

\section{Generalities} \label{sectGeneralities}
\subsection{Recap on Tannakian categories}
Let  $k$  be a field. Following  \cite{Tannaka} \S1.2, a \emph{tensor  category} $\T$ over $k$  is a  $k$-linear  rigid abelian tensor category, which is  ACU and satisfies $k \overset{\sim}{\rightarrow} \mathrm{End}(1)$ (\cite{Tannaka}, \S\S2.1, 2.7, 2.8). A \emph{fiber functor} from $\T$ to a scheme  $S$ over $k$ is an exact $k$-linear functor from $\T$ to the category of quasi-coherent sheaves on $S$, which is compatible with the tensor product and the ACU constraints.  A \emph{Tannakian category} is  a tensor category equipped with a fiber functor  $\omega$ to a non-empty scheme $S$. If $S = \Spec(B)$ is affine, then due to rigidity, $\omega$ necessarily  
lands in the  category of projective $B$-modules of finite type.

\begin{thm} \label{Tannakatheorem} (\cite{Saavedra}, corrected in \cite{Tannaka}). Let $\T$ be a Tannakian category with a fiber functor to $S$, a non-empty scheme over $k$.  Then 
the groupoid of tensor automorphisms $\Aut^{\otimes}_{\T}(\omega)$ is faithfully flat on $S\times S$, and  $\omega$ defines an equivalence of categories from $\T$ to the category of representations of  $\Aut^{\otimes}_{\T}(\omega)$.
\end{thm} 

In the applications, $k=\Q$ and all Tannakian categories we consider  will be neutralised by the Betti realisation. They  will possess  a   fiber functor $\omega_B$ to the category of vector spaces over $\Q$, and  a second fiber functor $\omega_{dR}$ to a smooth scheme $S$ over $\Q$. The space $S(\C)$ will be the domain for a family of periods.

\subsection{Matrix coefficients.}  The following construction is paraphrased from \cite{Tannaka} \S4.7.
 \label{sectMatrixcoeffs} Let $\T$ be a (small)  category over $k$, and let $B_1, B_2$ be two $k$-algebras, not necessarily commutative. The following discussion is excessively general
 for our purposes,   but this  actually simplifies the presentation.   Let 
 $\omega_i$, for $i=1,2$,  be a functor from $\T$ to the category of right projective $B_i$ modules of finite type. 

\begin{defn} 
A \emph{matrix coefficient} \label{gloss: matrixcoeff} in $\T$ is a   triple $ (M, f, v)$, where $M$ is an object of $\T$,   $f\in \omega_1(M)^{\vee}$ and $v\in \omega_2(M)$.
Consider  the following $k$-vector space 
$$\Pe^{\omega_1, \omega_2}_{\T} = \langle (M,f,v) \rangle_{k} / \sim$$
 spanned by symbols $(M,f,v)$ 
 modulo the following relations:
 \vspace{0.1in}

$(i)$. (Bimodule structure).   For all $\lambda_1,\lambda_2\in B_2$, and $\mu_1,\mu_2 \in B_1$, 
\begin{eqnarray}
(M,f,  v_1 \lambda_1+  v_2 \lambda_2) & \sim &   (M, f, v_1 )\lambda_1 + (M, f, v_2)   \lambda_2\nonumber \\
(M,\mu_1 f_1 + \mu_2 f_2, v) & \sim &  \mu_1 (M,  f_1 ,v ) + \mu_2 (M,  f_2,v)  \nonumber 
\end{eqnarray} 
Furthermore, if  $\lambda \in k$, then $(M,f,v) \lambda \sim \lambda (M,f,v)$. 
Thus $\Pe^{\omega_1, \omega_2}_{\T}$ is a left $B_1$-module and right $B_2$-module, whose induced $k$-vector space structures coincide. We shall 
call such an object a $(B_1,B_2)$-bimodule over $k$. 
\\

$(ii)$.  (Morphisms). If $\phi: M_1 \rightarrow M_2$ is a morphism in $\T$ then 
\begin{equation} 
(M_1, f_1, v_1) \sim (M_2, f_2, v_2) \nonumber 
\end{equation} 
\qquad \qquad whenever  $v_2 = \omega_2(\phi)(v_1)$ and $f_1 = (\omega_1(\phi))^t f_2$, where $t$ is the transpose.
\vspace{0.1in}

\noindent Denote the equivalence class of $(M,f,v)$ by $[M,f, v]_{\T}^{\omega_1, \omega_2}$, or simply $[M,f, v]^{\omega_1, \omega_2}$ when there is no ambiguity about the ambient category. \end{defn}

The space $\Pe^{\omega_1, \omega_2}_{\T}$ is denoted by  $L_k(\omega_1,\omega_2)$ in  \cite{Tannaka}. We use the letter  $\Pe$ because we shall eventually  think of elements of $\Pe^{\omega_1, \omega_2}_{\T}$ as periods.  Condition $(i)$ implies  that for all objects $M$ in $\T$,  there is a  morphism $f\otimes v \mapsto [M,f,v]^{\omega_1,\omega_2}$
\begin{equation} \label{omegaMuniversalmap}
\omega_1(M)^{\vee} \otimes_k \omega_2(M) \To \Pe^{\omega_1, \omega_2}_{\T}
\end{equation}
of $(B_1,B_2)$-bimodules over $k$,
which is functorial in $M$.   This is a  universal property satisfied  by  $\Pe^{\omega_1,\omega_2}_{\T}$: any such collection of functorial maps
from $\omega_1(M)^{\vee} \otimes_k \omega_2(M)$ for all $M$ into a $(B_1,B_2)$-bimodule  over $k$ factors through $\Pe^{\omega_1,\omega_2}_{\T}$.

 There is a natural $k$-linear map  $k \rightarrow \omega_2(M) \otimes_{B_2} \omega_2(M)^{\vee}$,
 which sends $1$ to the element corresponding to the identity via the isomorphism
$$\omega_2(M) \otimes_{B_2}  \omega_2(M)^{\vee} \overset{\sim}{\To} \mathrm{Hom}_{B_2} ( \omega_2(M), \omega_2(M))\ .  $$
Writing $ \omega_1(M)^{\vee} \otimes_k \omega_2(M)  =  \omega_1(M)^{\vee}\otimes_k  k \otimes_k\omega_2(M)  $ we deduce a map
 $$ \omega_1(M)^{\vee} \otimes_k \omega_2(M)    \To   \omega_1(M)^{\vee} \otimes_k \omega_2(M) \otimes_{B_2} \omega_2(M)^{\vee} \otimes_k \omega_2(M)$$
 which in turn induces a morphism of $(B_1,B_2)$-bimodules over $k$:
 \begin{equation} \label{eqn: coaction}
\Delta : \Pe^{\omega_1, \omega_2}_{\T}    \To  \Pe^{\omega_1, \omega_2}_{\T}  \otimes_{B_2} \Pe^{\omega_2, \omega_2}_{\T}\ .
\end{equation}
One verifies that it defines  a right coaction of $\Pe^{\omega_2, \omega_2}_{\T}$ on $\Pe^{\omega_1, \omega_2}_{\T}$. 
 Since $\omega_2(M)$ is projective of finite type, we can write $\omega_2(M)$ as a direct summand of $B_2^n$ for some $n$ and let $e_i$  (respectively $e_i^{\vee}$) for $1\leq i \leq n$ be  coordinates of  $B_2^n \rightarrow \omega_2(M)$ (respectively  $\omega_2(M) \rightarrow B_2^n$).   
The element $\sum_{i=1}^n  e_i \otimes e_i^{\vee}$ represents the identity on $\omega_2(M)$. This gives the following   formula for $(\ref{eqn: coaction})$ on the level of matrix coefficients:
\begin{equation}   
\Delta {[}M, f, v]^{\omega_1, \omega_2}  =  \sum_{i} [M, f, e_i]^{\omega_1, \omega_2} \otimes  [M, e^{\vee}_i, v]^{\omega_2, \omega_2}  
\end{equation} 
   In a similar way, the space $\Pe^{\omega_1,\omega_1}_{\T}$ naturally coacts on $\Pe^{\omega_1,\omega_2}_{\T}$ on the left.

\subsection{Tannakian case} \label{sectTannakianCase} Now suppose that $B_1,B_2$ are commutative, $\T$ is a tensor category  and $\omega_i$ is a fiber functor to $\Spec(B_i)$, for $i=1,2$. 
The tensor structure on $\T$ implies that 
 $\Pe^{\omega_1, \omega_2}_{\T}$ is a commutative $k$-algebra. In formulae:
$$[M_1, f_1, v_1]^{\omega_1, \omega_2} \times [M_2, f_2, v_2]^{\omega_1, \omega_2}  = [M_1\otimes M_2, f_1\otimes f_2, v_1\otimes v_2]^{\omega_1, \omega_2} \ ,$$
which is well-defined as one easily checks.  

Consider the affine scheme over $B=B_1\otimes_k B_2$ defined by 
$$\Hom_{\T}^{\otimes}(\omega_2,\omega_1) = \Spec \, \Pe^{\omega_1,\omega_2}_{\T}\ .$$
 If  $R$ is any commutative $B$-algebra then its $R$-points are given (\cite{Tannaka}, proposition 6.6) by collections of 
 homomorphisms of $R$-modules 
 \begin{equation} \label{phiM}  {\phi}_M: R \otimes_{B} \big( B_1 \otimes_k \,\omega_2(M) \big) \To  \big( \omega_1(M) \otimes_{k} B_2 \big) \otimes_{B} R\ 
 \end{equation}
 which are functorial in $M$ and respect the tensor product. The corresponding homomorphism $\phi: \Pe^{\omega_1,\omega_2}_{\T} \rightarrow R$ is given on matrix coefficients  by the formula  
 $$  \phi [M, f, v]^{\omega_1,\omega_2}  = f(\phi_M(  v)) \ . $$
 It follows from the existence and properties of duals in $\T$ that  the $\phi_M$'s are  automatically isomorphisms and therefore 
   $\Isom_{\T}^{\otimes}(\omega_2, \omega_1)  \overset{\sim}{\rightarrow} \Hom_{\T}^{\otimes}(\omega_2, \omega_1) $ is an isomorphism  and we indeed have  $\Pe^{\omega_1,\omega_2}_{\T} = \Or( \Isom_{\T}^{\otimes}(\omega_2, \omega_1))$.

Applying the above in the case when both fiber functors are equal  implies that   for $i=1,2$, $\Pe^{\omega_i,\omega_i}_{\T}$ is a commutative bialgebra over $B_i$. 
It has  an antipode, which on matrix coefficients is the involution $S: [M,f,v]^{\omega_i,\omega_i} \rightarrow [M^{\vee}, v,f]^{\omega_i, \omega_i}$, unit $[1,1,1]^{\omega_i,\omega_i}$  and counit $\varepsilon: [M,f,v]^{\omega_i,\omega_i} \mapsto f(v)$, 
and is a Hopf algebra with respect to these structures.  It therefore defines an affine group scheme  $\mathrm{Aut}_{\T}^{\otimes}(\omega_i)$ over $B_i$ which we shall denote by 
$$G_{\T}^{\omega_i} = \mathrm{Spec}\, \Pe^{\omega_i, \omega_i}_{\T} \ .$$
 If $R$ is a commutative $B_i$-algebra, then an $R$-valued point $g\in G_{\T}^{\omega_i}(R)$ can be viewed
as a functorial  collection of isomorphisms 
 $$g_M:  R \otimes_{B_i} \omega_i(M) \overset{\sim}{\To} \omega_i(M) \otimes_{B_i} R$$ for every object $M$  of $\T$, which are compatible with the tensor product.  The homomorphism from $P^{\omega_i,\omega_i}_{\T}$ to $R$ is defined by   the formula
$  g [M, f, v]^{\omega_i, \omega_i} = f ( g_M v)$, and every such homomorphism arises in this way.

We shall use the main theorem \ref{Tannakatheorem} in the following case:

\begin{thm} \label{thm:  Tannaka}   The functor  $\omega_i: \T \rightarrow \mathrm{Rep}\, (G_{\T}^{\omega_i})$
is an equivalence of categories. 
\end{thm}

The scheme
$ \Isom^{\otimes}_{\T}(\omega_2, \omega_1 ) =  \mathrm{Spec}\, \Pe^{\omega_1, \omega_2}_{\T}$  is a $G_{\T}^{\omega_1}\times G_{\T}^{\omega_2}$ - bitorsor, with $G_{\T}^{\omega_1}$ acting on the left, $G_{\T}^{\omega_2}$ on the right. The  action of $G_{\T}^{\omega_1}\times G_{\T}^{\omega_2}$ on $\Pe^{\omega_1, \omega_2}_{\T}$ is given as  follows. 
The right coaction $(\ref{eqn: coaction})$ is equivalent to a left action of $G_{\T}^{\omega_2}$
\begin{equation} \label{eqn: leftgroupaction}
G_{\T}^{\omega_2} \times \Pe^{\omega_1, \omega_2}_{\T} \To   \Pe^{\omega_1, \omega_2}_{\T}
\end{equation} 
via  the formula $g(\xi) = (\id \otimes g)\Delta(\xi)$.  We shall call  $(\ref{eqn: leftgroupaction})$  the `Galois action' on $\Pe^{\omega_1,\omega_2}_{\T}$. 
 Concretely, if $R$ is a $B_2$-algebra, then  a point $g\in G_{\T}^{\omega_2}(R)$ acts by
$$ g [ M, f, v]^{\omega_1, \omega_2} = [ M, f, g_M ( v)]^{\omega_1, \omega_2}\ .$$
There is  correspondingly a right action on $\Pe^{\omega_1,\omega_2}_{\T}$  by  $G^{\omega_1}_{\T}$ which acts  on the element $f \in \omega_1(M)^{\vee}$ on the right.  Depending on the situation, it can happen that one or other action of $G^{\omega_i}_{\T}$ on $\Pe^{\omega_1,\omega_2}_{\T}$ plays a more important role than the other.

\subsection{Minimal  objects}  \label{sectminimalobject}
Return to the situation of \S\ref{sectMatrixcoeffs}. The following useful lemma  will be used to attach quantities to  motivic periods. 

Suppose that  $\T$ is an abelian category such that  every object of $\T$ has finite length, and suppose that the functors $\omega_1,\omega_2$ are exact. 

\begin{lem} 
 Let $\xi \in \Pe^{\omega_1,\omega_2}_{\T}$. There exists a smallest object $M(\xi)$ of $\T$, unique up to isomorphism, such  that $\xi$ is a matrix coefficient of $M(\xi)$, i.e.,   every object of $\T$ of which   $\xi$ is a matrix coefficient admits a subquotient isomorphic to  $M(\xi)$.\end{lem} 

\begin{proof} First observe that if $[M_1, f_1, v_1]$ and $[M_2, f_2, v_2]$ are two matrix coefficients in $\T$, then the relations imply that 
$$[M_1, f_1, v_1] + [M_2, f_2, v_2] = [ M_1 \oplus M_2, (f_1,f_2), (v_1,v_2)]\ .$$
Similarly, every linear combination of  elements $\xi \in \Pe^{\omega_1,\omega_2}_{\T}$ can be represented 
by a single matrix coefficient $[M, f, v]$.  In fact, $\Pe^{\omega_1,\omega_2}_{\T}$ can alternatively be defined as  the set of equivalence classes $(M,f,v)$ with respect to 
relation $(ii)$ of \S\ref{sectMatrixcoeffs}, equipped  with a  well-defined bimodule structure  $(i)$.
 
Now let $M$ be an object of $\T$, and $f\in \omega_1(M)^{\vee}, v\in \omega_2(M)$. Define
$M_v$
to be the smallest subobject of $M$ such that $v$ is in the image of $\omega_2(M_v)$. It exists, because if every subobject $N\subsetneq M$ satisfies $v\notin \mathrm{Im} \, \omega_2(N)$, then 
$M_v= M$, otherwise we can replace $M$ by any subobject $N$ such that $v\in \mathrm{Im}\, \omega_2(N)$, and proceed by induction on the length. It is unique up to isomorphism;  
if $v \in  \omega_2(N_1)$ and $v\in \omega_2(N_2)$, where $N_1,N_2$ are both minimal subobjects of $M$, then by writing 
$N_1 \cap N_2 = \ker (N \rightarrow N/N_1 \oplus N/N_2)$ and using 
the exactness of $\omega_2$, it follows that 
$v \in  \omega_2(N_1 \cap N_2) \cong \omega_2(N_1) \cap \omega_2(N_2)$ and hence by minimality $N_1 \cong N_1 \cap N_2 \cong N_2$.  Similarly, let ${}_f M$ be the smallest quotient object of $M$ such that $ f\in \omega_1({}_f M)^{\vee}$. 

Now consider a morphism $\phi: M \rightarrow M'$.  We first show that 
$$\phi(M_v)\cong \phi(M_v)_{\phi(v)} \cong M'_{\phi(v)}\ .$$
The second isomorphism holds  since $\phi(M_v)$ is a subobject of $M'$. Note that $\omega_2(\phi(M_v))$ contains $\phi(v)$.
We show that $\phi(M_v)$  is minimal for this property.  For if  $N\subset \phi(M_v)$ is a 
subobject such that $\phi(v) \in \omega_2(N)$, then $\phi^{-1}(N):= \ker (M_v \rightarrow \phi(M_v)/N)$ satisfies $v\in \omega_2(\phi^{-1}(N))$, and
so by definition of $M_v$ we have $\phi^{-1}(N) \cong M_v$, and $N\cong  \phi(M_v)$. This proves the first isomorphism.

It follows that if $\phi: M \rightarrow M'$ is surjective, so too is $\phi: M_v \rightarrow M'_{\phi(v)}$. It follows from the definition
that if $\phi$ is injective, $M_v \cong M'_{\phi(v)}$ is an isomorphism.  The same statement holds for ${}_fM$ with the words
injective and surjective interchanged.

Now apply the first remark   to the surjective map $M_v \rightarrow \!\!\!\!\!\rightarrow {}_f(M_v)$. Denote the image of $v$ in $\omega_2({}_f(M_v))$ by $v$ also. We obtain a commutative diagram
$$
\begin{array}{ccc}
(M_v)_v  &  \rightarrow \!\!\!\!\!\rightarrow    & ({}_f(M_v))_v  \\
 \downarrow &   & \downarrow  \\
   M_v  &  \rightarrow \!\!\!\!\!\rightarrow   &  {}_f(M_v)
\end{array}
$$
where the two vertical maps are injective. Since $(M_v)_v = M_v$, it follows that the  vertical map on the right is an isomorphism. On the other
hand, applying the second remark to $M_v \hookrightarrow M$, we obtain an injection ${}_f(M_v) \hookrightarrow {}_fM$, and hence an isomorphism
$ ({}_f (M_v))_{v} \cong ({}_fM)_v$. We conclude that 
$${}_f (M_v) \cong ({}_fM)_v$$
and we shall subsequently denote this object simply by ${}_f M_v$.  

We now show that the map which assigns to a matrix coefficient $(M,f,v)$ the isomorphism class of ${}_f M_v$ respects the equivalence relation
$(ii)$.  Consider an equivalence $ [ M, f, v] = [N, f', v']$ arising from a morphism $\phi: M \rightarrow N$.   
In other words, $f'= \phi^t(f)$ and $v' = \phi (v)$. If   $\phi$ is injective, $M_v \cong N_{v'}$ by definition and hence
${}_f M_v  \cong {}_f(M_v) \cong {}_f (N_{v'})   \cong {}_{f'} N_{v'} $. When $\phi$ is surjective,   we again have  
${}_f M_v \cong {}_{f'} N_{v'} $ by a  similar argument.  Since any morphism can be expressed as a composition of an injection  and surjection, 
we deduce that  the isomorphism class of $M(\xi):= {}_fM_v$ only depends on the equivalence class $\xi = [M,f,v].$ It follows from the definitions that the morphisms  $M\rightarrow {}_fM \leftarrow {}_fM_v$ define an equivalence $\xi = [ {}_fM_v, f, v]$. 
\end{proof} 

Now suppose that $\T$   satisfies the more stringent conditions of \S\ref{sectTannakianCase}, and is in particular Tannakian. This implies that every object has finite length.

Let  $\xi \in \Pe_{\T}^{\omega_1, \omega_2}$. The $G_{\T}^{\omega_2}$ representation it generates is the $\omega_2$-image of an object of $\T$,
by the Tannaka theorem. We show that it is isomorphic to $M(\xi)$.

\begin{cor} \label{corGomegarep}  Consider  $\xi \in \Pe_{\T}^{\omega_1, \omega_2}$.
Let $\langle G_{\T}^{\omega_2} \xi \rangle_{B_2}$ (respectively   $\langle \xi\, G_{\T}^{\omega_1}\rangle_{B_1}$)  denote the representation of $G_{\T}^{\omega_2}$ (resp. $G_{\T}^{\omega_1}$) it generates. Then 
\begin{equation} 
\omega_2 (M(\xi)) \cong \langle G_{\T}^{\omega_2} \xi \rangle_{B_2}  \qquad \hbox{ and }  \qquad\omega_1(M(\xi)) \cong \langle \xi G_{\T}^{\omega_1}\rangle_{B_1} \ . 
\end{equation}

\end{cor}
\begin{proof}
The morphisms $(\ref{omegaMuniversalmap})$ induce functorial  morphisms 
$$\omega_1(M)^{\vee} \otimes \omega_2(M) \To \Pe^{\omega_1,\omega_2}_{\T}$$
in the (ind)-category of left $G^{\omega_2}_{\T}$-representations. By the Tannaka theorem, 
there exists an ind-object $ \Pe^{\omega_1,\bullet}_{\T}$ of $\T$ whose $\omega_2$-image is $\Pe^{\omega_1,\omega_2}_{\T}$, and the 
above morphisms are the $\omega_2$-image of a family of functorial morphisms 
$$\omega_1(M)^{\vee} \otimes M \To \Pe^{\omega_1,\bullet}_{\T}\ $$
of ind-objects of $\T$. 
Apply   $\omega_1$ to  obtain a family of functorial morphisms
$$\omega_1(M)^{\vee} \otimes \omega_1(M) \To  \omega_1(\Pe^{\omega_1,\bullet}_{\T})\ ,$$
and by the universal property we deduce that $\omega_1(\Pe^{\omega_1,\bullet}_{\T}) \cong \Pe^{\omega_1,\omega_1}_{\T}$.

 Now suppose that $\xi = [M, f, v]^{\omega_1,\omega_2}$. By the previous lemma,
we can assume that  $M =  {}_fM_v$.
  Consider the morphism
\begin{equation} \label{inproofetapage8} 
 \eta \mapsto [ M, f,  \eta ]^{\omega_1,\omega_2}  \quad : \quad  \omega_2(M) \To  \Pe^{\omega_1, \omega_2 }_{\T}\ . 
 \end{equation}
 It respects the action of $G^{\omega_2}_{\T}$ on both sides. Therefore 
 by the Tannaka theorem it is the $\omega_2$-image of a  morphism 
$$e_f: M \To \Pe^{\omega_1, \bullet }_{\T} $$
 of ind-objects of $\T$. Its  $\omega_1$-image is 
$$
p  \mapsto  [ M, f,  p ]^{\omega_1,\omega_1} 
\quad : \quad  \omega_1(M)  \overset{\omega_1(e_f)}{\To}   \Pe^{\omega_1, \omega_1 }_{\T} \ .$$
Under the dual map, the counit $\varepsilon: \Pe^{\omega_1, \omega_1}_{\T} \rightarrow k$ maps to $f \in \omega_1(M)^{\vee}$.  
Incidentally, this shows that $e_f$ defines an equivalence $\xi = [\Pe^{\omega_1, \bullet}_{\T}, \varepsilon, \xi]^{\omega_1, \omega_2}$.

Let $N \subset \Pe^{\omega_1, \bullet }_{\T}$ be the sub object such that $\omega_2(N)\subset \Pe^{\omega_1, \omega_2}_{\T} $ is the $G_{\T}^{\omega_2}$-representation generated by $\xi$.   We obtain a commutative diagram
$$
\begin{array}{ccc}
 \omega_2(M)   &  \overset{(\ref{inproofetapage8})}{\To}   &   \Pe^{\omega_1,\omega_2}_{\T} \\
  \cup  &   &  \cup  \\
 \langle G^{\omega_2}_{\T} v \rangle   &   \To  &    \omega_2(N) 
\end{array}
$$
By assumption,  $M= M_v$ and hence the vertical inclusion on the left is an equality.  The map along the bottom is surjective since $\omega_2(N)$ contains $\xi$, the image of $v$ under $(\ref{inproofetapage8})$.
Therefore $\omega_2(N)$ is isomorphic to the image of $\omega_2(M)$ in $\Pe^{\omega_1,\omega_2}_{\T}$,  and the map $e_f$ factors as
$$e_f: M \To  \!\!\!\!\!\!\!\!\! \To N \subset \Pe^{\omega_1, \bullet}_{\T} .$$
 On the other hand, the image of the counit $\varepsilon \in \Pe^{\omega_1, \omega_1}_{\T}$ in $\omega_1(N)$ maps to $f\in \omega_1(M)^{\vee}$. Since $M= {}_fM$, it has no non-trivial such quotients, and  $M\cong N$. This proves that 
 $\omega_2(M(\xi)) = \langle G_{\T}^{\omega_2} \xi \rangle_{B_2}$. The other statement is similar.
 \end{proof}

\subsection{Coradical filtration and decomposition} \label{sectCoradical} 
Let $U$ be a pro-unipotent algebraic group over a field $k$ of characteristic $0$, and let $M$ be a left $U$-module over $k$, i.e.,  $M$ is a  right $\Or(U)$-comodule.  Denote the coaction by  $\Delta: M \rightarrow M \otimes_k \Or(U)$, and let  $\Or(U)_+$ be  the kernel of the augmentation $\varepsilon: \Or(U) \rightarrow k$.

Define a filtration $C_i M$ on $M$ by $C_{-1}M=0$, and 
$$C_i M = \{x \in M  \ :\  \Delta(x) = x\otimes 1 \pmod{C_{i-1}M}\}$$
  Equivalently, 
 $C_iM$ is the fastest increasing filtration on $M$ such that $C_{-1}M =0$ and such that  $U$ acts trivially on $\gr^C M$.  In particular, $C_0 M = M^U$.    This filtration is functorial with respect to morphisms of $U$-modules.
    It exhausts $M$, i.e., 
 $M = \bigcup_{i\geq 0} C_i M$.
 This  follows from Engel's theorem in the case when $M$ is of finite type over $k$, for then the action of $U$ on $M$ factors through a unipotent algebraic matrix group, and $M$ has a non-trivial fixed vector $v\in C_0M$.  Replacing $M$ with $M/C_0M$ and by induction on the dimension of $M$,  we deduce that $M=C_nM$ for some $n$. The general case  follows from the fact that  $M$ is the inductive limit of its sub $\Or(U)$-comodules   of finite type.
 
 Now consider $\Or(U)$, viewed as a right $\Or(U)$-comodule in the natural way, and denote by
 $\Delta : \Or(U) \rightarrow \Or(U) \otimes_k \Or(U)$  the coproduct dual to the  multiplication in $U$. The above construction defines an increasing 
 filtration $C_i \Or(U)$. It satisfies $C_{0} \Or(U) = k$, because $\id = (\varepsilon\otimes \id) \Delta$ in any Hopf algebra and hence $x \in C_0 \Or(U)$ satisfies
 $x = \varepsilon(x)$ and is constant, via  $\Or(U) = \Or(U)_+ \oplus k$. 
 
Now let us denote by $\Delta^r=\Delta - \id \otimes 1: M \rightarrow M\otimes_k \Or(U)$, and 
$$\Delta' = \Delta - \id \otimes 1 - 1 \otimes \id : \Or(U) \rightarrow \Or(U) \otimes_k \Or(U)\ .$$
Note that $x \in C_n M$  if and only if $(\Delta^r)^{n+1}x = 0$. If $M = \Or(U)$ and $n \geq 1$, $x \in C_n\Or(U)_+$ if and only if 
 $(\Delta')^n x=0$. This follows since 
   $$\Delta^r x = \Delta' x + 1 \otimes x \equiv  \Delta' x \pmod{C_{i-1}\Or(U)}$$
for all $x \in C_i\Or(U)$, and $i\geq 1$. In particular, $\gr^C_1 \Or(U) = (C_1 \Or(U))_+$ is the space  of primitive elements in $\Or(U)$, which we shall occasionally denote by  $\hbox{Prim}(\Or(U))$.

 \begin{lem} The coaction satisfies 
$$\Delta C_n(M) \subseteq \sum_{i+j=n} C_i(M) \otimes_k C_j (\Or(U))\ .$$
\end{lem}
\begin{proof}  
 The coassociativity  $(\Delta \otimes \id ) \Delta = (\id \otimes \Delta) \Delta$ implies, by substituting in the definitions of $\Delta'$ and $\Delta^r$, the following 
 identity
 \begin{equation} \label{Deltarcoassoc} 
 (\Delta^r\otimes \id) \Delta^r = (\id \otimes \Delta') \Delta^r\ .
 \end{equation}
 Now let $x\in C_n(M)$, and write  using a variant of Sweedler's notation $$\Delta^r x = \sum_{0 \leq i \leq n}  x_i \otimes \alpha_i$$
   where the  $x_i$ are in $C_i(M)$ and are linearly independent modulo $C_{i-1}(M)$. 
 Since $x\in C_n(M)$, we have $(\Delta^r)^{k+1}x\in C_{n-k-1}M $. By coassociativity  $(\ref{Deltarcoassoc})$
 $$(\Delta^r)^{k+1}x  = (\id \otimes \Delta')^k \Delta^r x =  \sum_{0 \leq i \leq n}  x_i \otimes (\Delta')^k \alpha_i \ . $$
 By the independence assumption on the $x_i$, we must have $(\Delta')^k \alpha_i =0 $ whenever $i\geq n-k$, and hence $\alpha_{n-k} \in C_{k}$.
\end{proof}

It follows from the lemma that 
$$\Delta^r \, C_n M \,  \subseteq \,  C_{n-1} M \otimes_k C_1\Or(U) + C_{n-2}M \otimes_k \Or(U) $$
and  taking the quotient modulo $C_{n-2}M\otimes_k \Or(U)$  defines a map 
$$\delta:  \gr^C_n M \To \gr^C_{n-1} M \otimes_k C_1\Or(U)$$
which is injective by definition of the filtration $C$.  The exact sequence 
$$0 \To C_0 \Or(U) \To C_1 \Or(U) \To \gr^C_1 \Or(U)\To 0$$
splits via the augmentation map $\varepsilon$ which sends $C_1 \Or(U) \rightarrow C_0\Or(U) =k$. Thus 
$C_1 \Or(U) = \gr^C_1 \Or(U) \oplus k$. Since in  any Hopf algebra $(\id \otimes \varepsilon )\Delta^r =0$,
the map $\delta$ lands in the kernel of $\id \otimes \varepsilon$, namely $ \gr^C_{n-1} M \otimes_k \gr^C_1\Or(U)$.
\begin{defn} Iterating $\delta$ we deduce an injective map 
\begin{equation}  \label{firstdefnofPhi} \Phi:  \gr^C M \To M^U  \otimes_k T^c( \gr^C_1 \Or(U)) 
\end{equation}
where $T^c(V) = \bigoplus_{n\geq0 } V^{\otimes n}$ denotes the tensor coalgebra on  $V$, graded by the length of 
tensors.  Recall that $M^U = \gr^C_0 M$. The map $\Phi$ respects the grading on both sides.  We shall call this map the \emph{decomposition into primitives}. \label{gloss: decompprim} 
\end{defn}

In the case $M= \Or(U)$,  $M^U=k$ and $(\ref{firstdefnofPhi})$ yields  an injective  graded map
$$ \Phi: \gr^C_n \Or(U) \To T^c( \gr^C_1 \Or(U))\ .$$
Equip the tensor coalgebra $T^c(V)$  with the deconcatenation coproduct
\begin{eqnarray} 
\Delta^{\dec} : T^c(V)  &\To & T^c(V) \otimes_k T^c(V)   \nonumber \\
   v_1\otimes \ldots \otimes v_n  & \mapsto  &  \sum_{k=0}^n (v_1 \otimes \ldots \otimes v_k) \otimes (v_{k+1} \otimes \ldots \otimes v_n) \ .\nonumber
\end{eqnarray}

\begin{lem} 
The following diagram commutes
$$
\begin{array}{ccc}
 \gr^C M  & \overset{\gr^C \Delta}{\To}   & \gr^C M \otimes_k  \gr^C \Or(U)   \\
  \downarrow &   & \downarrow   \\
   M^U \otimes_k T^c(\gr^C_1 \Or(U))    & \overset{\id \otimes \Delta^{\dec}}{\To}    &   M^U \otimes_k T^c(\gr^C_1 \Or(U))  \otimes_k T^c(\gr^C_1 \Or(U)) 
\end{array}\ .
$$
where the vertical map on the left is $\Phi$, and that on the right is $\Phi \otimes \Phi$. 
\end{lem}

\begin{proof}
The lemma can be proved directly from the recursive definition of $\Phi$ in terms of $\delta$. It is more instructive to  give a  conceptual  proof by interpreting $\Phi$  in the following way.
Consider the sequence
$$0 \To \gr^C_{n-1} M \To C_nM / C_{n-2} M \To \gr^C_{n} M \To 0\ .$$
Since $U$ acts trivially on $\gr^C M$, it follows that the action of $U$ on $C_n M / C_{n-2}M$ factors through   $ U^{ab}$.  If we write  $\uu=\Lie\, U$, then we deduce a map 
$$\uu^{ab} \times  \gr^C_n M \To \gr^C_{n-1} M $$
and hence, denoting $\LL(\uu^{ab})$ the free Lie algebra on $\uu^{ab}$, we have
\begin{equation} \label{LLuaction} 
\LL (\uu^{ab}) \times \gr^C M \To \gr^C M\ .
\end{equation}
The affine ring of $\LL(\uu^{ab})$ is  the tensor coalgebra $T^c(\gr^C_1 \Or(U))$, so the dual of the previous map
is $\gr^C M \rightarrow \gr^C M \otimes_k T^c  (\gr^C_1 \Or(U))$. The map $\Phi$ is obtained by projecting onto  the component 
of $\gr^C M$ of degree zero.  The fact that 
$\LL (\uu^{ab}) \times \gr^C M \rightarrow \gr^C M$ is a left action can be expressed as a commutative diagram. Dualising it and projecting onto $C_0M$
gives precisely the statement of the lemma.
\end{proof}

\subsubsection{Multiplicative structure}
 Now suppose in addition that $M$ is a commutative $k$-algebra, and that the action of $U$ respects the multiplication $\mu$ on $M$. An induction on the indices $i, j$ shows that $\mu ( C_i M \times C_j M) \subseteq C_{i+j}M$.  In particular $C_0M = M^U$ is a subalgebra of $M$. 
Recall that the tensor coalgebra $T^c(V)$ is equipped with a commutative product  $\sha$ called the shuffle product.

 \begin{lem} If $M$ is a commutative $k$-algebra, then $\Phi$ is a homomorphism of graded commutative $k$-algebras, 
 where  $T^c(V)$ is equipped with the shuffle product. 
 \end{lem}
 \begin{proof} Since  $\Delta$ is a homomorphism, $\Delta^r(xy) = \Delta^r(x) (y \otimes 1) + ( x\otimes 1) \Delta^r(y) + \Delta^r(x) \Delta^r(y)$
 for all $x,y\in M$.  It follows that $\delta$ is a derivation 
 \begin{equation}
 \delta(xy) = (x\otimes 1 ) \delta(y) + \delta(x)  (y \otimes 1)
 \end{equation}
 for $x,y \in \gr^C M$,
 and  multiplication by $1$ denotes the identity on $\Or(U)_+$. Now denote by $\partial$ the right deconcatenation map
  $\partial : T^c(V) \rightarrow T^c(V) \otimes V$
  which is given by the formula  $\partial (v_1\otimes \ldots \otimes v_n ) = (v_1\otimes \ldots \otimes v_{n-1}) \otimes v_n$.   It follows from the  definition of the map $\Phi$ as the iteration of $\delta$ that  
  $$\partial \Phi  = (\Phi\otimes \id) \delta \ .$$
  Suppose that $\Phi(ab) = \Phi(a) \sha \Phi(b)$ for  all $a ,b \in \gr^C M$ of total degree  $< n$. If $a$ or $b$ is in $\gr^C_0M$ then the statement is trivial.  Then for 
  $x,y \in \gr^C M$ of total degree $n$ and degree $\geq 1$, we have
  $$\partial \Phi(xy) = (\Phi\otimes \id) \delta(xy) = (\Phi\otimes \id) \big( (x\otimes 1 ) \delta(y) + \delta(x)  (y \otimes 1)\big)\ .$$
  By induction hypothesis applied to $\Phi\otimes \id$,  the right hand side is 
  $$(\Phi\otimes \id)(x\otimes 1) \sha (\Phi\otimes \id) \delta(y) + (\Phi\otimes \id)(y \otimes 1) \sha (\Phi\otimes \id) \delta(x)$$
 which by the previous equation gives 
 $\partial \Phi(xy) = \Phi(x) \sha \partial \Phi(y) + \Phi(y) \sha \partial \Phi(x)$.   This is in fact  one of the many  equivalent definitions of the  shuffle product, and proves,  by the injectivity of $\partial$,  that $\Phi(xy) = \Phi(x) \sha \Phi(y)$.  
 \end{proof} 

Another way to see this lemma is simply to note that 
the action $(\ref{LLuaction})$ respects the multiplication on $\gr^C_M$ and to encode this by a commutative diagram. The dual diagram,
after projecting to $C_0M$ in the appropriate place implies the lemma.

\begin{rem} \label{remunipmonod} One can think of $C_i \Or(U)$ as functions of
`unipotent monodromy' of degree $i$   in the  following way. 
For any   $f \in \Or(U)$, viewed as a function on groups $U(R)$, for  all commutative $k$-algebras $R$,   define  a new function $M_u f$  by  
$(M_u f)  (x)  = f(ux)-f(x)$.  An $ f\in C_i(\Or(U))$  satisfies
$M_{u_1} \ldots M_{u_{n}} f =0$  for all $u_1,\ldots, u_n \in U(R)$
whenever $n\geq i+1$.  In particular, elements of $C_0 \Or(U)$ are constant, and elements of  $C_1 \Or(U)$ are functions $f$ satisfying $f(ab) = f(a) + f(b)$. 
\end{rem}

\subsubsection{Cohomological interpretation} Let $M, U$ be as above.  
Consider the normalised cobar complex,  dual to  (\cite{Weibel}, page 283),
$$ 0 \rightarrow M   \rightarrow M\otimes_k \Or(U) \rightarrow  M \otimes_k \Or(U)_+ \otimes_k \Or(U) \rightarrow   M \otimes_k \Or(U)^{\otimes 2}_+ \otimes_k \Or(U) \rightarrow \ldots $$
where the $(n+1)^{\mathrm{th}}$ arrow is  given by 
$$ d_n =  \sum_{i=0}^n (-1)^i\,  \id^{\otimes i}\otimes \Delta \otimes \id^{\otimes n-i}\ ,$$
(e.g.,  $d_1 = \Delta$, and $d_2 = \Delta \otimes \id - \id \otimes \Delta$), followed by projection onto $\Or(U)_+$ in the appropriate places. 
It is a resolution of $M$ in the category of $\Or(U)$-comodules.  
Derive  the functor 
$\otimes_{\Or(U)} k$, where $k$ is viewed as an $\Or(U)$-module via the augmentation map, by applying it  to the previous complex with the first two terms removed. It defines a complex
$$\mathcal{R}_M: \qquad 0 \To M  \overset{\Delta}{\To} M\otimes_k \Or(U)_+  {\To} M \otimes_k \Or(U)_+^{\otimes 2} {\To} \ldots \ ,$$
with essentially the same differentials $d_n$ as before. 
Since the category of $\Or(U)$-comodules  is equivalent to the category of representations of $U$ of finite type, we deduce that
$$H^n(\mathcal{R}_M) =\mathrm{Ext}^n_{\Or(U)-\hbox{comod}}(k, M)=  \mathrm{Ext}^n_{\mathrm{Rep}(U)}(k, M)=  H^n(U; M) \ .$$
Therefore  $H^0(U;M)  = H^0(\mathcal{R}_M)) =  C_0 M$, since the image of   $\Delta m$ vanishes in $ M \otimes \Or(U)_+$ if and only if  
$\Delta m = m\otimes 1$. In the special case $M= k$,  we have
$$\mathcal{R}_{k} : \qquad 0 \To k  \overset{0}{\To}  \Or(U)_+  \overset{\Delta'}{\To} \Or(U)_+^{\otimes 2}{\To} \ldots $$
and it follows that   $H^0(U;k)=k$, and 
\begin{equation} 
H^1(U; k ) = \gr^C_1(\Or(U)) \ . 
\end{equation}
This is a cohomological interpretation for the terms in the right-hand side of the decomposition map $\Phi$.  The map $\delta$ can in fact be viewed
as a differential in a certain  spectral sequence, as we shall show below.

\begin{lem} Viewing $\Or(U)$ as a left $U$-module (right $\Or(U)$-comodule)
$$H^0(U;\Or(U))=k \qquad \hbox{ and } \qquad H^n(U; \Or(U))= 0  \hbox{  for all  } n \geq 1 \ . $$
\end{lem}
\begin{proof} Take the  cobar resolution  with $M=k$ and reverse all tensors to give 
$$ 0 \rightarrow k   \rightarrow  \Or(U) \rightarrow  \Or(U) \otimes_k \Or(U)_+ \rightarrow  \Or(U) \otimes_k  \Or(U)^{\otimes 2}_+  \rightarrow \ldots $$
It is  a resolution for the same reasons as the cobar resolution. It agrees with $\mathcal{R}_{\Or(U)}$ from the 
second term onwards, and has the same differentials up to a possible overall sign,  so we can read off the cohomology of $\Or(U)$. \end{proof}

\begin{prop} Suppose that $U$ has cohomological dimension $1$. Then 
$$0 \To \gr^C_n M \overset{\delta}{\To} \gr^C_{n-1} M \otimes_k H^1(U;k) \To \gr^C_{n-1} H^1(U;M) \To 0$$
is exact for all $n \geq 1$.
\end{prop} 
\begin{proof}Filter the complex  $\mathcal{R}_M$ by $F^p \mathcal{R}_M = \mathcal{R}_{C_{-p}M}$.  It defines  a spectral sequence with 
$E^0_{p,q} = \gr_F^p \mathcal{R}_M^{p+q}$
and $E^1_{p,q} = H^{p+q}(U; \gr^{C}_{-p} M)$, and converges to $\gr_F^p H^{p+q}(U; M)$.  Since $\gr^C M$ is a trivial $U$-module we have
$$E^1_{p,q} = \gr^C_{-p} M \otimes_k H^{p+q}(U;k)\ .$$
The differential $d^1: E^1_{-p,p} \rightarrow E^1_{1-p,p}$ is the operator $\delta$ defined earlier.  Since $H^j(U;k)$ vanishes for all $j\geq 2$,
 $E^1_{p,q}$ vanishes unless  $p+q\in \{0,1\}$ and the spectral sequence degenerates.  Therefore the following sequence is exact: 
 $$0 \rightarrow   \gr_F^{-n} H^0(U;M) \rightarrow \gr^C_n M \overset{\delta}{\rightarrow} \gr^C_{n-1} M \otimes_k H^1(U;k) \rightarrow \gr_F^{1-n} H^1(U;M) \rightarrow 0$$
The result follows since $H^0(U;M) = C_0 M$ and $\gr^{-n}_F H^0(U;M)=0$ if $n \geq 1$.
\end{proof} 
\begin{cor} \label{cordecompisom} Suppose that $M=  T\otimes_k \Or(U) $, where $T$ is a trivial $U$-module, and $U$ has cohomological dimension $1$.  
Then the  map $\Phi$  is an isomorphism 
$$\Phi: \gr^{C}  M \overset{\sim}{\To} T \otimes_k  T^c(H^1(U;k))\ .$$
\end{cor}
\begin{proof} By the previous two lemmas, $H^1(U;M) = H^1(U; \Or(U))\otimes_k T=0$, and therefore 
$\delta : \gr^C_n M \rightarrow \gr^C_{n-1} M \otimes_k H^1(U;k)$ is surjective as the last term in the exact sequence of the previous lemma vanishes. The iterations of $\delta$ are therefore also surjective,
hence so is $\Phi$. 
\end{proof}

\section{Motivic periods over $\Q$} \label{sectMotperQ}
For  the rest of this section, we only require the results of \S\ref{sectMatrixcoeffs} and \S\ref{sectTannakianCase} in the case $k=B_1= B_2 =\Q$.

\subsection{Category $\HH$ of Betti and de Rham realisations}  \label{sect: PeriodsoverQ}
Based on \cite{DeP1} \S1.10, consider the category $\HH$ whose objects are triples 
$ 
(V_{B}, V_{dR},  c) $
consisting of the following data:
\begin{enumerate}
\item A finite-dimensional $\Q$-vector space $V_B$ with a finite increasing (weight)  filtration $W_{\bullet} V_B$. 
\item A finite-dimensional $\Q$-vector space $V_{dR}$ with a finite increasing  (weight) filtration $W_{\bullet} V_{dR}$ and finite decreasing (Hodge) filtration $F^{\bullet} V_{dR}$. 
\item An isomorphism 
$$c: V_{dR} \otimes_{\Q} \C \overset{\sim}{\To} V_B\otimes_{\Q}\C$$
which respects the filtrations $W_{\bullet}$ on both sides.
\item A linear involution $F_{\infty}: V_B \overset{\sim}{\rightarrow} V_B$ called the real Frobenius.
\end{enumerate}
This data is subject to the conditions:
\begin{itemize}
\item  if  $c_{dR}$ (resp. $c_B$) is the $\C$-antilinear involution on $V_{dR} \otimes \C$ (resp. $V_B \otimes \C$) given by $x\otimes \lambda \mapsto x \otimes \overline{\lambda}$, then  the following diagram commutes:
$$
\begin{array}{ccc}
V_{dR} \otimes \C   &  \overset{c}{\To}   &   V_B \otimes \C  \\
 \downarrow_{c_{dR}}  &   &   \downarrow_{F_{\infty} \otimes c_B}  \\
 V_{dR} \otimes \C   & \overset{c}{\To}    &   V_B \otimes \C  
\end{array}
$$ 
  In particular, $c\, c_{dR} \, c^{-1}$  preserves the lattice $V_B \subset V_B \otimes \C$. 
\item that $V_B$, equipped with the weight filtration $W_{\bullet}$ and  Hodge filtration $c F^{\bullet}$ on $V_{\C}=V_B \otimes_{\Q} \C$, is a $\Q$-mixed Hodge structure. Writing $F$ instead of $cF$,  this is equivalent to $\gr^{W}_{n} V_{\C} =\bigoplus_{p+q=n} F^p \cap \overline{F}^q$.
We assume furthermore that this mixed Hodge structure is graded-polarizable. 
\end{itemize}
The morphisms in the category $\HH$ are given by morphisms of triples respecting the above data.  It is shown in \cite{DeP1}, 1.10,  that $\HH$ is a Tannakian category,
the essential point being that mixed Hodge structures form an abelian category \cite{De2}.   This category could be further enriched
by adding more realisations.

By $(3$), the weight filtration defines a filtration on $(V_B, V_{dR}, c)$ by subobjects:
\begin{equation} \label{eqn: Wnsubobj}
W_n (V_B, V_{dR}, c) = (W_n V_B, W_n V_{dR}, c\big|_{W_n}) \ .
\end{equation} 
Denote the \emph{Hodge numbers} \label{gloss: Hodgenumbers}   of an object $V=  (V_B, V_{dR}, c) $ in $\HH$ by 
\begin{equation}  \label{eqn: Hodgenos}
h_{p,q}(V) = \dim_{\Q} (W_{p+q} \cap F^p) V_{dR} = \dim_{\C} (F^p \cap \overline{F}^q) V_{\C}\ .
\end{equation}
The Tate objects $\Q(n)$, where $n\in \Z$, are the unique triples $(\Q,\Q, c)$ such that $c(1) = (2\pi i)^{-n}$, 
with weight $-2n$ and Hodge filtration on the second vector space $\Q$ defined by $F^{-n} \Q =\Q$, $F^{1-n} \Q =0$.  

\subsection{The ring of $\HH$-periods} The category $\HH$  has two fiber functors
\begin{eqnarray} 
\omega_{\bullet} : \HH   &\To &  \mathrm{Vec}_{\Q} \qquad \qquad \bullet= B \hbox{ or } dR \nonumber \\ 
 (V_{B}, V_{dR}, c) & \mapsto & V_{\bullet} \ ,\nonumber 
\end{eqnarray}
so we can apply  $\S\ref{sectMatrixcoeffs}$ with $\T=\HH$, $k=\Q$.  Let us write\footnote{In \cite{SVMP, BrMTZ} I wrote $\mm = ( \omega_{dR}, \omega_{B})$, and put the coaction on the left, for purely psychological reasons. The corresponding rings of periods are identical.}
$$ \mm = ( \omega_B, \omega_{dR}) \qquad \hbox{ and } \qquad \dR = (\omega_{dR}, \omega_{dR})\ .$$
This defines
 rings $\Pe^{\mm}_{\HH}$ and $\Pe^{\dR}_{\HH}$ as in \S\ref{sectMatrixcoeffs} and \S\ref{sectTannakianCase}, and a canonical element
$$c \in \mathrm{Isom}_{\HH}^{\otimes}(\omega_{dR}, \omega_B)(\C)$$ which is  given by the data $(3).$
\begin{defn} The  ring of $\HH$-periods  is  
$\Pe^{\mm}_{\HH}$. It is equipped with 
\begin{itemize}
\item a period homomorphism 
$$\per: \Pe^{\mm}_{\HH} \To \C$$
which sends $[(V_B, V_{dR}, c), \sigma, \omega]^{\mm}$ to $\sigma c (\omega)$. 
\item an  increasing weight filtration 
$$W_n \Pe^{\mm}_{\HH} = \langle  [(V_{B}, V_{dR}, c),  \sigma ,\omega ]^{\mm}: \hbox{ where } \omega \in W_{n} V_{dR}  \rangle_{\Q} $$
Equivalently, this is the subspace spanned by $[M, \sigma, \omega]^{\mm}$ for objects $M$ in $\HH$  satisfying  $W_n M =M$ where $W_n$ was defined in (\ref{eqn: Wnsubobj}).
\item a right coaction
$
\Delta^{\mm}: \Pe^{\mm} \To \Pe^{\mm}_{\HH}  \otimes_{\Q} \Pe^{\dR}_{\HH}  $
or equivalently, a left action 
$$G_{\HH}^{dR} \times \Pe_{\HH}^{\mm} \To \Pe_{\HH}^{\mm}\ .$$
It respects the weight filtration on $\Pe_{\HH}^{\mm}$ by $(\ref{eqn: Wnsubobj})$. Likewise,
   $\Pe^{\mm}_{\HH}$  admits a right action of the Betti Galois group $G_{\HH}^B$ which also preserves $W$.  
\item the real  Frobenius  involution
$$F_{\infty}:  \Pe^{\mm}_{\HH} \overset{\sim}{\To}    \Pe^{\mm}_{\HH} $$
defined by $F_{\infty} [M,  \sigma ,\omega ]^{\mm} = [M,  \sigma \circ F_{\infty},\omega]^{\mm}$. It has the property that 
$$\per (F_{\infty} \xi) = \overline{ \per(\xi)}\ ,$$
where the bar denotes complex conjugation.
In particular, $F_{\infty}$-invariant motivic periods have  periods in $\R$. The $G_{\HH}^{dR}$-action    commutes with $F_{\infty}$.
\end{itemize}
\end{defn}
Note that, since $\Pe^{\mm}_{\HH}$ admits a $G^{dR}_{\HH}$-action, it is the $\omega_{dR}$ image of  an (ind-)object of the category $\HH$ via theorem $\ref{thm:  Tannaka}$.
Therefore $\Pe^{\mm}_{\HH}$ also carries, in addition to the weight filtration, a decreasing Hodge filtration $F$. The subspace $F^n \Pe^{\mm}_{\HH}$
is spanned by the $[M,\sigma,v]^{\mm}$ with $v\in F^n M_{dR}$.  The Hodge filtration is not preserved by the group $G^{dR}_{\HH}$, but  will nonetheless play a role later on, and is of course preserved by the (right) action of the Betti  Galois group $G^B_{\HH} = \Aut^{\otimes}_{\HH} (\omega_B)$.

\subsection{Some variants}

Denote the subspace $\Pe^{\mm, +}_{\HH} \subset \Pe^{\mm}_{\HH}$ of  \emph{effective $\HH$-periods} \label{gloss: effectiveHp}  to be the subspace of $\HH$-periods of
objects with non-negative Hodge numbers
$$\Pe^{\mm, +}_{\HH} = \langle [M, \sigma , \omega]^{\mm} : M\in Ob(\HH) \hbox{ such that } h_{p,q}(M)=0   \hbox{ unless } p,q\geq 0\rangle_{\Q}$$
 It  forms a ring and is stable under the action of $G^{\dR}_{\HH}\times G^B_{\HH}$.

Similarly, the ring of \emph{mixed Artin-Tate $\HH$-periods}\footnote{the ring of mixed Tate periods corresponds 
to the $\HH$-periods of mixed Tate objects: those whose associated weight-graded  is a direct sum of Tate objects $\Q(n)$}
 \label{gloss: mixedTate}  is  the subspace
$$\Pe^{\mm}_{HT} = \langle [M, \sigma ,\omega]^{\mm} : M\in Ob(\HH) \hbox{ such that } F^p W_{p+q} M_{dR}=0 \hbox{ if } p>q \rangle_{\Q}\ .$$
It is again a ring and is stable under the action of $G^{\dR}_{\HH} \times G^B_{\HH}$. 
The notation is justified as follows: let  $HT\subset \HH$ be the full Tannakian subcategory of $\HH$ consisting of triples $(V_B, V_{dR}, c)$,
whose underlying mixed Hodge structure has Hodge numbers $h_{p,q}=0$ if $p\neq q$. Its ring of motivic periods is $\Pe^{\mm}_{HT}$.
Suppose that $M$ is any object of $\HH$ of this type. Then the weight filtration on its de Rham vector space splits via
$$\gr^W_{2n} M_{dR} = W_{2n} \cap F^n M_{dR}\ .$$
It follows that  $\Pe^{\mm}_{HT}$ is also graded by the weight and we can write
$$\Pe^{\mm}_{HT} = \bigoplus_{n\in \Z}  \gr^W_{2n}  \Pe^{\mm}_{HT}\ .$$
Note that the weight grading is not preserved by $G^{dR}_{\HH}$, only the weight filtration. The weight grading is, however, preserved by 
$G^B_{\HH}$ since the latter preserves the Hodge filtration. 

Combining the two, we have a ring of \emph{effective mixed Artin Tate $\HH$-periods}
$$\Pe^{\mm,+}_{HT} = \Pe^{\mm, +}_{\HH} \cap \Pe^{\mm}_{HT}\ . $$
One can show that it is the largest subalgebra of $\Pe^{\mm}_{HT}$ which is stable under the action of $G^{dR}_{\HH}$ 
and has non-negative weights.

The ring of  $\HH$-\emph{de Rham} \label{gloss: HdeRham}  periods $\Pe_{\HH}^{\dR}$ has similar properties to $\Pe^{\mm}_{\HH}$ (weight filtration, left $G^{dR}_{\HH}$ action), except that it
does not have a real Frobenius involution, and the  period map is replaced by the evaluation map (counit)
 \begin{eqnarray} \mathrm{ev}: \Pe^{\dR}_{\HH} & \To & k \nonumber \\
 {[}M, e,v]^{\dR} &\mapsto&  e(v)\ ,  \nonumber
 \end{eqnarray}
 which is nothing other than evaluation on the element $1 \in G^{\dR}_H(k)$.  There are  analogous effective and Artin-Tate versions.

 \begin{rem}
 The fact that the weight filtration is strict on the category of mixed Hodge structures \cite{De2} implies that the functor 
 $$(V_B, V_{dR}, c) \,   \mapsto \, \gr^W_{\bullet} (V_B, V_{dR}, c) = (\gr^W V_B, \gr^W V_{dR}, \gr^W c)$$
 is exact. Thus, by composing with $\omega_{dR}$ or $\omega_B$ one obtains  new fiber functors we denote by 
 $\omega_{\underline{dR}} = \omega_{dR}\,  \gr^W $ and $\omega_{\underline{B}} = \omega_{B} \,\gr^W$. 
 \end{rem}

 \subsection{\!\!\!\!*\,\, Weight filtrations on $\Pe^{\dR}_{\HH}$}
 The ring $\Pe^{\dR}_{\HH}$ is none other than $\Or(G^{dR}_{\HH})$ viewed as a left $G^{dR}_{\HH}$-module (or right $\Or(G^{dR}_{\HH})$-comodule). It could be written $\Pe^{l, \dR}_{\HH}$ to distinguish from $\Pe^{r, \dR}_{\HH}$ and $\Pe^{c,\dR}_{\HH}$ which are the same vector space, but considered with the right (respectively 
 conjugation) action of $G^{dR}_{\HH}$.
   We shall never  consider  $\Pe^{r,\dR}$, except to remark that  the antipode (inversion in $G^{dR}_{\HH}$)  interchanges $\Pe^{l,\dR}_{\HH}$ and $\Pe^{r,\dR}_{\HH}$.
 By the Tannaka theorem, these rings all define ind-objects in $\HH$ and  in particular are  equipped with weight filtrations. It is important to note that, since the $G^{dR}_{\HH}$-action is different in each case, these structures are distinct.
 
 To avoid ambiguity,  one 
can distinguish the following two  coactions:
\begin{eqnarray}
\Delta^{\mm, l} : \Pe^{\mm}_{\HH} \To \Pe^{\mm}_{\HH} \otimes_{\Q} \Pe^{l, \dR}_{\HH} \nonumber \\
\Delta^{\mm, c} : \Pe^{\mm}_{\HH} \To \Pe^{\mm}_{\HH} \otimes_{\Q} \Pe^{c, \dR}_{\HH} \nonumber 
\end{eqnarray}
They are given by an identical formula, namely $(\ref{eqn: coaction})$, but differ in that
the interpretation of the right-hand side is slightly different.
If $G$ is a group acting on a set $X$,   the former corresponds to the action of $G$ on $G\times X$ via $g (h, x) = (gh, x)$, where $g,h\in G$ and $x\in X$.  This is the usual formula for a left action. It  satisfies
$\Delta^{\mm,l} ( g\xi) = (\id \otimes g)\Delta^{\mm,l} \xi$  for $g\in G^{dR}_{\HH}$ and hence $\Delta^{\mm,l} W_n\Pe^{\mm}_{\HH}  \subset \Pe_{\HH}^{\mm} \otimes_{\Q} W_n   \Pe^{l, \dR}_{\HH}$. 

The second coaction $\Delta^{\mm,c}$ corresponds to the action of $G$ on $G\times X$ given by the formula $g (h,x) = (ghg^{-1}, gx)$. Therefore 
$$ \Delta^{\mm,c} ( g\xi) = (g \otimes c_g) \Delta^{\mm,c} \xi \qquad \hbox{ for } g\in G^{dR}_{\HH}$$
where $c_g$ is conjugation by $g$. In this case we have
$$\Delta^{\mm,c} W_n\Pe^{\mm}_{\HH}  \subset \sum_{i+j=n} W_i \Pe_{\HH}^{\mm} \otimes_{\Q} W_j   \Pe^{c, \dR}_{\HH}\ .$$ 
 Note that the ring $\Pe^{l, \dR}_{\HH}$ has elements in negative weights, but that
 \begin{equation} \label{NoNegWeights}
 W_{-1} \Pe^{c, \dR}_{\HH} = 0 \ .
 \end{equation} 
To see this, observe that the canonical map $(\ref{omegaMuniversalmap})$
 \begin{equation} \label{canmaptoPecdr}
 \omega_{dR}(M)^{\vee} \otimes_{\Q} \omega_{dR}(M) \rightarrow \Pe^{c,\dR}_{\HH}
 \end{equation} 
 is compatible with the action of $G^{dR}_{\HH}$, and hence respects the weight filtration  on both sides. The left-hand side can be identified with $\mathrm{End}(\omega_{dR}(M))^{\vee}$. 
Since elements of $G^{dR}_{\HH}$ preserve the weight filtration on $\omega_{dR}(M)$ the natural transformation $G^{dR}_{\HH} \rightarrow \mathrm{End}(\omega_{dR}(M))$  of functors from commutative $\Q$-algebras to sets, lands in $W_0 \mathrm{End}(\omega_{dR}(M))$, so dually,  the image of $W_{-1} (\omega_{dR}(M)^{\vee} \otimes_{\Q} \omega_{dR}(M))$ under $(\ref{canmaptoPecdr})$ is zero.
Since the direct sum of the maps $(\ref{canmaptoPecdr})$   generates $\Pe^{c,\dR}_{\HH}$, it follows that $W_{-1}  \Pe^{c,\dR}_{\HH}=0$.

\subsection{Some periods which could be called motivic}
The period homomorphism $\per: \Pe^{\mm}_{\HH} \rightarrow \C$ is surjective since any complex number can be obtained as a period of a mixed Hodge structure.  A \emph{motivic period} will be   an element of $\Pe^{\mm}_{\HH}$ which comes from the cohomology of an algebraic variety in a precise way.
The following family of examples is sufficient for our purposes and covers many cases considered in \cite{KoZa}. 

\begin{example}  \label{ex: motivicperiods} Let $X$ be a smooth  scheme over $\Q$ and $D\subset X$ a  normal crossing divisor over $\Q$.  Consider the triple consisting of
\begin{itemize}
\item relative Betti cohomology $H_B^n (X,D) = H^n(X(\C), D(\C);\Q).$
Since complex conjugation on the topological spaces $X(\C), D(\C)$ is  continuous it  defines an involution $F_{\infty}: H_B^n (X,D) \overset{\sim}{\rightarrow} H_B^n (X,D)$.
\item relative algebraic de Rham cohomology  $H_{dR}^n (X,D)$.
It is the hypercohomology of the sheaf of K\"ahler differential forms on a cosimplicial variety constructed out of the irreducible components of $D$.
\item the comparison isomorphism (\cite{Groth})  $$\comp_{B,dR}: H_{dR}^n (X,D)\otimes_{\Q} \C \overset{\sim}{\To} H^n_{B}(X, D) \otimes_{\Q} \C \ .$$
\end{itemize}
It follows from the existence of a natural  mixed Hodge structure \cite{De2, De3} that 
\begin{equation} \label{ExampleHmotive} 
H^n(X,D) :=   ( H_B^n(X,D), H^n_{dR}(X,D), \comp_{B,dR})
\end{equation}
is an object in the category $\HH$. Given a cohomology class $\omega \in H_{dR}^n (X,D)$ and a relative  homology cycle $\sigma \in H_B^n (X,D)$ we can define the \emph{motivic period}  \label{gloss: motperiod}  associated to this
data to be the matrix coefficient
\begin{equation} \label{eqn: motperiod}
[ H^n(X,D) , \sigma , \omega]^{\mm} \in P^{\mm}_{\HH}\ .
\end{equation} 
Its period   $ \sigma(\comp_{B,dR}\, \omega)\in \C$ could be written  $\int_{\sigma} \omega$ and is given by a linear combination of 
integrals.  The Hodge numbers  $h_{p,q}$ of $H^n(X,D)$ are all zero unless $0\leq p, q \leq 2n$, so  $(\ref{eqn: motperiod})$
is effective and lies in $W_{2n} \Pe^{\mm,+}_{\HH}$.
\end{example}

\begin{defn}
The space of \emph{effective motivic periods} $\Pe^{\mm,+}$ is the subspace of $\Pe^{\mm,+}_{\HH}$  spanned by the elements
$(\ref{eqn: motperiod})$. 
\end{defn}

We shall not need to  define a ring of motivic periods which are not effective (we shall never write $\Pe^{\mm}$ except in this sentence).

This working definition of effective motivic periods amply suffices for many  purposes (e.g.,  the constant cosmic Galois group \cite{Cosmic}).
 By the K\"unneth formula,   $\Pe^{\mm,+}$  is closed under multiplication   and it is immediate from  
$(\ref{eqn: coaction})$ that it is closed under the action of $G_{\HH}^{dR}$ and $G^B_{\HH}$.  
Likewise, the ring of effective de Rham periods $\Pe^{\dR,+}\subset \Pe^{\dR, +}_{\HH}$ is the subspace spanned by 
$[ H^n(X,D) , v, \omega]^{\dR}$, where $v\in H^n_{dR}(X,D)^{\vee}$.

Now define  $G^{dR}$ to be  the quotient of $G_{\HH}^{dR}$ by the subgroup which acts trivially on the ring  $\Pe^{\mm,+}$ of effective motivic periods.
The affine group scheme $G^{dR}$ acts faithfully on  $\Pe^{\mm,+}$, and is an approximation to a  motivic Galois group. 
Its category of representations $\mathrm{Rep}\,  G^{dR}$ is a crude version of  a Tannakian category of mixed motives. 
A key point
is that the groups $G^{dR}$ and $G_{\HH}^{dR}$ act in an identical manner on  $\Pe^{\mm,+}$.

A folklore version of Grothendieck's period conjecture states that
\begin{conj} \label{conj1} The period homomorphism $\per : \Pe^{\mm,+} \rightarrow \C$ is injective.
\end{conj} 
 
Note that this conjecture is weaker than the analoguous conjecture for the motivic periods of the category of Nori motives, for example.

\subsection{Some terminology} We list a sample of possible quantities to describe $\HH$ periods.
For any object $M$ in $\HH$ one can make the following definitions. 
\begin{itemize}
\item  Let  $M_B^{+}, M_B^-$ denote the $\pm$ eigenspaces for $F_{\infty}$, and
set $\mathrm{rank}^{\pm} M = \dim_{\Q} M^{\pm}_B$. The comparison $M_{dR} \otimes_{\Q} \C \overset{\sim}{\rightarrow} M_{B} \otimes_{\Q} \C$ implies that 
$$\mathrm{rank}(M)= \mathrm{rank}^+(M) + \mathrm{rank}^-(M) = \dim_{\Q} M_{dR}\ .$$

\item Define the  \emph{de Rham Galois group} \label{gloss: GaloisGroup}  $G^{dR}(M)$ of $M$ to be the largest quotient of $G^{dR}_{\HH}$ which acts faithfully on $M_{dR}$. Equivalently, it is the de Rham Tannaka group of the full
 Tannakian subcategory of $\HH$ generated by $M$.  Define the \emph{Betti Galois group} in the same way on replacing $dR$  by $B$. 
 The comparison map gives a canonical  isomorphism
 $$G^{B}(M) \times \C \overset{\sim}{\To} G^{dR}(M) \times \C\ .$$
   Define the \emph{transcendence dimension}  \label{gloss: transdim}  of $M$ to be 
 $$\dim_{\tr} (M) = \dim G^{dR}(M) = \dim G^{B}(M) \ . $$
  Define the \emph{component group}  \label{gloss: component group} 
  $\pi^{\bullet}_0(M) $ to be $\pi_0 (G^{\bullet}(M))$, where $\bullet = B,dR$ and $\pi_0$    is the \'etale group scheme whose  affine ring $\Or(\pi^{\bullet}_0(M))$ is the largest separable subalgebra of $\Or(G^{\bullet}(M))$ (see \cite{Waterhouse}, \S6.5-7). Since $\pi_0$ commutes with change of base field, the  comparison isomorphism gives   $$\pi^{B}_0(M) \times \C \overset{\sim}{\rightarrow} \pi^{dR}_0(M)\times \C\ .$$

 \item Define the \emph{Hodge polynomial}  \label{gloss: Hodgepoly}  of $M$ to be   (see  $(\ref{eqn: Hodgenos})$)
 $$
 h(M) (r,s) = \sum_{r,s}    h_{p,q}(M) \, r^p s^q   \quad \in \quad  \Z[r^{\pm},s^{\pm}]
$$

 \item Define the class of a \emph{period matrix} \label{gloss: periodmatrix}  of $M$ as follows. 
 Let  $M = (M_{B}, M_{dR}, c_{M})$, and $r= \mathrm{rank}(M)$.   Choose  isomorphisms
 $M_B \cong \Q^r$ and $M_{dR} \cong \Q^r$,   which are adapted to the weight (resp. Hodge and weight) filtrations, and write $c_M$ in this basis. This gives a well-defined 
 element
 $$    [c_{M}] \quad \in \quad  W_0 \mathrm{GL}(M_B)  \backslash W_0 \mathrm{GL}( \C^r) / F^0 W_0 \mathrm{GL}(M_{dR})$$
where  $W_0 \mathrm{GL}(M_B)$ denotes  the subgroup of $\mathrm{GL}(M_B)(\Q)$  which preserves $W$, 
 and so on.   It is an equivalence class of square $d\times d$ matrices of complex numbers, where $d$ is the rank  of $M$.  
  The matrix $c_M$ in fact lies in the subspace satisfying $F_{\infty} c_M = \overline{c}_M$. Thus if we furthermore choose our Betti basis $M_B \cong \Q^r$
 to be compatible with the decomposition $M_B = M_B^+ \oplus M^-_B$,  then  the rows corresponding to $M_B^+$ have entries in $\R$, and those corresponding to $M_B^-$ have
 entries in $i\R$. 
 
 The  \emph{determinant}  \label{gloss: determinant}  $\det(M)$ is defined to be  $\det(c_{M})$  in $\C / \Q^{\times}$.

\end{itemize}

\begin{defn}  Now  let  $\xi \in \Pe^{\mm}_{\HH}$ and denote its  minimal object (\S\ref{sectminimalobject})  by   $M(\xi)$. This enables us to attach the following 
invariants to $\xi$:

 \begin{enumerate}
 \item  Define the \emph{space of de Rham Galois conjugates} \label{gloss: conjugates}   of $\xi$ to be the right $\Or(G_{\HH}^{dR})$-subcomodule of $\Pe^{\mm}_{\HH}$ generated by $\xi$. It is isomorphic to  $M(\xi)_{dR}$ by corollary \ref{corGomegarep}. 
 
  Likewise, define the 
 \emph{space of Betti Galois conjugates} to be the left $\Or(G_{\HH}^{B})$-subcomodule of $\Pe^{\mm}_{\HH}$ generated by $\xi$.  It is isomorphic to  $M(\xi)_{B}$ by corollary \ref{corGomegarep}.
 
 Define the \emph{space of biconjugates of} $\xi$ to be the left $\Or(G_{\HH}^{B})$ and right $\Or(G_{\HH}^{dR})$-subcomodule of $\Pe^{\mm}_{\HH}$ generated by $\xi$. It is the vector space spanned by the set of all  matrix coefficients $[M(\xi), M(\xi)_{B}^{\vee}, M(\xi)_{dR}]^{\mm}$. 
 
 Define the \emph{algebra of  de Rham/Betti/bi-conjugates of $\xi$}  to be the subalgebras of $\Pe^{\mm}_{\HH}$ generated by the above vector spaces.\footnote{One might be tempted, by  analogy with algebraic numbers,  to define notions of Betti/de Rham/biconjugates of $\xi$ by considering the orbits of $\xi$ under the group of $R$-points of the corresponding groups $G^{B / dR}_{\HH}(R)$, for $R$ a commutative $\Q$-algebra. We  shall not.}

 \item Define  the 
 $\mathrm{rank}( \xi): = \mathrm{rank} \,M(\xi)$ (similarly  $\mathrm{rank}^{\pm} (\xi) = \mathrm{rank}^{\pm} \,M(\xi)$).
 \label{gloss: rank}

 \item Define the de Rham (resp. Betti) \emph{Galois group} $G^{\bullet}_{\xi}= G^{\bullet}(M(\xi))$, where $\bullet \in \{B, dR\}$.  
   Define the \emph{transcendence dimension} \label{gloss: transdim2} of $\xi$ to be 
 $$\dim_{\tr} (\xi) = \dim G^{dR}_{\xi} = \dim G^{B}_{\xi}  \ . $$
  Define the \emph{component group}  \label{gloss: component group2}   to be 
  $\pi^{\bullet}_0(\xi):=\pi_0 (G^{\bullet}_{\xi})$, and   the \emph{degree} (for want of a better word) of $\xi$ to be $\deg(\xi) = \big| \pi^{\bullet}_0(\xi)(\C) \big| \in \N$.

 \item Define the \emph{Hodge polynomial} \label{gloss: Hodgepoly2}  of $\xi$ to be 
 $
 h(\xi) (r,s) =  h(M(\xi)) (r,s)$.
 
 \item Define the  class of a \emph{period matrix}  \label{gloss: periodmatrix2}  of $\xi$ to be   $[c_{\xi}] = [c_{M(\xi)}]$, and define the \emph{determinant} 
 \label{gloss: det2}  to be $\det(\xi)=\det(c_{\xi})$.
 \end{enumerate}
\end{defn}

Many of the above definitions go through for an $\HH$-de Rham period $\xi \in \Pe^{\dR}_{\HH}$ with the obvious changes, which we leave to the reader.

The Hodge polynomial $h(\xi)$ is symmetric in $r, s$ and satisfies $\mathrm{rank}\, \xi = h(\xi)(1,1)$. 
 The element $\xi$ is \emph{effective} if and only if $h(\xi)\in \Z[r,s]$, and  is  \emph{mixed Artin-Tate} (lies in $\Pe^{\mm}_{HT}$) if and only if  $\xi \in \Z[(rs)^{\pm 1}]$.

It follows from the formulae for sums and  products of matrix coefficients that  if $\xi_1, \xi_2 \in \Pe^{\mm}_{\HH}$ then 
$M_{\xi_1+ \xi_2}$ is a subquotient of $M_{\xi_1} \oplus M_{\xi_2}$ and $M_{\xi_1 \xi_2}$ is a subquotient of $M_{\xi_1} \otimes M_{\xi_2}$. Therefore the 
 rank satisfies 
$$\mathrm{rank} (\xi_1 + \xi_2 ) \leq \mathrm{rank}( \xi_1) +  \mathrm{rank}( \xi_2)  \quad \hbox{ and } \quad  \mathrm{rank} (\xi_1 \xi_2 ) \leq \mathrm{rank}( \xi_1) \,\mathrm{rank}( \xi_2)\ .$$
More generally, the Hodge polynomial satisfies
$$h(\xi_1+ \xi_2)  \preccurlyeq h(\xi_1) + h(\xi_2) \quad \hbox{ and  } \quad h(\xi_1\xi_2)  \preccurlyeq h(\xi_1)  h(\xi_2)$$ where $ \preccurlyeq$ means that the inequality $\leq$ holds coefficient by coefficient.
\begin{rem} The transcendence dimension of $\xi$ is   $\dim G^{dR}_{\xi} = \deg_{\tr} \Or(G^{dR}_{\xi})$. The latter is the ring generated  by the  `de Rham biconjugates' $[M(\xi), v, w]^{\dR}$ for all  $v \in M(\xi)^{\vee}_{dR}$ and $w \in M(\xi)_{dR}$.  Applying the comparison map,  this is isomorphic to the ring generated by the biconjugates $[M(\xi), \sigma, w]^{\mm}$ where $\sigma \in M(\xi)^{\vee}_{B}$ and $w \in M(\xi)_{dR}$, tensored with $\C$.  It follows from this argument that
\begin{equation}
 \dim_{\tr}(\xi) = \deg_{\tr} \langle \hbox{Ring of biconjugates of } \xi \rangle 
 \end{equation}
Therefore the  period conjecture \ref{conj1} implies the following (compare  \cite{An1}):

\begin{conj} Let $\xi \in \Pe^{\mm, +}$ be a motivic period.  Let $P_{\xi} \subset \C$ be the $\Q$-algebra  generated by the images of the Galois biconjugates of $\xi$ under the period homomorphism. Then the transcendence degree of $P_{\xi}$ satisfies
$\deg_{\tr} P_{\xi} = \dim_{\tr} (\xi)$.
\end{conj}

\end{rem}

\subsection{Semi-simple and unipotent periods} \label{sectSemiSimpleUnipotent}
Let $\HH^{ss}$ denote the full Tannakian subcategory of $\HH$ generated by semi-simple objects. Define the \emph{ring of semi-simple (or pure) periods}
\label{gloss: semisimple}   to be 
$ \Pe^{\mm}_{\HH^{ss}}$, and respectively $ \Pe^{\dR}_{\HH^{ss}}$ its de Rham version.

Every object of $\HH^{ss}$ is graded by the weight filtration.  It follows that $\Pe^{\mm}_{\HH^{ss}} $ is also graded by the weight filtration.  
 The action of the group $G^{dR}_{\HH}$ on $\Pe^{\mm}_{\HH^{ss}}$ factors through
a quotient we denote by $S^{dR}_{\HH}$. It is a projective limit of reductive affine algebraic groups over $\Q$, and there is an exact sequence
\begin{equation} \label{UHexactsequence} 
1 \To U^{dR}_{\HH} \To G^{dR}_{\HH} \To S^{dR}_{\HH} \To 1 
\end{equation}
where $U_{\HH}^{dR}$ is pro-unipotent. Define the  ring of \emph{unipotent de Rham periods}  \label{gloss: unipotentperiods}  to be
$$\Pe^{\uu}_{\HH}= \Or(U^{dR}_{\HH})\ ,$$
equipped with the conjugation action of $G^{dR}_{\HH}$. The left action of $U^{dR}_{\HH}$ on $\Pe^{\mm}_{\HH}$ is equivalent to a right coaction, which we call the \emph{unipotent de Rham} coaction
\begin{equation} \label{Deltau} \Delta\!^u: \Pe^{\mm}_{\HH} \To \Pe^{\mm}_{\HH} \otimes_{\Q} \Pe^{\uu}_{\HH}\ .
\end{equation}
It is given by the same formula as   $(\ref{eqn: coaction})$, where elements on the right hand side of the tensor product are viewed as functions on $U^{dR}_{\HH}$.  This coaction  is equivariant with respect to the action of $G^{dR}_{\HH}$, i.e.,   $\Delta\!^u (g \xi) = (g\otimes c_g) \Delta\!^u \xi$, where $c_g$ denotes conjugation by $g\in G^{dR}_{\HH}$.   In particular, $\Delta\!^u W_n \subset \sum_{i+j = n} W_i \otimes W_j$.

The ring of unipotent periods $\Pe^{\uu}_{\HH}$ is equipped with an \emph{antipode} \label{gloss: U-antipode}  
$$S: \Pe^{\uu}_{\HH} \rightarrow \Pe^{\uu}_{\HH}\ ,$$
which is dual to inversion in the group $U^{dR}_{\HH}$.
Using notation introduced earlier, the restriction gives a natural surjective map $\Pe^{c, \dR}_{\HH} \rightarrow \Pe^{\uu}_{\HH}$ which is $G^{dR}_{\HH}$ equivariant, and it follows from a previous calculation
$(\ref{NoNegWeights})$   that  $\Pe^{\uu}_{\HH}$ has non-negative weights:
$$W_{-1} \Pe^{\uu}_{\HH} = 0\ . $$

\begin{rem} \label{remsplit} If one replaces the de Rham functor with the graded de Rham functor $\omega_{\underline{dR}}$, then the analogous sequence to $(\ref{UHexactsequence})$ is canonically split, since $\gr^W$ is a fiber functor from $\HH$ to $\HH^{ss}$ which is the identity on $\HH^{ss}$. Thus  $G^{\underline{dR}}_{\HH} = U^{\underline{dR}}_{\HH} \rtimes S^{dR}_{\HH}$.  
\end{rem}

\begin{prop} \label{lemPemtensor} There is a non-canonical isomorphism of algebras 
$$ \Pe^{\mm}_{\HH}\otimes_{\Q} \overline{\Q} \cong   \Pe^{\mm}_{\HH^{ss}}   \otimes_{\Q}  \Pe^{\uu}_{\HH} \otimes_{\Q}  \overline{\Q} \ .$$
It does not respect the coalgebra structure.  
\end{prop}
\begin{proof}
Choose   points in $ \mathrm{Isom}_{\HH}(\omega_{B}, \omega_{dR})(\overline{\Q})$ and
$ \mathrm{Isom}_{\HH} (\omega_{dR}, \omega_{\underline{dR}})(\overline{\Q}). $ They induce  isomorphisms over $\overline{\Q}$ 
$$\mathrm{Isom}_{\HH} (\omega_{dR}, \omega_B) \times \overline{\Q}\overset{\sim}{\To} \mathrm{Isom}_{\HH} (\omega_{dR}, \omega_{dR})\times \overline{\Q} \cong  \mathrm{Isom}_{ \HH} (\omega_{\underline{dR}}, \omega_{\underline{dR}})\times \overline{\Q}\ .$$
The group in the middle is $G_{\HH}^{dR}\times \overline{\Q}$, and the one on the right is 
 $G_{\HH}^{\underline{dR}}\times \overline{\Q}$ which splits canonically since  $G_{\HH}^{\underline{dR}}\cong U_{\HH}^{\underline{dR}} \rtimes S^{dR}_{\HH}$ by remark  \ref{remsplit}. On the level of affine rings we deduce isomorphisms  of algebras   $\Or(G_{\HH}^{dR}) \otimes_{\Q} \overline{\Q}\cong \Or( G_{\HH}^{\underline{dR}})\otimes_{\Q} \overline{\Q} =   \Or(S_{\HH}^{dR})\otimes_{\Q}   \Or( U_{\HH}^{\underline{dR}}) \otimes_{\Q} \overline{\Q}$. 
 This gives a non-canonical isomorphism of algebras 
 $$ \Pe^{\mm}_{\HH} \otimes_{\Q}  \overline{\Q} \cong    \Or(S_{\HH}^{dR})   \otimes_{\Q}  \Or(U_{\HH}^{dR}) \otimes_{\Q} \overline{\Q}  \ ,$$
 which, on taking $U_{\HH}^{dR}$-invariants, induces  $\Pe_{\HH^{ss}}^{\mm} \otimes_{\Q} \overline{\Q} \cong \Or(S_{\HH}^{dR}) \otimes_{\Q} \overline{\Q}$. The statement follows from the identification 
 $\Pe^{\uu}_{\HH} =  \Or(U_{\HH}^{dR}) $.
\end{proof} 

\subsection{Filtration by unipotency and decomposition}
The existence of the weight filtration implies that we can  apply the  constructions of \S\ref{sectCoradical} to
$$U = U^{dR}_{\HH} \qquad \hbox{ and } \qquad M = \Pe^{\mm}_{\HH}\ ,$$
where $\Pe^{\mm}_{\HH}$ is equipped with the comodule structure $\Delta^u:M \rightarrow M \otimes_{\Q} \Or(U)$.
\begin{defn} We shall say that an element  $\xi$   in $\Pe^{\mm}_{\HH}$  is  of \emph{unipotency degree} \label{gloss: unipdegree}  or \emph{coradical degree} $\leq i$ if it lies in $C_i\Pe^{\mm}_{\HH}$.
\end{defn} An element $\xi \in \Pe^{\mm}_{\HH}$ is of unipotency degree zero if and only if $\Delta(\xi)= \xi \otimes 1$, so it is $U^{dR}_{\HH}$-invariant and hence  semi-simple:
$$C_0 \Pe^{\mm}_{\HH} =\Pe^{\mm}_{\HH^{ss}}\ .$$
An element $\xi$ of unipotency degree at most one corresponds
to a period of a simple extension. This is discussed in further detail  in \S\ref{sectClassif}.
 
Recall that $\Delta^{u,r} = \Delta^u - \id \otimes 1.$
Then $\xi \in C_i \Pe^{\mm}_H$ if and only if 
$$ (\Delta^{u,r})^{i+1} \xi  = 0\ .$$

As in \S\ref{sectCoradical}, we deduce the existence of a derivation 
$$\delta : C_n \Pe^{\mm}_{\HH} \To C_{n-1} \Pe^{\mm}_{\HH} \otimes_{\Q} H^1(U) \ . $$

\begin{defn}   The \emph{decomposition into primitives map} \label{gloss: decompprim2}  $$\Phi: \gr^C_{\bullet} \Pe^{\mm}_{\HH} \To  \Pe^{\mm}_{\HH^{ss}} \otimes_{\Q} T^c(H^1(U)   )$$
is defined by iterating $\delta$. 
\end{defn}

The map $\Phi$ is a homomorphism of $S^{dR}_{\HH}$-modules. To see this, recall  that the coaction 
$\Delta^u$  is equivariant with respect to the  action of $G^{dR}_{\HH}$ on the left on  $\Pe^{\mm}_{\HH}$, and by conjugation 
on $\Or(U)$. Therefore so is $\Delta^{u,r}$, and likewise  $\delta$.  On the other hand, $U$ acts trivially on both $\gr^C \Pe^{\mm}_{\HH} $  and $H^1(U)$, 
so the action of $G^{dR}_{\HH}$ factors through its quotient $S^{dR}_{\HH}$. Hence $\delta$ is $S^{dR}_{\HH}$-equivariant, and by iteration, so  is $\Phi$.

The map $\Phi$, together with the invariants defined above, give the first steps towards a classification of motivic periods by group theory 
and provides a   tool for proving linear or algebraic independence of motivic periods (e.g., \cite{BrMTZ}).
 This is discussed  in \S\ref{sectClassif}, where we shall also show that $\Phi$ is in fact an isomorphism.
     Note that $\Phi$ is not to be confused with the notion of symbol (\S \ref{sectSymbols}).

\section{\!\!\!*\,\, Further remarks on motivic periods} \label{sectFurther}
The paragraphs below are independent from each other and can be skipped. 

\subsection{Universal period matrix and `single-valued' periods} \label{sectSVconstant}
Let $M$ be an object of $\HH$. Then there is a canonical morphism
\begin{equation} \label{univcoact}  M_{dR} \To M_{B} \otimes_{\Q} \Pe_{\HH}^{\mm} 
\end{equation} 
 obtained by composing the natural map $\delta^{\vee} \otimes \id:  \Q \otimes_{\Q} M_{dR} \rightarrow M_B\otimes_{\Q} M_B^{\vee} \otimes_{\Q} M_{dR}$ with the map $(\ref{omegaMuniversalmap})$  $M_B^{\vee} \otimes_{\Q} M_{dR} \rightarrow \Pe^{\mm}_{\HH}$. It is given by the formula
 $$ v \mapsto \sum_i e_i \otimes [M, e_i^{\vee}, v]^{\mm}$$
 where $e_i$ (resp. $e_i^{\vee}$) is a basis (resp. dual basis) of $M_B$.
  Extending scalars from $\Q$ to $\Pe^{\mm}_{\HH}$,  it defines an isomorphism
 $$c^{\mm}_M: \,  M_{dR} \otimes_{\Q} \Pe^{\mm}_{\HH} \overset{\sim}{\To} M_B \otimes_{\Q} \Pe^{\mm}_{\HH}$$
 which is functorial in $M$ and which we think of as a \emph{universal comparison map}.  \label{gloss: universalcomparison}
 It is equal to   the isomorphism of fiber functors $\iota_M$  (the notation is defined in $(\ref{phiM})$; set $B_1=B_2=k=\Q$ and $R= \Pe^{\mm}_{\HH}$), where 
 $$\iota \in \Isom_{\HH}^{\otimes}(\omega_{dR}, \omega_B) ( \Pe^{\mm}_{\HH})$$ is the  element corresponding to the identity on $\Pe^{\mm}_{\HH}$. Since $c \in \Isom_{\HH}^{\otimes}(\omega_{dR}, \omega_B)(\C)$ is, by definition of the period homomorphism, equal to $\per(\iota)$, it follows that 
the  comparison map   $c_M: M_{dR}  \rightarrow M_B \otimes_{\Q} \C$ is obtained from $(\ref{univcoact})$ by applying the period homomorphism; i.e.,  $c_M = (\id \otimes \per) c^{\mm}_M$.

As a first application,  the universal coaction  $(\ref{univcoact})$ defines a lift of the period matrix of $M$ to the ring of $\HH$-periods:
 $$    [c^{\mm}_M] \quad \in \quad  W_0 \mathrm{GL}(M_B)  \backslash W_0 \mathrm{GL}(\Pe^{\mm}_{\HH}, r ) / F^0 W_0 \mathrm{GL}(M_{dR})\ $$
 where $r = \mathrm{rank}\, M$.
We have $[c_M] = \per [c^{\mm}_M]$ since the period homomorphism is $\Q$-linear. Applying this to the minimal object $M= M(\xi)$ defines an invariant $[c^{\mm}_{\xi}]$ of any element $\xi \in \Pe^{\mm}_{\HH}$.   Its determinant  is an  element  $\det(c^{\mm}_{\xi}) \in \Pe^{\mm}_{\HH}/\Q^{\times}$.

Another application is to construct single-valued versions of $\HH$-periods, inspired by \cite{Be-De}. We only need the fact that the real Frobenius $F_{\infty}$ defines a $\Q$-linear involution on $\Pe^{\mm}_{\HH}$ which commutes with the action of $G^{dR}_{\HH}$. Let
$$f \in     \Isom_{\HH}^{\otimes}(\omega_{dR}, \omega_B) ( \Pe^{\mm}_{\HH})$$
correspond to $F_{\infty} : \Pe^{\mm}_{\HH} \overset{\sim}{\rightarrow} \Pe^{\mm}_{\HH}$. It satisfies $f_M = (F_{\infty} \otimes \id) \iota_M = (\id \otimes F_{\infty}) \iota_M$.  Since $\Isom_{\HH}^{\otimes}({\omega_{dR}, \omega_B})$ is a right $G^{dR}_{\HH}$-torsor, there is a unique  element 
\begin{equation} \s \in G^{dR}_{\HH}(\Pe^{\mm}_{\HH}) \quad \hbox{such that} \quad  f \, \s = \iota  \ 
\end{equation}
which is computed explicitly below. This gives rise to a homomorphism  (\emph{single-valued map}) \label{gloss: singlevalued1}  
\begin{equation} \label{single-valuedmapdefn} 
\s^{\mm}: \Pe^{\dR}_{\HH} \To \Pe^{\mm}_{\HH}
\end{equation}
which is $G^{dR}_{\HH}$-equivariant if one equips the left-hand side $\Pe^{\dR}_{\HH}=\Or(G^{dR}_{\HH})$ with the action of $G^{dR}_{\HH}$ by conjugation.  

\begin{rem} In \cite{SVMP} a slightly different single-valued map  $\mathrm{sv}^{\mm}$ was defined 
on the ring of mixed Tate periods, which is graded by weight.  It is defined by a similar formula on replacing $F_\infty$ by $F_{\infty}$ twisted by $(-1)^n$ in weight $2n$.
\end{rem} 

The situation is summarised by the following  commutative diagram 
$$
\begin{array}{ccc}
M_{dR} \otimes_{\Q} \Pe^{\mm}_{\HH}   & \overset{\iota_M}{\To}    & M_B \otimes_{\Q} \Pe^{\mm}_{\HH}    \\
 \downarrow_{\s^{\mm}_M}   &   & \uparrow_{ \id \otimes F_{\infty}}    \\
M_{dR} \otimes_{\Q} \Pe^{\mm}_{\HH}   & \overset{
\iota_M}{\To}   &   M_B \otimes_{\Q} \Pe^{\mm}_{\HH}
\end{array}
$$
where all maps are isomorphisms, which  is functorial in $M$.

The \emph{single-valued $\HH$-period matrix} \label{gloss: singlevaluedperiodmatrix1}  $s^{\mm}_M:M_{dR} \rightarrow M_{dR} \otimes_{\Q} \Pe^{\mm}_{\HH}$ can be computed as follows. Since $F_{\infty}^{-1} = F_{\infty}$, it follows from the  above that $s^{\mm}_M$ is given by the composite $\iota_M^{-1} \circ  (\id \otimes F_{\infty})^{-1} \circ \iota_M= \iota_M^{-1} \circ  (\id \otimes F_{\infty}) \circ \iota_M$ which is explicitly
$$M_{dR} \overset{c^{\mm}_M}{\To} M_B \otimes_{\Q} \Pe^{\mm}_{\HH} \,\,{\overset{\id \otimes F_{\infty}}{\To}} \,\, M_B \otimes_{\Q} \Pe^{\mm}_{\HH} \overset{(c^{\mm}_M)^{-1}}{\To} M_{dR} \otimes_{\Q} \Pe^{\mm}_{\HH}\ .$$
Thus if $C_M$ is a matrix representing the map $c^{\mm}_M$ with respect to some choice of bases for $M_{dR}, M_B$, then $s^{\mm}_M$ is represented by $(F_{\infty} C_M)^{-1} C_M$. 
This is indeed invariant under change of basis for $M_B$, which amounts to replacing $C_M$ with $P C_M$ for some  $P \in \mathrm{GL}(M_B;\Q)$; the quantity $(F_{\infty} C_M)^{-1} C_M$ is unchanged since $F_{\infty}$ acts trivially on the coefficients of $P$ because they are rational.

Finally, the \emph{single-valued period matrix}  $ s_M $ is obtained by applying the period map to $s^{\mm}_M$, and is given directly from the usual comparison map by  $s_M = \overline{C}_M^{-1} C_M$.

\subsection{Motivic philosophy} It is  hoped that there  exists a neutral Tannakian category $\MM_{\Q}$ of mixed motives over $\Q$ equipped, in particular, with  Betti and de Rham realisations, and hence a  functor $\MM_{\Q} \rightarrow \HH$  and thus  a homomorphism \begin{equation} \label{eqn:  MMtoH}
\Pe^{\mm}_{\MM_{\Q}} \To \Pe^{\mm}_{\HH} \ .
\end{equation}
The elements $(\ref{ExampleHmotive})$ should certainly be in its image, and  the following diagram \begin{eqnarray}
\Pe^{\mm}_{\MM_{\Q}}  &\To&   \Pe^{\mm}_{\MM_{\Q}}     \otimes_{\Q} \Pe^{\dR}_{\MM_{\Q}} \nonumber\\
\downarrow &  &  \qquad \downarrow  \qquad  \nonumber \\ 
\Pe^{\mm}_{\HH} &\To&   \Pe^{\mm}_{\HH}  \otimes_{\Q}  \Pe^{\dR}_{\HH}     \nonumber \ , 
\end{eqnarray} 
where the horizontal maps are given by the coactions $(\ref{eqn: coaction})$, 
would commute. 
Therefore the action of $G^{dR}_{\HH}$ on the de Rham realisation $M_{dR}$ of an object $M\in \MM_{\Q}$ would be motivic, i.e.,   would factor  through
$G^{dR}_{\HH} \rightarrow G^{dR}_{\MM_{\Q}}$. 
 Grothendieck's period conjecture  states that  the period 
$\per :  \Pe^{\mm}_{\MM_{\Q}} \rightarrow \C$ is injective, which motivates  many classical conjectures in transcendence theory  \cite{An1,An2, Be}.  Since $\per$ factors through $(\ref{eqn:  MMtoH})$ this would imply conjecture \ref{conj1} and 
 a much weaker conjecture: namely that the homomorphism $(\ref{eqn:  MMtoH})$ is injective. In this case, relations between elements of $\Pe^{\mm}_{\MM_{\Q}}$
 could be detected in $\Pe^{\mm}_{\HH}$, and $G^{dR}_{\HH} \rightarrow G^{dR}_{\MM_{\Q}} $ would have the same image in $\Aut(M_{dR})$ for every object  $M \in \MM_{\Q}$ (and likewise for  Betti). 
 It is for this reason that we allow ourselves to call the  periods  $(\ref{eqn: motperiod})$   `motivic'.

One important situation in which much of the above  certainly works is the case $\MT(\Q)$ of mixed Tate (or Artin-Tate) motives over $\Q$\footnote{these exist over any number field, but for the time being we are working only over $\Q$} \cite{LevineTate, DeGo}.  One then has  a morphism
$\Pe^{\mm}_{\MT(\Q)} \rightarrow \Pe^{\mm}_{\HH}$
which is known to be injective (by the full faithfulness of the Hodge realisation).
Borel's deep results on the rational algebraic $K$-theory of $\Q$ (see example $\ref{exampleMTcase}$ below)
give  a precise upper bound for the size of the ring $\Pe^{\mm}_{\MT(\Q)}$. Several applications of motives to number theory rely in an essential way on this upper bound. Note that even if one has the `right' definition of $\MM_{\Q}$, this upper bound is not available 
in general.

\begin{rem} Defining mixed motives as a full subcategory of realisations (\`a la Jannsen, Deligne)
 as opposed to by explicit generators and relations (\`a la Nori) 
 would not give the same answer  if the realisation functors are not fully faithful.    Furthermore, theorems about the \emph{independence} of motivic periods proved in the  ring $\Pe^{\mm}_{\HH}$  will carry over unconditionally to any reasonable
 definition of a category of mixed motives (irrespective of whether  $(\ref{eqn:  MMtoH})$ is injective or not). 
  On the other  hand, 
 when proving \emph{relations} between motivic periods, it is preferable to prove them using morphisms of mixed Hodge structures which come from  geometry, in which case
 they would also carry over to any suitably defined $\Pe^{\mm}_{\MM_{\Q}}$.
\end{rem}

\subsection{Projection map} \label{sectProjection}  An inconvenience of working with de Rham periods is the lack of a (complex) period homomorphism. One
way around this is to construct  single-valued periods as we did in \S\ref{sectSVconstant}. Another approach
is to write de Rham periods as images of motivic periods. The latter works particularly well in the mixed Artin-Tate case as we now explain.

\begin{prop}  Every effective motivic period of weight zero is a motivic algebraic number. The period map gives  an isomorphism
 $$\mathrm{per}: W_0 \Pe^{\mm,+} \overset{\sim}{\rightarrow} \overline{\Q} \ . $$
\end{prop} 
\begin{proof} See \S\ref{sectExfamily} below.
\end{proof}

Suppose that $M$ is an object of $\HH$ which is  effective  (all Hodge numbers $h_{p,q}(M)$ vanish unless $p, q\geq0$).
Say that $M$ is \emph{separated} \label{gloss: separated}   if 
$$W_0 M_{dR} \To M_{dR} \To M_{dR} / F^1 M_{dR}$$
is an isomorphism. This implies that  there is a splitting
\begin{equation}\label{sepsplit} 
M_{dR} = W_0 M_{dR} \oplus F^1 M_{dR}\ .
\end{equation}
Equivalently,  $h_{p,q}(M) =0$ unless $(p,q)=(0,0)$ or $p,q>0$.

Define a comparison map    $c_0^t:M_B^{\vee}\otimes_{\Q} \Pe^{\mm}(W_0M) \rightarrow M_{dR}^{\vee} \otimes_{\Q} \Pe^{\mm}(W_0M)$ to be the dual of the composition $c_0$ 
of the maps
$$ M_{dR} \To M_{dR}/F^1 M_{dR} \cong W_0 M_{dR} \overset{c^{\mm}_{M}}{\To} W_0 M_B\otimes_{\Q} \Pe^{\mm}(W_0M) \subset M_B \otimes_{\Q} \Pe^{\mm}(W_0M)$$
where $\Pe^{\mm}(W_0M)$ 
is the vector space of motivic  periods  of $W_0M$. By the previous proposition, these are algebraic motivic periods (see \S\ref{sectAlgebraicNos}) in the case where $M$ is an object of the form $(\ref{ExampleHmotive})$. In this case, we obtain a linear map from the motivic periods of $M$ to its de Rham periods:
$$[ M, \sigma, \omega]^{\mm} \mapsto [M, c_0^t (\sigma), \omega]^{\dR} : \Pe^{\mm}(M) \To \Pe^{\dR}(M) \otimes_{\Q} \overline{\Q}\ .$$
If  $M$ is of Artin-Tate type (Hodge numbers equal to $(p,p)$ only) and effective, then it is necessarily separated.  So, writing $\Pe^{\bullet,+}_{HT} = \Pe^{\bullet,+ } \cap \Pe^{\bullet}_{HT}$ 
for $\bullet \in \{\mm, \dR\}$,  we obtain a linear map 
\begin{equation} \label{pidRmm}  \pi_{\dR, \mm+} : \Pe^{\mm,+}_{HT} \To \Pe^{\dR,+}_{HT} \otimes_{\Q} \overline{\Q} \ . \end{equation}

Another way to define $(\ref{pidRmm})$ is by the coaction
$$\Pe^{\mm,+}_{HT} \overset{\Delta}{\To} \Pe^{\mm,+}_{HT}\otimes_{\Q} \Pe^{\dR,+}_{HT} \To \overline{\Q} \otimes_{\Q} \Pe^{\dR,+}_{HT} $$
where the second map is the projection  of $\Pe^{\mm,+}_{HT}$ onto its weight $0$ component (recall that  it is graded by the weight and has non-negative  degrees).
Note that the projection map, restricted to the subring of motivic periods of mixed Tate motives, lands in $\Pe^{\dR,+}$, i.e., without tensoring with $\overline{\Q}$. 

One possible application of the projection map is to prove  identities between de Rham periods of mixed Tate motives using the complex period map.  One can even deduce identities between  $p$-adic periods using complex analysis. The idea is the following. Take a relation 
$P(\xi_1, \ldots, \xi_n)=0$
between motivic periods in, say $\Pe^{\mm,+}_{\MT(\Z)}$ for simplicity.  Such a relation can be proved by combining
 the coaction and complex analysis \cite{MZVdecomp}.   By applying the projection map we deduce a polynomial identity between de Rham periods. Finally,
 take the $p$-adic period to deduce
 $ P(\xi^{(p)}_1, \ldots, \xi^{(p)}_n)=0$
    where $\xi_i^{(p)} =  \per_{p} \pi^{\dR, \mm+} \xi_i $.  This  answers a question of Yamashita (\cite{Yamashita}, remark 3.9): the motivic Drinfeld associator $\mathcal{Z}^{\mm}$ defined in \cite{SVMP} provides a common source for relations between both the complex and $p$-adic multiple zeta values via 
    the period map $\per$ for the former, and via $\per_p \pi^{\dR, \mm+}$ for the latter.
The fact that $\Lef^{\dR}\neq 0$ but $\pi^{\dR, \mm+} \Lef^{\mm}=0$ explains the confusing fact that  it  is  sometimes  stated   that `$2\pi i=0$' and sometimes that `$2 \pi i =1$' 
in this context.

Stated differently, let $\overline{G}_{MT}^{\mm}$ and $\overline{G}_{MT}^{dR}$ be the affine  (monoid) schemes defined by the 
spectra of $\Pe^{\mm,+}_{MT}$ and $\Pe^{\dR, +}_{MT}\subset \Pe^{\dR}_{\MT}$. Then the projection is a morphism
$$G_{MT}^{dR} \To \overline{G}_{MT}^{dR} \To \overline{G}_{MT}^{\mm}$$
and a Frobenius element  $F_p \in  G_{MT}^{dR} (\Q_p)$ maps to a $\Q_p$-valued point on $
\overline{G}_{MT}^{\mm}$.

\section{Some basic examples of motivic periods} \label{sectExamples}
Before proceeding further with the discussion, we list some very simple examples of motivic periods to illustrate the concepts introduced earlier.
\subsection{Algebraic numbers} \label{sectAlgebraicNos} This is the study of Artin motives (\cite{DeP1}, 1.16) which 
in principle reduces to Grothendieck's version of Galois theory. However, the point of view of motivic periods leads to some interesting twists on this well-known tale. 
Let $P\in \Q[x]$ be an irreducible polynomial, set $F=\Q[x]/(P)$, and apply example \ref{ex: motivicperiods} with   $X= \mathrm{Spec}\, F$, $D=\emptyset$,  and $n=0$. 
The  object  $H^0(X)$ is (the realization of) an Artin motive.  Its de Rham and Betti realizations are  $H^0_{dR}(X) = F$, and 
$H^0_B(X)= H_0(X(\C);\Q) ^{\vee}= \mathrm{Hom}(F,\C)^{\vee}.$ 
Let $\overline{\Q}$ denote the algebraic closure of $\Q$ in $\C$. 
Given  $\alpha \in \overline{\Q}$ such that $P(\alpha)=0$, denote by $\sigma_{\alpha}:F \hookrightarrow \C$
the unique embedding of $F$  such that   $ \sigma_{\alpha}(x)=\alpha$. Define a \emph{motivic algebraic number} \label{gloss: motivicalgebraicnumber} 
$$\alpha^{\mm} = [H^0(X), \sigma_{\alpha},x]^{\mm} \in W_0 \Pe^{\mm,+}\ .$$
Its period is by definition $\per (\alpha^{\mm}) = \alpha$.  
The diagonal  $X\rightarrow X\times X$ gives rise to a morphism 
$H^0(X)\otimes H^0(X)\rightarrow  H^0(X)$ in the category $\HH$.  Using this and the defining relations between  matrix coefficients, we
deduce that 
$$f(\alpha^{\mm}) = [ H^0(X), \sigma_{\alpha}, f(x)]^{\mm}$$
for $f = x^n$ and by additivity, for  any polynomial $f\in \Q[x]$. 
By embedding $\Q(\alpha), \Q(\beta)$ into $\Q(\alpha, \beta)$,  we deduce that $\alpha \mapsto \alpha^{\mm} : \overline{\Q} \rightarrow  \Pe^{\mm,+}$ is a homomorphism. It follows that  the period map is an isomorphism
 \begin{equation}
 \label{perisomonalg}  \langle \alpha^{\mm}: \alpha \in \overline{\Q}\rangle_{\Q} \overset{\per}{\To} \overline{\Q} \subset \C
 \end{equation}
 so we can identify algebraic numbers with their motivic versions.
Note that the minimal object $M_{\alpha^{\mm}}$ associated to  $\alpha^{\mm}$ is a strict subquotient of $H^0(X)$ whenever  $\alpha \notin \Q$, since it is a factor of  $\mathrm{coker}(H^0(\Spec \Q)  \rightarrow H^0(\Spec F) )$.  The category of Artin motives  $AM_{\Q}$ over $\Q$ is  equivalent to the full tensor subcategory of $\HH$ generated by the objects $H^0(X)$, and we could take this as its definition.

It is customary to consider only the Betti Galois group. 
The absolute Galois group $\mathrm{Gal}(\overline{\Q}/\Q)$ acts on the left on $\mathrm{Hom}(F,\C) = \mathrm{Hom}(F, \overline{\Q})$ via its action on $\overline{\Q}$
and gives an automorphism of the Betti fiber functor. 
Indeed, as  is well-known,  the Betti Galois group $G_{AM(\Q)}^B$ is the constant group scheme over $\Q$ corresponding to $\mathrm{Gal}(\overline{\Q}/\Q)$.   Therefore the  (right) action of $G_{AM_{\Q}}^{B}(\Q)$ on 
 motivic algebraic numbers  is equivalent to the 
(left) action of  $\mathrm{Gal}(\overline{\Q}/\Q)$ on $\overline{\Q}$ via the isomorphism $(\ref{perisomonalg})$.
  The action of real Frobenius $F_{\infty}$ on the latter corresponds to complex conjugation on the former. The story usually ends here.

Now consider, somewhat unconventionally, the  case of the de Rham Galois group. 
Its action  on $H^0_{dR}(\Spec F)$ respects the diagonal map and hence the  multiplication on $F$. Furthermore,  it preserves $H^0_{dR}(\Spec K)$ for all subfields $ K \subset F$. 

\begin{defn}  Consider  the functor
$$\mathcal{A}_F(R) = \{ \alpha \in \mathrm{Aut}_R( F\otimes_{\Q} R) \hbox{ such that } \alpha (K\otimes_{\Q} R) \subset K\otimes_{\Q} R \hbox{ for all } K \subset F\} $$
 from commutative $\Q$-algebras $R$ to groups. 
 Define the \emph{group  of field automorphisms} to be the projective limit over all  field extensions $F/\Q$ of finite type
$$\mathcal{A}_{\overline{\Q}} = \varprojlim_{F} \mathcal{A}_F\ .$$  
\end{defn} 
\noindent It follows from the fact that all algebraic relations between $\alpha^{\mm}$ are induced by linearity,  inclusions of fields $K \subset F$ and  diagonals $\Spec F \rightarrow \Spec F \times \Spec F$ that   
$$G^{dR}_{AM(\Q)} \cong \mathcal{A}_{\overline{\Q}}$$ (which shows in particular that the right-hand side  is representable: its affine ring is generated by 
matrix coefficients $[H^0(\mathrm{Spec}\, F), f,v]^{\dR}$, where  $f\in F^{\vee}$ and $v\in F$).
The comparison isomorphism implies that 
$$G^B_{AM(\Q)} \times \C \overset{\sim}{\To} G_{AM(\Q)}^{dR} \times \C$$ so the usual absolute Galois group can be  retrieved as the complex points (or $\overline{\Q}$-points) of the de Rham Galois group.

Now consider a motivic algebraic number $\alpha^{\mm}$, for $\alpha \in \overline{\Q}$.   The 
 \emph{degree} of  $\alpha$ can be  retrieved as the number of connected components of $G^{\bullet}_{\alpha^{\mm}}$ for $\bullet = B,dR$.  Its \emph{minimal object} $M(\alpha^{\mm})$
 is an object of $AM_{\Q}$.  Its  periods generate the Galois closure $F'$ of $\Q(\alpha)$.  Its de Rham group scheme is $G^{dR}_{\alpha^{\mm}} = \mathcal{A}_{\Q(\alpha)}$, and its  Betti group $G^{B}_{\alpha^{\mm}}$ is the constant group scheme of $\mathrm{Gal}(F'/\Q)$.   The quantity $\mathrm{rank}( \alpha^{\mm}) = \dim_{\Q} M_B(\alpha^{\mm})$ is the dimension of the  $\Q$-vector space spanned by the Galois conjugates of $\alpha$ over $\Q$. 
This is called the \emph{conjugate dimension}  of $\alpha$ and, surprisingly,  was introduced only very recently  (see  \cite{conjdim} and references therein). 

Note that although the Betti orbit $G^{B}(\Q)\alpha^{\mm}$  of $\alpha^{\mm}$ corresponds the usual notion of Galois conjugates of $\alpha$,
 the de Rham orbit $G^{dR} (R) \alpha^{\mm}$ is sensitive to $R$.

\subsection{Motivic $2 \pi i$ (Lefschetz motive)}  \label{sectLefschetz} Let $X= \Pro^1 \backslash \{0,\infty\}$, $D=\emptyset$. Consider
$ H^1(X) = \Q(-1)$ in example $\ref{ex: motivicperiods}$.
 Its de Rham version is $H^1_{dR}(X;\Q) = \Q [{dx\over x}]$ and its Betti version is
$H_1(X(\C)) = \Q[\gamma_0]$ where $\gamma_0$ is a small loop winding around $0$ in the positive direction.
 Define the Lefschetz motivic period\footnote{In \cite{SVMP}  the Lefschetz motivic period was viewed as an object in $\Pe^{\mm}_{\MT(\Z)}$, where $\MT(\Z)$ is the 
 category of mixed Tate motives unramified over  $\Z$. There is an  injection $\Pe^{\mm}_{\MT(\Z)}\rightarrow\Pe^{\mm,+} \subset  \Pe^{\mm}_{\HH}$, and the object $\Lef^{\mm}$ defined here is its image. A similar remark applies for the later examples.}
$$\Lef^{\mm} = [H^1(X), [\gamma_0],  [{dx/x}]]^{\mm} \quad \in \quad  W_2\, P^{\mm,+}\ . $$
It satisfies $F_{\infty}\Lef^{\mm}=-\Lef^{\mm}$. By Cauchy's theorem, its period is 
$$\per (\Lef^{\mm}) = \int_{\gamma_0} {dx \over x} = 2 \pi i\ .$$
It is the `motivic version' of $2\pi i$. Since $H^1_{dR}(X)$ is  a one-dimensional representation of  $G^{dR}$, we obtain a homomorphism of 
affine group schemes
$$\lambda: G^{dR} \To \mathbb{G}_m\ .$$
Thus the group $G^{dR}(\Q)$ acts upon $\Lef^{\mm}$ by multiplication 
$$ g (\Lef^{\mm})  = \lambda(g) \Lef^{\mm}\ ,$$
where $\lambda(g) \in \Q^{\times}$. 
The character $\lambda_g$ is non-trivial: if  $H^1_{dR}(X)$ were the trivial representation,  then by  theorem \ref{thm:  Tannaka},  $H^1(X)$ would
be equivalent to the trivial object $\Q(0)=H^0(pt)$ which has rational periods.  Since
$$\per(\Lef^{\mm}) = 2  \pi i  \notin \Q$$
is irrational, we conclude that $\lambda$ is non-trivial (this also follows  from the fact that the  Hodge structure on $H^1(X)$ is  $\Q(-1)$ which is pure of  weight 2). It follows that $\Lef^{\mm}$ is transcendental: if there were a polynomial $P \in \Q[x]$ such that
$P(\Lef^{\mm}) = 0 $,
 then every conjugate $\lambda(g) \Lef^{\mm}$ would also be a root of $P$. Since a non-zero polynomial has only finitely many roots,
it would follow that $P=0$. 

On the other hand, it is convenient to define the de Rham version of $\Lef^{\mm}$, denoted  $\Lef^{\dR}\in \Pe_{\HH}^{\dR} = \Or(G^{dR})$, to be  the matrix coefficient
$$\Lef^{\dR} = [H^1(X) ,  [dx/x]^{\vee} ,[dx/x] ]^{\dR} \ .$$
The coaction $\Delta (\Lef^{\mm} ) = \Lef^{\mm} \otimes  \Lef^{\dR} $ is given by application of $(\ref{eqn: coaction})$, and  it follows that 
$g(\Lef^{\dR}) = \lambda(g) \Lef^{\dR}$, and $\ev(g(\Lef^{\dR}))= \lambda(g)$, since $\ev(\Lef^{\dR})=1$.

\subsection{Motivic logarithms (Kummer motive)}  Let $X= \Pro^1\backslash \{0,\infty\}$ and $D= \{1, \alpha\}$ for some $1< \alpha \in \Q$.
Consider the object in $\HH$  known as a Kummer motive
$$K_{\alpha}= H^1(\Pro^1 \backslash \{0,\infty\}, \{1,\alpha\})\ .$$
It sits in an exact sequence $0 \rightarrow \Q(0) \rightarrow K_{\alpha} \rightarrow  H^1(X) \rightarrow 0$.  A basis for the de Rham cohomology
$(K_{\alpha})_{dR}$ is given by the relative cohomology classes of the forms 
${dx \over x}$ and ${dx \over \alpha-1}$, which vanish along $D$. Let $\gamma_0$ be as in \S\ref{sectLefschetz}, and  $\gamma_1$ denote the  interval $
[1,\alpha]\subset X(\R)$. Their boundaries are contained in $D(\C)$, and they form a basis for $(K_{\alpha})^{\vee}_B$. The comparison isomorphism is represented by the matrix
\begin{equation}  \label{eqn: periodmatrixforlog}
\left(
\begin{array}{cc}
\int_{\gamma_1} {dx \over \alpha-1}  & \int_{\gamma_1} {dx \over x}        \\
  \int_{\gamma_0} {dx\over \alpha-1 }  &  \int_{\gamma_0} {dx \over x}  
\end{array}\right) = \left(
\begin{array}{cc}
 1   & \log(\alpha)     \\
 0 & 2\pi i  \end{array}\right) 
 \end{equation} 
with respect to  this choice of basis. Define the motivic logarithm to be 
$$\log^{\mm}(\alpha) = [K_{\alpha}, [\gamma_1], [ \textstyle{dx\over x}]]^{\mm} \in W_2 \Pe^{\mm,+}\ .$$
Its period  is $\per(\log^{\mm}(\alpha)) = \log(\alpha)$, and $F_{\infty} \log^{\mm} \alpha = \log^{\mm} \alpha$.  The group $G^{dR}$ acts on $(K_{\alpha})_{dR}$, fixing the subspace
$\Q(0)_{dR}$ and acting on the quotient $H_{dR}^1(X)=\Q(-1)_{dR}$ via $\lambda_g$ as in the previous example. Thus  we have  a homomorphism 
$$(\nu_{\alpha}, \lambda): G^{dR} \To \GG_a \rtimes \GG_m\ .$$
 Equivalently, the   de Rham action is given for  $g\in G^{dR}(\Q)$  by
\begin{equation} \label{eqn: Gactiononlog} 
g \log^{\mm}(\alpha) = \lambda(g) \log^{\mm}(\alpha) + \nu_{\alpha}(g) \ .
 \end{equation}
For illustration, we can prove the functional equation of the motivic logarithm  as follows.  Let $1<\beta \in \Q$, and
consider the morphisms  of pairs of spaces
\begin{eqnarray} 
 (\GG_m,\{1,\alpha\}) \overset{\times \beta}{\To} (\GG_m, \{\beta, \alpha\beta\}) &  \subseteq  & (\GG_m, \{1, \beta, \alpha\beta\})  \nonumber \\
 \quad (\GG_m, \{1, z\}) &\subseteq &  (\GG_m, \{1, \beta, \alpha\beta\}) \quad  \hbox{ for }  z\in \{\beta, \alpha\beta\} \ . \nonumber 
 \end{eqnarray} 
 Since ${dx \over x}$ is invariant under multiplication, these give relations
\begin{eqnarray} 
\log^{\mm}(\alpha)  &=  & [ H^1 (\GG_m, \{1, \beta, \alpha\beta\}),  [\beta, \alpha\beta], [\textstyle{{dx \over x}}]]^{\mm} \nonumber \\ 
\log^{\mm}(z)  & =  & [ H^1 (\GG_m, \{1, \beta, \alpha\beta\}), [1, z],  [\textstyle{{dx \over x}}]]^{\mm} \quad  \hbox{ for }  z\in \{\beta, \alpha\beta\}\ . \nonumber
\end{eqnarray} 
Finally use additivity with respect to Betti classes $[1, \alpha\beta] = [1, \beta]+ [\beta, \alpha\beta]$ to obtain the expected relation between the three motivic periods
$$\log^{\mm} (\alpha\beta) = \log^{\mm}(\alpha) + \log^{\mm}(\beta)\ .$$
 It follows that the motivic logarithms over $\Q$ are linear combinations of the  motivic periods
$\log^{\mm}(p)$  for  $p\geq 2$ prime. Since $\log : \R_{>0} \rightarrow \R$ has a unique zero at $x=1$, the functional equation of the logarithm implies that  the numbers 
$\log(p)$ are linearly independent over $\Q$, and \emph{a fortiori}  the $\log^{\mm}(p)$.  By $(\ref{eqn: Gactiononlog})$  we have
$$ \Delta^u \log^{\mm}(p) = \log^{\mm}(p)\otimes 1 + 1 \otimes \nu_p$$
where $\nu_p$ is viewed in $\Or(U^{dR}_{\HH})$. We deduce that the decomposition map satisfies
$$\Phi ( \log^{\mm}(p)) = 1 \otimes \nu_p   \qquad \in \qquad \Pe^{\mm}_{\HH^{ss}} \otimes_{\Q} H^1(U^{dR}_{\HH})\ , $$
and the $\nu_p$ are independent, since $\Phi$ is injective.  
Since it is a homomorphism, $$\Phi\big( (\Lef^{\mm})^k \prod_i  \log^{\mm}(p_i)^{n_i} \big)    = (\Lef^{\mm})^k \otimes \prod_i (\nu_{p_i})^{\sha n_i}$$
where the products on the right-hand side are with respect to the shuffle product, $p_i$ are a finite set of primes, and $n_i \geq 0$.

 \begin{cor} Since $\Phi$ is injective, the set of elements $
\{\Lef^{\mm},   \log^{\mm}(p)  \hbox{ for } p \hbox{ prime}\}$  are algebraically independent over  $\Q$.
 \end{cor}
 This completes the description of all algebraic relations between the motivic periods $\log^{\mm} (\alpha)$, for $\alpha \in \Q$,  and $\Lef^{\mm}$.  

\subsubsection{Single-valued versions}

Define the    de Rham version 
$$\log^{\dR}(\alpha) =   [K_{\alpha}, [\textstyle{dx\over \alpha-1}]^{\vee},[ \textstyle{dx\over x}] ]^{\dR} \in W_2 \Pe^{\dR}\ ,$$
where $[{dx\over \alpha-1}]^{\vee} \in  (K_{\alpha})_{dR}^{\vee}$ takes the value $0$ on $  [ \textstyle{dx\over x}]$ and $1$ on $[{dx\over \alpha-1}]$. It is precisely  $\nu_{\alpha} \in \Or(G^{dR})$, and the  coaction formula  $(\ref{eqn: coaction})$ gives 
$$\Delta \log^{\mm}(\alpha) = \log^{\mm}(\alpha) \otimes  \Lef^{\dR}   +   1^{\mm}\otimes  \log^{\dR}(\alpha) \ ,$$
which is equivalent to $(\ref{eqn: Gactiononlog})$. It follows from the computations above that $\log^{\dR}(p)$ for $p$ prime are also algebraically independent over $\Q$
(use the de Rham version of $\Phi$).  The motivic period matrix associated to $\log^{\mm} \alpha$ is 
\begin{equation} \label{logmotivicperiodmatrix} C^{\mm} =  
\left(
\begin{array}{cc}
 1   & \log^{\mm}(\alpha)     \\
 0 & \Lef^{\mm} \end{array}\right) \ .
 \end{equation}
 The real Frobenius $F_{\infty}$ acts by $-1$ on the second row. Therefore 
 $$(F_{\infty} C^{\mm})^{-1} C^{\mm} =
\left(
\begin{array}{cc}
 1   & 2 \log^{\mm}(\alpha)      \\
 0 &  -1 \end{array}\right) $$
 and we deduce  that $\s^{\mm} (\Lef^{\dR}) = -1$ and $\s^{\mm} (\log^{\dR}(\alpha))= (1+ F_{\infty}) \log^{\mm}(\alpha) =2 \log^{\mm}(\alpha) $.

\subsection{Motivic multiple zeta values}  \label{sectMotivicMZV} Iterated integrals on the punctured projective line provide a  class of motivic periods for which one knows how to compute the motivic coaction. This is most developed in the case of multiple zeta values.
For any $n_1,\ldots, n_{r-1}\geq1$ and $n_r\geq 2$, there are motivic multiple zeta values
\begin{equation} \label{mmzv}
\zeta^{\mm}(n_1,\ldots, n_r) \in \Pe^{\mm}_{HT} \cap\Pe^{\mm,+}  \subset \Pe^{\mm,+}
\end{equation} 
of weight $2n_1+ \ldots +2n_r$ (recall $\Pe^{\mm,+}_{HT}$ is graded by $W$) whose periods are
$$\per(\zeta^{\mm}(n_1,\ldots, n_r)) = \zeta(n_1,\ldots, n_r) = \sum_{1\leq k_1 <\ldots  < k_r} {1 \over k_1^{n_1} \ldots k_r^{n_r}}\ .$$ 
They are defined as follows. Let $X= \Pro^1 \backslash \{0,1,\infty\}$ and set
$$\zetam(n_1,\ldots, n_r) = [ \Or(\pi_1^{\mm}(X, \tone_0, -\tone_1),    \mathrm{dch} , w ]^{\mm}$$
where $\pi_1^{\mm}(X)$ is the motivic torsor of paths on $X$ \cite{DeGo} from the unit tangent vector at $0$ to minus the unit tangent
vector at $1$, $w$ is the  word  $w = e_0^{n_1-1} e_1 \ldots e_0^{n_r-1} e_1$ in $e_0={dx \over x}$, and $e_1 = {dx \over 1-x}$, and $\mathrm{dch}$ is the Betti image of the straight line path from $0$ to $1$.
For further details, see \cite{BrICM}.  This actually defines the motivic multiple zeta values as motivic periods of the category $\MT(\Z)$ of mixed Tate 
motives over $\Z$. The latter admits  a fully faithful functor to the category $\HH$  \cite{DeGo}, and so the ring of periods  $\Pe^{\mm}_{\MT(\Z)}$ injects into  $\Pe^{\mm}_{\HH}$
and we can  identify it with its image.  Furthermore, Beilinson's construction of the motivic  torsor of path  given in \cite{DeGo} can be realised in the form of example $\ref{ex: motivicperiods}$,
so we can also view the images of the $\zetam(n_1,\ldots, n_r) \in \Pe^{\mm}_{\HH}$ as elements of $\Pe^{\mm,+}$ as claimed above.

The depth of $(\ref{mmzv})$ is defined to be $r$. 
The fact that the depth filtration is motivic implies the following bound for the unipotency degree
$$ \hbox{u. d.}( \zeta^{\mm}(n_1,\ldots, n_r)) \leq r   \   .$$
The unipotency degree has sometimes been referred to as the  `motivic depth'. A fascinating feature of multiple zeta values is  the existence of a discrepancy between the unipotency degree and the depth,  related to modular forms for $\mathrm{SL}_2(\Z)$. 

  The simplest examples are the motivic zeta values, $\zetam(2n+1) \in C^1 \Pe^{\mm,+}$, for $n\geq 1$, which admit  
a motivic period matrix of the expected form
$$\left(
\begin{array}{cc}
 1   & \zetam(2n+1)     \\
 0 & (\Lef^{\mm})^{2n+1} \end{array}\right) \ 
  $$
   Let us define some symbols $f_{2n+1} \in H^1(U^{dR}_{\HH})$, for $n \geq 1$, as the images of the motivic zeta values under the decomposition map
   $$\Phi (\zetam(2n+1)) = 1\otimes f_{2n+1}$$
Each $f_{2n+1}$, for $n\geq 1$ spans a copy of $\Q(-2n-1)$ and has weight $4n+2$. The interpretation of these elements will be explained in \S \ref{sectClassif} below.  Either from the explicit
formula for the coaction on motivic multiple zeta values, or from the results of \S \ref{sectClassif}, the decomposition map gives an injective homomorphism 
$$\Phi: \gr^C \Pe^{\mm,+}_{\MT(\Z)} \To \Q[ \Lef^{\mm}] \otimes_{\Q} \Q \langle f_3, f_5,\ldots  \rangle $$
where the right-hand side denotes the shuffle algebra (tensor coalgebra) on symbols $f_{2n+1}$ over $\Q$. 
In this case, it is in fact known that $\Phi$ is an isomorphism (theorem \ref{thmdecompisom}). Now, the main result of \cite{BrMTZ}
is a computation of the image  under $\Phi$ of the elements 
$\zetam(n_1,\ldots, n_r)$ where $n_i \in \{2,3\}$, and a proof  that their images are linearly independent.\footnote{The use of the  decomposition map considerably simplifies  many of the arguments of \cite{BrMTZ}}  Thus we can use these elements to split the 
coradical filtration $C$ on $\Pe^{\mm}_{\MT(\Z)}$, and deduce the existence of a canonical  isomorphism \cite{MZVdecomp} 
$$\phi \quad : \quad  \Pe^{\mm,+}_{\MT(\Z)} \cong \gr^C \Pe^{\mm,+}_{\MT(\Z)}  \overset{\Phi}{\To} \Q[ \Lef^{\mm}] \otimes_{\Q} \Q \langle f_3, f_5,\ldots  \rangle\ . $$
One could use a different splitting of the coradical filtration $C$, which would lead to a different choice of isomorphism $\phi$.

It is  hard  to  understand Galois aspects of  multiple zeta values without some sort of model of this kind.
Indeed, using this model we can  easily write down the invariants defined earlier. 
If $\xi \in \Pe^{\mm,+}_{\MT(\Z)}$ corresponds to $(\Lef^{\mm})^{k} f_{a_1} f_{a_2} \ldots f_{a_r}$ under $\phi$, then
the representation generated by $\xi$ is the vector space
$$M(\xi)_{dR} = \langle  \ell^{k} f_{a_1} \ldots f_{a_i}:  \hbox{ for } 0\leq i \leq r \rangle_{\Q}\ $$
 obtained by slicing off letters from the right.   A representative for the period matrix for $\xi = (\Lef^{\mm})^{k} f_{a_1} f_{a_2} $ is
$$
\per\left(
\begin{array}{ccc}
      (\Lef^{\mm})^{k}   &  (\Lef^{\mm})^{k} f_{a_1} &     (\Lef^{\mm})^{k} f_{a_1} f_{a_2}     \\
0 &    (\Lef^{\mm})^{k+a_1}   &     (\Lef^{\mm})^{k +a_1} f_{a_2}     \\
0 & 0 &  (\Lef^{\mm})^{k +a_1+ a_2}           \end{array}
\right)
$$
which means the top left-hand entry is $\per((\Lef^{\mm})^k) = (2 \pi i)^k$,  and so on.
The general pattern is clear from this example.

Applying the projection map $\pi$ to motivic multiple zeta values leads to \emph{de Rham} multiple zeta values
$\zeta^{\dR}(n_1,\ldots, n_r) = \pi^{\dR,\mm +} \zetam(n_1,\ldots, n_r)$.
It is proved in \cite{BrMTZ} that the kernel of $ \pi_{\dR,\mm +}$ on the ring generated by motivic multiple zeta values is the ideal generated by $\zetam(2)$. Thus de Rham 
multiple zeta values are motivic MZV's modulo $\zetam(2)$. The former have single-valued periods, and a calculation using the period matrix for $\zetam(2n+1)$  similar to the one for the logarithm gives the single-valued versions 
  $\s^{\mm}(\zeta^{\dR}(2n+1)) = 2 \zetam(2n+1)$.
  The de Rham versions of multiple zeta values also have $p$-adic periods, 
  which  can be thought of as follows. There are canonical Frobenius elements \cite{Yamashita}
$$F_p \in G^{dR}_{\MT(\Z)}(\Q_p)\ ,$$
and hence  homomorphisms $\per_p : \Pe^{\dR}_{\MT(\Z)} = \Or(G^{dR}_{\MT(Z)}) \rightarrow \Q_p$. 
The projection map enables us to associate $p$-adic periods to motivic multiple zeta values, which are a certain kind of  $p$-adic multiple zeta values.\footnote{This point of view  quickly leads to new constructions. For example, one can consider      curious  hybrid quantities defined by    the
 convolution of $\per$ with $\per_p$:
 \begin{equation}  \nonumber
 \zeta_{\R *p }(n_1,\ldots, n_r ) : =  m( \per \otimes \per_p \pi^{\dR,\mm+}) \Delta \zetam(n_1,\ldots, n_r)  \in    \R \otimes_{\Q} \Q_p \ ,
 \end{equation}
where $m$ is multiplication, and $\Delta$ the coaction $(\ref{eqn: coaction})$.  }

\subsection{Motivic Euler sums} \label{sectMotEuler}
Euler sums are defined by the nested sums
$$\zeta(n_1,\ldots, n_r) = \sum_{1 \leq k_1 < \ldots<  k_r}  { \mathrm{sign}(n_1)^{k_1} \ldots \mathrm{sign}(n_r)^{k_r}  \over k_1^{|n_1|} \ldots k_r^{|n_r|} }$$
where $n_i \in \Z\backslash \{0\}$ and $n_r\neq 1$. Their depth is defined to be the quantity $r$. They can be written as iterated integrals
on $X= \Pro^1\backslash \{0,\pm 1, \infty\}$ from $0$ to $1$, which leads to a definition of motivic Euler sums
$$\zetam(w) = [\Or(\pi_1^{\mm}(X, \tone_0, -\tone_1),  \mathrm{dch}, w]^{\mm}$$
where  $w$ is  a certain word in  $e_0= {dx\over x}$ and $e_{\pm 1} = {dx \over x \pm 1}$, and $\mathrm{dch}$ is as above.   These are motivic periods of the category 
$\MT(\Z [ {1 \over 2}])$ of mixed Tate motives ramified at $2$. The decomposition map now provides an injective homomorphism
$$\Phi: \gr^C \Pe^{\mm, +}_{\MT(\Z[{1\over 2}])} \To \Q[\Lef^{\mm}] \otimes_{\Q} \Q \langle \nu_2, f_3, f_5, \ldots \rangle \ ,$$
where $\nu_2$, corresponding to the logarithm of $2$, was defined earlier.  It is an isomorphism by theorem \ref{thmdecompisom}.
The results of Deligne \cite{DeN}  can be translated into this setting. He proves that $\zetam(n_1,\ldots, n_{r-1}, -n_r)$ where the $n_i$ are odd $\geq 1$
and form a Lyndon word, are algebraically independent. 
An important difference with the case of multiple zeta values, which considerably simplifies matters, is that the 
depth filtration in this case coincides with the unipotency filtration. We can construct a splitting of the coradical filtration using this basis and hence  an isomorphism 
$$\phi^{(2)} :\Pe^{\mm, +}_{\MT(\Z[{1\over 2}])} \cong  \gr^C \Pe^{\mm, +}_{\MT(\Z[{1\over 2}])}  \overset{\Phi}{\To} \Q[\Lef^{\mm}] \otimes_{\Q} \Q \langle \nu_2, f_3, f_5, \ldots \rangle \ .$$
The periods of $\MT(\Z)$ correspond to elements with no $\nu_2$ in their $\phi$-image. Note, however, that the maps $\phi$ and $\phi^{(2)}$ are not  compatible.
To remedy this, one could replace  $\phi$ with the restriction of $\phi^{(2)}$  on $\Pe^{\mm,+}_{\MT(\Z)} \subset \Pe^{\mm,+}_{\MT(\Z[{1 \over 2}])}$, but this does not quite lead to an explicit basis for the periods of $\MT(\Z)$. These ideas are  studied in Glanois' thesis \cite{Glanois}, who also constructed a new basis for the motivic periods of $\MT(\Z)$ using certain modified Euler sums
where the summation involves non-strict inequalities,  weighted with certain powers of $2$.

\section{Towards a classification of motivic periods} \label{sectClassif}
We can use the decomposition into primitives  to classify $\HH$-periods up to elements 
of lower unipotency degree. 
In this section,  we shall drop the superscript $dR$  and subscript $\HH$ and write
$S, U, G$ instead of $S^{dR}_{\HH}$, $U^{dR}_{\HH}$, $G^{dR}_{\HH}.$

The decomposition map involves a space  $\gr^C_1(\Or(U))$, which is exactly
$$
  \mathrm{Prim} (\Or(U)) := \{f \in \Or(U): \Delta f = f \otimes 1 + 1 \otimes f \}  
$$
In this paragraph we analyse this space in some detail, which leads to further invariants of motivic periods, and a first step
towards their classification.

\subsection{Cohomology of $U$}
The exact sequence  $(\ref{UHexactsequence})$ will  now be written
$$ 1 \To U \To G \To S \To 1\ . $$

\begin{prop} \label{propcohomU} Let $n \geq 0$.  There is an isomorphism of (right) $S$-modules
$$H^n (U) \cong \bigoplus_{M \in \mathrm{Irr}(\HH^{ss})} \Ext^n_{\HH}(\Q, M) \otimes_{\mathrm{End}(M)} M^{\vee}_{dR}\ , $$
 where $\mathrm{Irr}(\HH^{ss})$ denotes a set of representatives of  isomorphism classes of simple objects in $\HH^{ss}$ (or equivalently, of irreducible $\Or(S)$-comodules).

\end{prop} 
\begin{proof}
First of all, we can write $U= \varprojlim U_n$ as a projective limit of unipotent affine group schemes $U_n$ of finite type (unipotent algebraic matrix groups). Likewise, a representation $V$ of $U$ is an inductive limit $V= \varinjlim V_n$   of  $U_n$-representations. Since $ \varinjlim  H^i(U_n, V_n)\overset{\sim}{\rightarrow } H^i( U, V) $,  the arguments which follow can be deduced from well-known results for matrix groups  and by taking limits. 

 Let $M$ be an irreducible  object of $\HH^{ss}$. Then  $M_{dR}$ is an irreducible $\Or(S)$-comodule. 
 A Hochschild-Serre spectral sequence gives
 $$ H^p(S, H^q(U,M_{dR})) \Rightarrow H^{p+q} (G, M_{dR})$$
and one knows that $S$ is of cohomological dimension $0$, since it is pro-reductive. Therefore since $U$ acts trivially on $M_{dR}$, 
$$ H^0(S, H^n(U,M_{dR})) =  H^n (U, M_{dR}) ^{S} \cong ( H^n(U) \otimes_{\Q} M_{dR})^{S}   \ ,$$
where $H^n(U)$ denotes $H^n(U;\Q)$, and we deduce that 
$$\big(H^n (U) \otimes_{\Q} M_{dR}\big)^S \cong H^n (G, M_{dR})\ .$$
Let $M, N$ be  irreducible $S$-modules. Then $(N_{dR}^{\vee} \otimes_{\Q}  M_{dR})^S = \mathrm{End}_S(M_{dR})$ if $N$ and $M$ are isomorphic, and zero otherwise,
by Schur's lemma.  It follows that
$$H^n( U) \cong \bigoplus_{M \in \mathrm{Irr}(\HH^{ss})} H^n(G,M_{dR})   \otimes_{\mathrm{End}_S(M_{dR})  } M^{\vee}_{dR}\ . $$
Note that since $U$ acts trivially on $M_{dR}$, we have $\mathrm{End}_G(M_{dR})= \mathrm{End}_S(M_{dR}) $. 
Since $\omega_{dR}: \HH \rightarrow \mathrm{Rep}(G)$ is an equivalence,  we deduce that 
$$H^n(G, M_{dR})= \Ext^n_{\mathrm{Rep}(G)} (\Q, M_{dR}) = \Ext^n_{\HH}(\Q(0), M)\ , $$
and $\mathrm{End}_S(M_{dR})= \mathrm{End}(M)$. 
\end{proof} 

It is a well-known fact due to Beilinson that  
$$\Ext^n_{\HH}(\Q, M) =0 \qquad \hbox{ for } n \geq 2\ .$$
\begin{cor} The cohomology  $H^n(U;M) $ vanishes  for all $n \geq 2$.  \end{cor}

Recall that 
$$H^1(U) \cong  \gr^C_1 \Or(U)\cong C^1 \Or(U)_+   \cong \mathrm{Prim} (\Or(U)) \ . $$ 

\begin{thm}  \label{thmdecompisom} The decomposition map
$$\Phi: \gr^C_{\bullet}  \Pe^{\mm}_{\HH} \To\Pe^{\mm}_{\HH^{ss}} \otimes_{\Q} T^c( \gr^C_1 \Or(U) ) $$
is an isomorphism of $S$-modules. 
\end{thm}
\begin{proof}  By proposition $\ref{lemPemtensor}$,   there is a non-canonical isomorphism $ \Pe^{\mm}_{\HH} \otimes_{\Q} \overline{\Q} \cong  \Pe_{\HH^{ss}}^{\mm}   \otimes_{\Q} \Or(U) \otimes_{\Q} \overline{\Q}.$
 The group $U$ is of cohomological dimension $1$, by the previous corollary. Now apply corollary $\ref{cordecompisom}$  with $T=  \Pe_{\HH^{ss}}^{\mm} \otimes_{\Q} \overline{\Q}$
 to conclude that the decomposition map, after extending scalars to $\overline{\Q}$, is an isomorphism. Since it was already  injective over $\Q$, it follows that it is surjective over $\Q$.
 \end{proof}

 \begin{rem} One can view $T^c(H^1(U) )$ as the associated graded, for the length filtration, of $H^0(\mathbb{B}(N))$ where $N$ is a DGA which
 computes the cohomology of $U$, and $\mathbb{B}$ is the (reduced) bar construction.  The decomposition map $\Phi$ therefore resembles  the bar construction of a fibration, and suggests thinking about elements of
  $\Pe^{\mm}_{\HH^{ss}}$ as functions on a `base' corresponding to $S$, and $T^c( H^1(U) )$ as iterated integrals on a `fiber' corresponding to $U$. 
 From this point of view, $\delta$ can be thought  of as a kind of  Gauss-Manin connection.
   If one wants to copy this setup for mixed Tate motives over number fields rather than mixed Hodge structures, this suggests replacing $N$ with 
 Bloch's cycle complex
 $\mathcal{N}$   and echoes the construction of \cite{Bloch-Kriz}.
 \end{rem}

\subsection{Primitives in $\Or(U)$} We analyse the
statement of  proposition  $\ref{propcohomU}$ in more detail in the case $n=1$. It gives  an isomorphism of $S$-modules
 \begin{equation} \label{primtoext} 
 \mathrm{Prim}(\Or(U)) \overset{\sim}{\To}  \bigoplus_{M \in \mathrm{Irr}(\HH^{ss})}\Ext^1_{\HH} (\Q(0), M^{\vee}) \otimes_{\mathrm{End}(M)} M_{dR}\ .
 \end{equation}
 First of all,  observe from remark $\ref{remunipmonod}$. 
 that 
 $$\mathrm{Prim}(\Or(U^{ab})) \overset{\sim}{\rightarrow} \mathrm{Prim} (\Or(U))\ .$$ 
 The action of  $S$  by conjugation on $U^{ab}$ induces an action of $S$ on  $\Or(U^{ab})$, and preserves the space of primitive elements.
Since $S$ is the (de Rham)  Tannaka group of  $\HH^{ss}$, the $S$-module generated by any element $f\in \mathrm{Prim}(\Or(U))$ defines a representation of $S$, and hence an  object  of $\HH^{ss}$ by theorem $\ref{thm:  Tannaka}$.
 
\begin{defn} For any $f\in \mathrm{Prim}(\Or(U))$, let $M_f$ denote the associated object of $\HH^{ss}$.  Its  de Rham vector space is the $\Or(S)$-comodule generated by $f$. It  comes equipped with a distinguished element $f\in (M_f)_{dR}$.  \end{defn}

 Let $f\in \mathrm{Prim}(\Or(U))$. One  associates an extension to $f$ as follows.  Consider  the short exact sequence of  right $U$-modules 
$$0 \To C_0 \Or(U) \To C_1 \Or(U) \To \mathrm{Prim}(\Or(U)) \To 0 \ .$$
It is not split, although   the underlying sequence of $\Q$-vector spaces  is  split by the augmentation map. By remark $\ref{remunipmonod}$, it can be rewritten
$$ 0 \To \Q \To C_1 \Or(U^{ab}) \To \mathrm{Prim}(\Or(U^{ab})) \To 0 \ .$$
It is an exact sequence in the category of right $U^{ab} \rtimes S$-modules.   We can pull back this extension along the 
inclusion $  (M_f )_{dR} \subseteq \mathrm{Prim}(\Or(U^{ab}))$ to obtain 
$$0 \To \Q \To \mathcal{E}_{dR} \To (M_f)_{dR} \To 0\ .$$
Now  choose an isomorphism $G/[U,U] \rightarrow U^{ab} \rtimes S$, i.e., a splitting of $1 \rightarrow U^{ab} \rightarrow G / [U,U] \rightarrow S \rightarrow 1$.  It exists by Levi's (Mostow's) theorem. Via this isomorphism, the  previous exact sequence  can be viewed  in the category of $G$-modules, and
hence, via the Tannaka theorem, as an exact sequence in $\HH$:
\begin{equation} \label{extclassdef} 
0 \To \Q \To \mathcal{E} \To M_f \To 0\ .
\end{equation} 
Another  choice of isomorphism $G/[U,U] \cong U^{ab} \rtimes S$ yields an isomorphic extension.
 The  dual extension $0 \rightarrow M_f^{\vee}  \rightarrow\mathcal{E}^{\vee} \rightarrow \Q \rightarrow 0$, together with the vector $ f \in (M_f)_{dR}$,  defines  a class
$$ [\mathcal{E}^{\vee}] \otimes f  \quad \in \quad  \mathrm{Ext}^1_{\HH}(\Q, M_f^{\vee}) \otimes_{\Q}  (M_f)_{dR}  $$
as required. By decomposing $M_f$ into $S$-isotypical components, we can project this element
into the right-hand side of $(\ref{primtoext})$. 
\begin{defn} Let us 
denote the extension class of $(\ref{extclassdef})$  by $\mathcal{E}_f$. 

\end{defn}

 In the other direction, consider an extension  in $\HH$
$$0 \To M^{\vee} \To \mathcal{E} \To \Q \To 0\ ,$$
and a vector $v\in M_{dR}$, where $M$ is a simple object. 
Choose a lift of  the element $1 \in \Q_{dR}$ to $f \in \mathcal{E}_{dR}$, and  a lift of $v$ to $\widetilde{v} \in \mathcal{E}_{dR}^{\vee}$ along the map
$\mathcal{E}_{dR}^{\vee} \rightarrow M_{dR}$. 
The image of the following unipotent matrix coefficient (\S\ref{sectSemiSimpleUnipotent})   in $\Or(U)_+$
$$\xi = [ \mathcal{E}, \widetilde{v}, f]^{\uu} \in \Or(U)_+$$
does not depend on the choices of $ \widetilde{v}, f$. For instance, if $f'$ is another lift of $1$, then $f-f' \in M^{\vee}_{dR}$, and 
$[ \mathcal{E}, \widetilde{v}, f -f']^{\uu}$ is equivalent to $[M^{\vee}, v, f-f']^{\uu}$, which is constant because $U$ acts trivially on $M_{dR}$. Similarly, if
$\widetilde{v}'$ is a lift of $v$ then $\widetilde{v}'-\widetilde{v} \in \Q_{dR}$ and $[ \mathcal{E}, \widetilde{v}'-\widetilde{v},f]^{\uu}$ is equivalent to 
a unipotent period of $\Q(0)$, hence  constant.  By Schur's lemma, a non-zero endomorphism  $\alpha: M \rightarrow M$ is an automorphism. If $\mathcal{E}^{\alpha}$ denotes  the extension $\mathcal{E}$ twisted
by $\alpha^{\vee}$, then the identity map  $\mathcal{E}\overset{\sim}{\rightarrow} \mathcal{E}^{\alpha}$ gives an equivalence of matrix coefficients $\xi = [ \mathcal{E}^{\alpha}, \widetilde{v}, \alpha_{dR}^{-1}(f)]$. It is straightforward to check using $(\ref{eqn: coaction})$ and the formulae which follow 
 that $\xi$ is a primitive element. This construction provides an inverse to  $(\ref{primtoext})$.

\subsection{Extensions in $\HH$} The contents of this section are standard and well-known.
Let $M =(M_B, M_{dR}, c)$ be an object in $\HH$. 
The following complex 
$$  W_0 M^+_B \oplus F^0 W_0 M_{dR} \overset{\id - c }{\To} (W_0 M_{B}\otimes_{\Q} \C)^{c_{dR}}$$
represents $R\mathrm{Hom}_{\HH}(\Q(0), M)$.  Recall that $c_{dR}$ is  complex conjugation on the right-hand factor of  $M_{dR} \otimes_{\Q} \C$. Its action 
on $M_B \otimes_{\Q} \C$ is   $c \,c_{dR} \,c^{-1}  = F_{\infty} \otimes c_B$ where $c_B$ is complex conjugation on the right-hand factor of $M_B \otimes_{\Q} \C$.

The kernel of the above complex is 
$$\mathrm{Hom}_{\HH} (\Q(0), M) \overset{\sim}{\To} W_0M_B^+ \cap  c(F^0 W_0 M_{dR})$$
and the map is given by the image of $1\in \Q_{dR} \overset{c}{\cong} \Q_{B}$. 
The cokernel is \cite{Carlson} 
\begin{equation} \label{Ext1formula}
\mathrm{Ext}^1_{\HH}(\Q, M)  \overset{\sim}{\To} W_0M_B^+\backslash  (W_0M_B\otimes_{\Q} \C)^{c_{dR}} / c( F^0W_0 M_{dR})\ .  \end{equation}  
  The map is given as follows. If $\mathcal{E}$ is an extension of $\Q(0)$ by $M$ in $\HH$, 
it gives rise, after applying a fiber functor $\bullet = B/dR$,  to two exact sequences
$$0 \To M_{\bullet} \To \mathcal{E}_{\bullet} \To \Q_{\bullet} \To 0 \ .$$
 Choose $B$ and $dR$ splittings  by choosing a lift of $1_B \in \Q_B$ to  $1_B \in W_0 \mathcal{E}_B^+$ and of $1_{dR} \in \Q_{dR}$ to $1_{dR} \in F^0 W_0 \mathcal{E}_{dR}$. 
Then  $1_B - c(1_{dR})$ gives a well-defined element in the right-hand side of $(\ref{Ext1formula})$.
Note that $(\ref{Ext1formula})$ is uncountably generated.
 Let $\HH(\R)$ denote the category of triples $(M_B, M_{dR}, c)$ where now $M_B, M_{dR}$ are vector spaces over $\R$ (replace the ground field
$\Q$ by $\R$). There is a  functor 
$  \otimes \R : \HH \rightarrow \HH(\R)$, 
sending $(M_B, M_{dR}, c)$ to $(M_B \otimes \R, M_{dR} \otimes \R, c\otimes \id)$. 

\begin{cor} Suppose that $W_{-1}M=M$. Then 
\begin{equation} \label{dimensionformula}
\dim_{\R} \Ext_{\HH(\R)}^1(\R(0),M\otimes \R)=  \dim_{\Q}  M^-_B  - \dim_{\Q} F^0 M_{dR}
\end{equation} 
\end{cor}   
\begin{proof}  Since $c_{dR}, F_{\infty}$ act trivially on $c(F^0 M_{dR}) \cap M_B^+$,  so too must $c_B$, since  by \S \ref{sect: PeriodsoverQ}  we have $F_{\infty} \otimes c_B = c c_{dR} c^{-1}$.
It follows that $c(F^0 M_{dR}) \cap M_B^+ \subset F^0 \cap \overline{F}^0 = 0$, since $M$ has weights $\leq -1$.  
Now $(M_B\otimes_{\Q} \C)^{c_{dR}}  = (M_B^+ \otimes \R) \oplus (M_B^{-} \otimes i \R)$, and  conclude using $(\ref{Ext1formula})$ together with the fact that $W_0M= M$.  
\end{proof}

The formula $(\ref{Ext1formula})$, together with $(\ref{primtoext})$ and theorem $\ref{thmdecompisom}$ provides a   complete  description of $\HH$-periods, graded for the coradical filtration,  in terms of semi-simple objects in $\HH^{ss}$. 
In practice, we often wish to fix a full Tannakian subcategory of pure objects in $\HH^{ss}$ (such as the one generated by  Tate objects $\Q(n)$), and consider all $\HH$-periods of objects whose semi-simplifications are of this type (periods of mixed Tate objects, in this case).  The above results give a precise description for the structure of $\HH$-periods of this type. 
  
  In order for this to be an accurate reflection of the structure of motivic periods, we need to know something about the image
  of the decomposition map $\Phi$ on the subspace $\Pe^{\mm,+}$, which we address presently.

\subsection{Speculation and context} 
Recall that $\Pe^{\mm,+}\subset \Pe^{\mm}_{\HH}$ was the ring of motivic periods, i.e., those which come from the cohomology
of an algebraic variety, and $G^{dR}$ is the quotient of $G^{dR}_{\HH}$ acting faithfully on $\Pe^{\mm,+}$. Let $U^{dR}$ denote its unipotent radical.
Let $\Pe^{\mm,+}_{ss} =( \Pe^{\mm,+})^{U^{dR}}$ denote the invariants under $U^{dR}$,  and  set
$$\mathcal{M} = C^1 \Or(U^{dR})_+ = \mathrm{Prim} (\Or(U^{dR}))\ .$$
Via $(\ref{primtoext})$ we think of $\mathcal{M}$ as `motivic' extension classes.

\begin{conj} \label{conjDecomp}  The decomposition map induces an isomorphism 
$$\Phi : \gr^{\mathcal{C}} \Pe^{\mm,+}\To  \Pe^{\mm,+}_{ss}  \otimes_{\Q} T^c(\mathcal{M})\  .$$
 \end{conj} 
This  conjecture  is a generalisation of Goncharov's  freeness conjecture for mixed Tate motives, and states that there should be 
no relations between the decompositions  of motivic periods.  We can now try to describe the constituent pieces. This is all conjectural, so 
we shall be brief.

The   putative Tannakian category of mixed motives over $\Q$ should have a functor 
$$\MM_\Q \overset{h}{\To} \HH $$
where $h= (\omega_B, \omega_{dR}, \comp_{B,dR})$ is fully faithful and hence morphisms
$$\Ext^1_{\MM_\Q}(\Q, M) \hookrightarrow \Ext^1_{\HH}(\Q, h(M)) \overset{ \otimes \R}{\To} \Ext^1_{\HH(\R)}(\R, h(M)\otimes \R) \ . $$
One expects $\Pe^{\mm,+}_{ss}$ to be generated by the cohomology of smooth projective algebraic varieties $X$ over $\Q$. 
There is a  definition for the  group $\Ext^1_{\MM_\Q}(\Q(0), M)$,  when $M=H^p(X)(q)$,
 in terms of  motivic cohomology \cite{MotCohom}.  Thus we expect $\mathcal{M}$ to be generated by the image of 
$H^{p+1}_{\mathcal{M}} (X, \Q(q)) \otimes H^p_{dR}(X)^{\vee}(-q)$  in the right-hand side of $(\ref{primtoext})$. Here, motivic cohomology
$H^{p+1}_{\mathcal{M}}(X,\Q(q))$ can be  defined either as a piece of the Adams grading of the algebraic $K$-theory of $X$, or via Bloch's 
higher Chow groups.  Finally,  Beilinson's   conjectures predict  the  rank of these groups.  In the simplest possible  case when  $M$ is pure and of weight $\leq -3$, then the  image of $\Ext^1_{\MM_\Q}(\Q, M)$  in $ \Ext^1_{\HH(\R)}(\R, h(M)\otimes \R) $ should be a lattice and its rank given by $(\ref{dimensionformula})$. See \cite{Nekovar} for further details.

  Putting these conjectural pieces together   gives a fairly complete  but highly speculative  picture for the structure of the ring of  motivic periods.
  In particular, we obtain a precise prediction for the `size' of the ring of motivic periods of given types. 
A strategy that one can pursue is to fix a given tensor category of pure motives (for example Tate motives), and try to construct geometrically
the iterated extensions (or equivalently, their motivic periods). The decomposition map is a tool to show, by computing periods, that the extensions one has are independent (\S\ref{sectMotivicMZV}, \S\ref{sectMotEuler})

\begin{rem} A more detailed account of this subject would include a discussion of regulators and special values of $L$-functions. In the present framework, one can translate Deligne's conjecture on critical $L$-values as giving a formula  for certain motivic periods of unipotency degree $0$. Beilinson's conjectures give a formula
for certain determinants of  motivic periods of unipotency degree $1$. It is tantalising to speculate that this might be the beginning of a tower of conjectural formulae describing certain  periods of all higher unipotency degrees.
\end{rem}

\subsection{Representatives for primitive elements} \label{sectRepresentatives} In this section we return  to general  $\HH$-periods.
One would like to  represent elements of $\gr^C_1 \Or(U^{dR}_{\HH})$ as concretely as possible.
It follows from corollary  \ref{cordecompisom} that  the decomposition in degree one
$$\Phi: \gr^C_1 \Pe^{\mm}_{\HH} \To  \Pe^{\mm}_{\HH^{ss}} \otimes_{\Q} \gr^C_1 \Or(U^{dR}_{\HH} )  $$
is surjective.  However, there is no canonical  map from $\gr^C_1 \Or(U^{dR}_{\HH})$ to $\Pe^{\mm}_{\HH}$, nor can we 
assign a  period to an element of $\gr^C_1 \Or(U^{dR}_{\HH})$ in any obvious way.
In some special cases, one can in fact assign numbers to primitive elements via the following two constructions:
\begin{enumerate}
\item  Call   $f\in \mathrm{Prim}(\Or(U^{dR}_{\HH}))$ \emph{stable}  \label{gloss: stable}  if $F^1 M_f =M_f$ and $M_f$ is effective.  This implies that all its Hodge numbers $h_{p,q}$ vanish for $p\leq 0$ or $q\leq 0$.  In this case, a representative  for the extension class  $ \mathcal{E}_f$
is separated (\S\ref{sectProjection}), and its de Rham realisation  splits by $(\ref{sepsplit})$ :  $$(\mathcal{E}_f)_{dR}  \cong \Q \oplus (M_f)_{dR} .$$
Thus we can view  $f\in (M_f)_{dR} =F^1 (\mathcal{E}_f)_{dR} $ and $1\in \Q^{\vee}_{dR} = F^0 \mathcal{E}^{\vee}_{dR}$ and define a canonical de Rham period
$$\xi_f= [ \mathcal{E}_f,  1, f  ]^{\dR} \in \Pe^{\dR}_{\HH}\ .$$
Its single-valued version $\s^{\mm}(\xi_f)$ lies in  $\Pe^{\mm}_{\HH}$ and we can take its period to obtain a number. The action of  $G^{dR}_{\HH}$ on $\s^{\mm}(\xi_f)$ is compatible with the conjugation action   of $G^{dR}_{\HH}$ on $\xi_f$.

\item  As in $(1)$, but also assume  that $(M_f)_B^+=0$. Then
  $(\mathcal{E}^+_f)^{\vee}_{B} \cong \Q^{\vee}_B$, and $1 \in \Q^{\vee}_B$ lifts
to $1\in (\mathcal{E}^+_f)^{\vee}_B$. We can directly define a motivic period
$[\mathcal{E}_f, 1 , f]^{\mm} \in \Pe^{\mm}_{\HH}$. Taking its period  assigns a number to such a primitive element.

\end{enumerate}

\subsection{Example: Mixed Tate motives over $\Q$}
 \label{exampleMTcase}   One of the  few situations in which Beilinson's conjectures are completely known  is  the category $\MT(\Q)$ of mixed
Tate motives   over $\Q$. Its simple objects are  Tate motives $\Q(n)$.   The  real Frobenius $F_{\infty}$ acts  on $\Q(n)_B$ by $(-1)^n$.   Thus
$$\Ext^1_{\MT(\Q)}(\Q(0), \Q(n))   \To \Ext^1_{\HH}(\Q(0), \Q(n)) = \R/(2 i \pi \Q)^n   $$
which has rank one if $n$ is odd, and zero if $n$ is even. In this case, Beilinson's conjecture is known as a consequence of  deep theorems due to  Borel, and we have
\begin{equation}  \label{MTextn}  \Ext^1_{\MT(\Q)}(\Q(0), \Q(n))  \cong K_{2n-1}(\Q) \otimes_{\Z} \Q
\end{equation}
which has rank $1$ for $n\geq 3$ odd and rank $0$ for $n$ even.  For $n=1$, 
\begin{equation}  \label{MTextn2}
\Ext^1_{\MT(\Q)}(\Q(0), \Q(1)) \cong K_1(\Q) \otimes_{\Z} \Q = \Q^*\otimes_{\Z} \Q\ 
\end{equation}
 is isomorphic to the infinite dimensional $\Q$-vector space with one generator for every prime $p$.  Furthermore, all higher Ext groups vanish. 
It follows that 
$$H^1(U^{dR}_{\MT(\Q)}) = \bigoplus_{n \geq 1}  \big(  K_{2n-1}(\Q) \otimes_{\Z} \Q(-n)_{dR}\big) \ .$$
Let  $ \Pe^{\mm,+}_{\MT(\Q)}$ denote the  ring of effective periods of $\MT(\Q)$. The subspace of  semi-simple periods  is generated
by  $\Lef^{\mm}$. 
The decomposition is an isomorphism 
\begin{equation} \label{eqndecompMTQ} 
 \gr^C_{\bullet } \Pe^{\mm,+}_{\MT(\Q)}  \overset{\sim}{\To}   \Q[\Lef^{\mm}]  \otimes T^c ( \bigoplus_{n\geq 1} K_{2n-1}(\Q) \otimes_{\Z}\Q(-n))\ .
 \end{equation}
  The  Tate objects $\Q(n)$ satisfy both conditions $(1)$ and $(2)$ of \S\ref{sectRepresentatives}. A generator  $f_{2n-1}$ of the image
of $(\ref{MTextn})$ in $\Ext^1_{\HH}$ gives a rational multiple of $\zetam(2n-1)$ under the second prescription. Choose the rational multiple to be one.\footnote{Note that prescription $(1)$ applied to $f_{2n-1}$ gives the single-valued motivic  zeta value which is exactly  double that, namely $2\zetam(2n-1)$.} Similarly, choose generators  $\nu_{p}$ of $(\ref{MTextn2})$ which correspond under $(2)$ to $\log^{\mm}(p)$ for $p$ prime.

 With this choice of generators, there is an isomorphism
$$T^c( H^1(U^{dR}_{\MT(\Q)})) \cong \Q \langle \nu_{p},   f_{3}, f_5, f_7, \ldots \rangle$$
where the right-hand side denotes the shuffle algebra (tensor coalgebra) on generators
 $\nu_{p}$, for $p$ prime, which span a copy of $\Q(-1)$  of weight 2, and $f_{2n+1}$, for $n\geq 1$  which span a copy of $\Q(-1-2n)$ of weight $4n+2$. 
 
 \begin{thm}  \label{thmdecompforMTisom} The decomposition into primitives  
 $(\ref{eqndecompMTQ})$ gives an isomorphism
 \begin{equation} \label{effectiveMTQperiodsmodel}
 \gr^C \Pe^{\mm,+ }_{\MT(\Q)}  \overset{\sim}{\To} \Q[\Lef^{\mm}]    \otimes    \Q \langle \nu_{p},   f_{3}, f_5, f_7, \ldots \rangle  \ ,
 \end{equation} 
where $\Lef^{\mm}$ is the Lefschetz period \S\ref{sectLefschetz}. 
\end{thm}
Since in this case we can associate to each $f_{2n-1}$ its  canonical period $(2)$ of \S\ref{sectRepresentatives} which equals $\zeta(2n-1)$, this gives an  elementary way 
to think about periods of mixed Tate motives over $\Q$ as formal words in tensor products of odd zeta values $\zeta(2n-1)$ and 
logarithms of prime numbers. See \cite{MZVdecomp} for examples.

\begin{rem} Beilinson's conjectures hold more generally for mixed Tate motives over number fields. Since for the time being
we are considering only periods over $\Q$, this discussion is  postponed to  \S \ref{sectMToverF}.

  \end{rem}

\section{Families of periods} \label{sectFamilies}
We now sketch a possible formalism for studying  periods varying in a family, and explain how this generalises
several concepts which have been used in the physics literature in the case of iterated integrals and polylogarithms.

\subsection{Vector bundles and local systems} See  \cite{DeP1}, \S10.24-10.52.  Let $S$ be a smooth geometrically  connected scheme  over a field $k\subset \C$.  An algebraic vector bundle $V$ on $S$ is  a locally free $\Or_S$-module of finite type.  Denote the corresponding analytic vector bundle on $S^{\an}$ by  $V^{\an}=  \Or_{S^{\an}}\otimes_{\Or_S} V $. Consider the category 
 \begin{multline}  \nonumber
 \mathcal{A}(S) = \hbox{Algebraic vector bundles on } S \hbox{  equipped with an }  \\
 \hbox{integrable connection with regular singularities at infinity}  \ . 
 \end{multline}
 Let $\omega$ denote the functor which to any object of $\mathcal{A}(S)$ 
 associates the underlying vector bundle and forgets the connection. 
 The category $\mathcal{A}(S)$
 is  a Tannakian category over $k$, and $\omega$ is a fiber functor over $S$.

 \begin{defn}  The de Rham algebraic fundamental groupoid
  is the groupoid (in the category of schemes over  $k$, acting on $S$) defined by  
 $$\pi_1^{\mathrm{alg},dR}(S) = \mathrm{Aut}_{\mathcal{A}(S)}^{\otimes}(\omega)\ .$$ 
  \end{defn}

 Consider also the category 
 $$\mathcal{L}(S) =  \hbox{Local systems of  finite-dimensional } k\hbox{-vector spaces on } S(\C) \ .   $$
 For any complex point $t\in S(\C)$, the  `fiber at $t$'  defines a functor
 $\omega_t : \mathcal{L}(S) \rightarrow \mathrm{Vec}_{k},$
 and $\mathcal{L}(S)$, equipped with $\omega_t$, is a neutral Tannakian category over $\Q$. 
 
 \begin{defn} The Betti algebraic fundamental group is the affine group scheme over $k$  defined by the Tannaka group of $\mathcal{L}(S)$ 
  $$\pi_1^{\mathrm{alg},B}(S, t) = \mathrm{Aut}_{\mathcal{L}(S)}^{\otimes}(\omega_t)\ .$$ 
   Given two complex points $t, t' \in S(\C)$, the fibers of the Betti algebraic  groupoid over $(t,t')$ are 
    $\pi_1^{\mathrm{alg},B}(S, t,t') =\mathrm{Isom}^{\otimes}_{\mathcal{L}(S)}(\omega_t, \omega_{t'}).$ 
      \end{defn} 

Denote the ordinary topological fundamental group of $S(\C)$ at a point $t \in S(\C)$ by $\pi_1^{\mathrm{top}}(S(\C), t)$, and recall that there is an equivalence of 
categories 
$$ \mathcal{L}(S) \overset{\sim}{\To}  \hbox{Finite-dimensional representations over $k$ of }  \pi_1^{\mathrm{top}}(S(\C), t)\ ,$$
which to a local system associates its fiber at $t$ together with its action of the topological fundamental group. Thus every element of $ \pi_1^{\mathrm{top}}(S(\C), t)$ defines an automorphism of the fiber functor $\omega_t$, giving   a natural homomorphism
 \begin{equation}  \pi_1^{\mathrm{top}}(S(\C), t) \To  \pi_1^{\mathrm{alg},B}(S, t)(k)\ 
 \end{equation}
which is  Zariski-dense. Similarly  there  is  a natural  morphism of groupoids $$  \pi_1^{\mathrm{top}}(S(\C), t,t') \To  \pi_1^{\mathrm{alg},B}(S, t,t')(k)\  $$
 where $\pi_1^{\mathrm{top}}(S(\C), t,t')$ are homotopy classes of paths from $t$ to $t'$ in $S(\C)$. 
  
Recall that  the  Riemann-Hilbert correspondence \cite{De4}, \cite{DeP1} 10.32(a),   is an equivalence of categories 
$\mathcal{A}(S \times \C)  \rightarrow \mathcal{L}(S \times  \C)$ over $\C$ (\cite{DeP1} 10.35). To a complex  vector bundle with  integrable connection $(V,\nabla)$   it assigns the locally constant sheaf of flat sections
$ (V^{\an})^{\nabla}$ of the corresponding analytic bundle.  Thus there is an isomorphism
of affine groupoid schemes over $\C$
 $$ \pi_1^{\mathrm{alg},B}(S \times \C , a,b)  \overset{\sim}{\To}  \pi_1^{\mathrm{alg},dR}(S\times \C, a,b)\ , $$
where $\pi_1^{\mathrm{alg},dR}(S, a,b)$ denotes the fiber of $\pi_1^{\mathrm{alg},dR}(S\times \C)$  over  $(a,b) \in (S\times S)(\C)$, or equivalently the 
affine group scheme over $\C$ given by $\mathrm{Isom}^{\otimes}_{\A(S\times \C)}(\omega_a, \omega_b)$ where $\omega_a,\omega_b$ denote the functors `fiber at $a,b$' respectively. 

\begin{rem} \label{remContractible} The basepoint $s$ can be replaced by any simply-connected subset $\X \subset S(\C)$. 
For any point $x\in X$, the fiber at $x$ defines a fiber functor $\omega_x: \mathcal{L}(S) \rightarrow \mathrm{Vec}_k$. A path $\gamma$ from $x$ to $x'$
defines an isomorphism of fiber functors $\omega_x \overset{\sim}{\rightarrow} \omega_{x'}$. Since   $X$ is  simply connected, such a path is unique
up to homotopy, and  the isomorphism $\omega_x =\omega_{x'}$ is canonical.  Thus $\omega_x$ depends only on  $x$ up to unique isomorphism.
We shall abusively denote  the resulting fiber functor, for any choice of $x\in X$, by $\omega_X$. 
Likewise, one can define the fundamental group of $S(\C)$ based at $\X$, which we denote by  $\pi^{\mathrm{top}}_1(S(\C), \X)$.  \end{rem} 

\subsection{A category of realizations} Let  $k=\Q$ and $S$ be as in the previous paragraph.
Based on  \cite{DeP1} \S1.21, consider the category $\HH(S)$   consisting of  triples $$( \V_{B}, \V_{dR}, c)$$
given by the following data:
 
 \begin{enumerate}
 \item  A local system $\V_B$ of finite-dimensional $\Q$-vector spaces over $S(\C)$, equipped with a finite  increasing  filtration 
 $W_\bullet \V_B$ of local sub-systems.
 
 \item  An algebraic vector bundle $\V_{dR}$ on $S$  in $\mathcal{A}(S)$ equipped with an integrable connection $\nabla: \V_{dR} \rightarrow \V_{dR} \otimes_{\Or_S} \Omega^1_S$ with regular singularities at infinity,
 a finite increasing filtration $W_\bullet \V_{dR}$ of $\V_{dR}$ by sub-objects in $\mathcal{A}(S)$,  and a finite decreasing fitration $F^{\bullet}$ 
 of  algebraic sub-bundles  satisfying Griffiths transversality
 $ \nabla : F^p \V_{dR} \subset   F^{p-1} \V_{dR} \otimes_{\Or_S}  \Omega^1_S$.
  \item An  isomorphism  of analytic vector bundles with connexion
 $$ c: \V^{\an}_{dR}  \overset{\sim}{\To} \V_B \otimes_{\Q} \Or_{S^{\an}}\ ,$$
 which respects the filtrations $W$, and where the connexion on $ V_B \otimes_{\Q} \Or_{S^{\an}}$ is the one for which sections of $V_B$ are flat. This is equivalent to an isomorphism 
 $ (\V^{\an}_{dR} )^{\nabla} \cong \V_B \otimes_{\Q} \C$ of  local systems of  complex vector spaces on $S(\C)$ which respects the weight filtrations on both sides.  
 \item The data $\V_B, c$ is functorial in the choice of algebraic closure $\C$ of $\R$. 
 In particular,  there is an isomorphism of local systems
 $$F_{\infty} : \V_B \overset{\sim}{\To} \sigma^* \V_B$$
 where $\sigma: S(\C) \overset{\sim}{\rightarrow} S(\C)$ is induced by  complex conjugation.
  \end{enumerate}
 This data is subject to  the following conditions:
\begin{itemize}
\item At each point $t\in S(\C)$, the vector space $(\V_B)_t$ equipped with the filtration $W$ and $c F$ on $(\V_B)_t \otimes_{\Q} \C$  is a  graded-polarisable mixed Hodge structure.
\item We shall not consider taking limits in these notes, but if one wishes to, one should  add further constraints \cite{SteenbrinkZucker}
to demand that $\V_B$  defines an admissible variation of mixed Hodge structures, has  locally quasi-unipotent monodromy, admits relative weight filtrations, and so on.   
\item Let $\Or_{\overline{S}^{\an}}$ denote the sheaf of antiholomorphic functions on $S^{\an}$.  Pulling back the  comparison  $(3)$ to $\overline{S}(\C)$ via $\sigma^*$ induces an $\Or_{\overline{S}^{\an}}$-linear  isomorphism $ \overline{c}: \V_{dR} \otimes_{\Or_S} \Or_{\overline{S}^{\an}} \overset{\sim}{\rightarrow}  \sigma^*(\V_B) \otimes_{\Q} \Or_{\overline{S}^{\an}}.$ The following diagram commutes:
$$
\begin{array}{ccc}
  c: \V_{dR}\otimes_{\Or_S} \Or_{S^{\an}}  & \overset{\sim}{\To}   &  \V_{B}\otimes_{\Q} \Or_{S^{\an}}   \\
 \downarrow &   & \downarrow   \\
  \overline{c}:   \V_{dR}\otimes_{\Or_S} \Or_{\overline{S}^{\an}} &    \overset{\sim}{\To}  & \sigma^*(\V_{B})\otimes_{\Q} \Or_{\overline{S}^{\an}}     
\end{array}
$$
where the vertical map on the left (resp. right) is   the identity on $\V_{dR}$  (resp. $F_{\infty}: \V_{B} \rightarrow \sigma^*(\V_{B})$) tensored with the map $f\mapsto \overline{f}: \Or_{S^{\an}} \rightarrow \Or_{\overline{S}^{\an}}.$
\end{itemize}
 The morphisms in $\HH(S)$ respect the above data.  
 
 The category $\HH(S)$  is Tannakian and has exact, faithful, tensor functors:
$$\begin{array}{ccccccc}
   \omega_{dR}:  \HH(S)  &\To  &\mathcal{A}(S) & \qquad,\qquad &   \omega_{B}:  \HH(S)  &\To  &\mathcal{L}(S)   \\
    (\V_{B}, \V_{dR}, c) & \mapsto &   \V_{dR} &  &(\V_{B}, \V_{dR}, c) & \mapsto &   \V_{B}  \nonumber \ .
  \end{array}
$$
One can think of $c$ as an isomorphism of functors (not strictly speaking fiber functors) from $\omega_{RH} \circ \omega_{dR}$ to $\omega_B \otimes \C$, where $\omega_{RH} : \mathcal{A}(S) \rightarrow \Lo(S) \otimes \C$ is $\V \mapsto (\V^{\an})^{\nabla}$.

\subsubsection{Fiber functors} \label{sectFiberfunctors} In order to motivate the following definitions, note that:
\begin{itemize}
\item families of periods (such as the dilogarithm $\Li_2(x)$) are multivalued functions, i.e., functions on a universal covering space of $S(\C)$. In applications, we are often given a region $X\subset S(\C)$ where the function has a prescribed branch, or, for example, a local Taylor expansion (the dilogarithm has an expansion  $\sum_{n\geq 1} {x^n\over n^2}$ which converges on   $ X = \{x\in \C : \quad |x|<1\}$).  This is the region where the chain of integration is unambiguous.  
\item in some  applications, including Feynman integrals, we cannot control the poles of the integrand. However  we may be given a region $Y\subset S(\C)$ which  is guaranteed to be free of poles. In this region, the integrals are finite and can be evaluated (bearing in mind that the integrals may also be multi-valued).   We shall allow the region $Y$ to be empty.
\end{itemize}

Having made these preliminary coments, we now define Betti and de Rham fiber functors relative to some extra data $X, Y$, as follows. 
  First of all, for any simply-connected  $\X\subset S(\C)$ we have, by remark \ref{remContractible}, a fiber functor
$$ \omega^{\X}_{B}=\omega_{\X} \omega_B  : \HH(S) \To \mathrm{Vec}_{\Q}$$
which neutralizes $\HH(S)$ over $\Q$. On the other hand, the functor
$$\omega : \HH(S) \To S$$
which to a triple $(\V_{B}, \V_{dR}, c)$ associates the vector bundle underlying $\V_{dR}$  and forgets the connection, is a fiber functor according to the definition \S\ref{sectGeneralities}. 
For any morphism $u:T \rightarrow S$  of schemes over $\Q$, the composite of $\omega$ followed by $u^*$ defines a fiber functor over $T$. 
Now consider a region $\Y \subset S(\C)$, and furthermore assume that $\Y$ is contained in some $U(\C)$ where $U\subset S$ is affine.
Then we can consider the ring 
$\Or_{S,Y} = \varinjlim_{U} \Or_U $
where the limit is over all open affine    $U\subset S$ such that $\Y \subset U(\C)$.  It is non-zero by assumption. 
 Define a fiber functor by pulling back along the morphism    $u_Y: \Spec(\Or_{S,\Y}) \rightarrow S$, which we denote by 
$$\omega^{\Y}_{dR} = u_Y^* \omega : \HH(S) \To \mathrm{Proj}(\Or_{S,Y})$$
and takes values in  the category of projective modules of finite type over $\Or_{S,Y}$.  The fiber functor $\omega^{\Y}_{dR}$ is simply   $\Gamma( \Or_{S,\Y}, \bullet)$.
 We shall mainly consider:
  \begin{enumerate}
\item $Y=\emptyset$. Then $u_Y$ is the generic point of $S$, and $\Or_{S,Y}= K_S$, where $K_S$ is the field of fractions of $S$. Our fiber functor is
$$\omega^{\mathrm{gen}}_{dR}: \HH(S) \To \mathrm{Vec}_{K_S}\ .$$
\item Let $Y = \{s\}$ where  $s\in S(\Q)\subset S(\C)$ is a rational point of $S$. Then  $\Or_{Y,S}= \Or_s$ is the local ring of $S$ at $s$. The fiber functor
$$ \omega^{s}_{dR}  : \HH(S) \To \mathrm{Proj}( \Or_s)  \ ,$$
takes values in  projective (hence free) modules over $\Or_s$ of finite type. 
\item $S = \Spec B$ is affine, and $Y= S(\C)$. Then $\Or_{S,Y} = \Or_S = B$, and the fiber functor
$\omega^{S}_{dR}: \HH(S) \rightarrow  \mathrm{Proj}(B)$ is the global sections functor $\Gamma(S, \omega(\V_{dR}))$. \end{enumerate}

 Denote the corresponding Tannaka groups by $G^{B}_{\HH(S),\X}=\Aut_{\HH(S)}^{\otimes}(\omega^B_{X})$,  and $G^{dR}_{\HH(S),\Y}=\Aut_{\HH(S)}^{\otimes}(\omega^{\Y}_{dR})$. The functor $\omega_{dR}$ gives a morphism  
 \begin{equation} \label{restrict2algdR}
 \pi_1^{dR, \mathrm{alg}} (S, \omega^{\Y}_{dR}) \To G^{dR}_{\HH(S), \Y}
 \end{equation} of affine group schemes over $\Or_{S,\Y}$, where $   \pi_1^{dR, \mathrm{alg}} (S, \omega^{\Y}_{dR}) = \Aut^{\otimes}_{\mathcal{A}(S)}(\omega^{\Y}_{dR}).$ Similarly, 
 the functor  $\omega^{\X}_{B}$ defines  a morphism $ \pi_1^{\mathrm{alg}, B}(S,\X)   \rightarrow  G^{B}_{\HH(S),\X} $    of affine group schemes  over $\Q$, and in particular 
   a   \emph{monodromy homomorphism}: \label{gloss: monodromyhomorphism}  
\begin{equation} \nonumber
  \pi_1^{\mathrm{top}}(S(\C),\X)   \To  G^{B}_{\HH(S),\X}(\Q)\ . 
 \end{equation}

 \subsection{Ring of $\HH(S)$-periods.}  Let $\X , \Y \subset S(\C)$ be as above, and let $\omega^{\X}_{B}$, $\omega^{\Y}_{dR}$ be the corresponding fiber functors on $\HH(S)$.  Define a ring   
 $$\Pe^{\mm,\X,\Y}_{\HH(S)}= \Or(\mathrm{Isom}^{\otimes}_{\HH(S)}(\omega^{\Y}_{dR},  \omega^{\X}_{B})) $$
of  matrix coefficients (denoted $ \Pe^{\omega_{B}^{\X}, \omega^{\Y}_{dR}}_{\HH(S)}$ as in    \S\ref{sectMatrixcoeffs}), where $B_1= \Q, B_2 = \Or_{S,\Y},$ and $k=\Q$.  It is a $\Q \otimes_{\Q} \Or_{S,\Y} $-bimodule, and generated by matrix coefficients
\begin{equation} \label{eqnfamilymatrixcoeff} [(\V_{B}, \V_{dR}, c), \sigma, \omega]^{\mm}
\end{equation}
where $\sigma \in  \omega^{\X}_{B}\V_B^{\vee}$,  $\omega \in \omega^{\Y}_{dR} \V_{dR}$.
 Similarly, define  a ring of `de Rham periods' to be $\Pe^{\dR,\Y}_{\HH(S)} = \Aut^{\otimes}_{\HH(S)}(\omega^{\Y}_{dR})$. It is an  $\Or_{S,\Y} \otimes_{\Q} \Or_{S,\Y}$-bimodule.   It is generated by matrix coefficients $ [(\V_{B}, \V_{dR}, c), v, \omega]^{\mm}$ where $\omega$ is as above and $v\in \omega^{\Y}_{dR} \V_{dR}^{\vee}$.

 These rings  are functorial in the following way. 
Let  $S'$ be a smooth geometrically  connected scheme over $\Q$, and  $f:S \rightarrow S'$ a smooth morphism. Let $\X \subset S'(\C)$ be simply-connected such that $f(\X) \subseteq \X'$, 
 and  let $Y'\subset S'(\C)$ such that $f(\Y)\subseteq \Y'$.   The pull-back  defines a functor  $f^*:\HH(S') \rightarrow \HH(S)$ and hence a
  morphism 
 \begin{equation} \label{Pmfunctorial} 
 f^* :\Pe^{\mm,\X',\Y'}_{\HH(S')} \rightarrow \Pe^{\mm,\X,\Y}_{\HH(S)}  
 \end{equation}
 and a similar map on replacing $\mm$ by $\dR$ and making the obvious changes.

\subsubsection{Constants} Now apply this to $S'= \Spec\, \Q$, $\X'=\Y'=S'(\C) $  (example (3) of the previous section)  and $f: S \rightarrow \Spec(\Q)$ the structural map.  One checks that 
 the category $\HH(S')$ is equivalent to $\HH$, and that the Betti and de Rham fiber functors on $\HH$ and $\HH(S')$ coincide. Therefore 
  $$\Pe^{\mm, pt, pt}_{\HH(\Spec(\Q))} = \Pe^{\mm}_{\HH}\ ,$$
  and we obtain  canonical homomorphisms (\emph{`constant' maps}) \label{gloss: constantmap}
  \begin{equation} \label{constant}
   \Pe^{\mm}_{\HH} \To \Pe^{\mm,\X,\Y}_{\HH(S)}  \quad \hbox{ and  } \quad \Pe^{\dR}_{\HH} \To \Pe^{\dR,\Y}_{\HH(S)}
   \end{equation}
 In this way,  $\HH$-periods can be viewed as  `constant' $\HH(S)$-periods, since the functor $f^*:    \HH \rightarrow \HH(S)$ associates 
 to $(V_{B}, V_{dR}, c)$ a triple $(\V_{B}, \V_{dR}, c)$
 where $\V_{dR}=V_{dR} \otimes_{\Q} \Or_S$  with $\nabla = \id \otimes d$,  and $\V_B$ is the constant local system with fibres $V_B$.  
 
 \subsubsection{Evaluation} Now suppose that  there is a rational point
 \begin{equation}  \label{tcondition} t\in S(\Q) \quad  \hbox{ such that }\quad   t \in X \cap Y  \ .
 \end{equation}
There are    \emph{evaluation  maps} \label{gloss: evaluationmap}    at the point  $t$
 \begin{equation} \label{eqn: specialisation} 
 \ev_t: \Pe^{\mm,\X,\Y}_{\HH(S)} \To \Pe^{\mm}_{\HH} \qquad \hbox{ and  } \quad  \ev_t: \Pe^{\dR,\Y}_{\HH(S)} \To \Pe^{\dR}_{\HH}\ ,  \end{equation}
which are  induced by the functor `fiber at $t$' from $\HH(S) \rightarrow \HH$ via $t: \Or_{S,\Y} \rightarrow \Q$. 
   The constant maps $(\ref{constant})$ are 
  sections of the evaluation maps $(\ref{eqn: specialisation})$.
  Note that one may wish to weaken the condition $(\ref{tcondition})$ if one bears in mind that 
  for $t\notin X$ the evaluation map is not well-defined (it is ambiguous up to the action of monodromy), and for $t \notin Y$ the evaluation may be infinite due to  the presence of poles.

\subsection{Some properties of $\HH(S)$-periods}
  The ring $\Pe^{\mm,\X,\Y}_{\HH(S)}$ has a  left \emph{Galois action} by the group $G^{dR,\Y}_{\HH(S)}$, or equivalently, a right coaction
\begin{equation}   \label{H(S)coaction}
\Delta^{\!\mm}: \Pe^{\mm,\X,\Y}_{\HH(S)} \To\Pe^{\mm,\X,\Y}_{\HH(S)} \otimes_{\Or_{S,\Y}} \Pe^{\dR,\Y}_{\HH(S)} \ , 
\end{equation}
given by the same formula as $(\ref{eqn: coaction})$.   Since the fiber functor $\omega^{\Y}_{dR}$ factors through $\omega_{dR}$, the action of   $G^{dR,\Y}_{\HH(S)}$ restricts to an action of the algebraic de Rham fundamental group $\pi_1^{dR, \mathrm{alg}}(S, \omega^{\Y}_{dR})$ via the map  $(\ref{restrict2algdR})$.  More generally,  there is an action of the 
de Rham algebraic fundamental groupoid:
$$  \pi_1^{dR, \mathrm{alg}}(S, \omega^{\Y_1}_{dR}, \omega^{\Y_2}_{dR}) \times   \Pe^{\mm,\X,\Y_1}_{\HH(S)} \To\Pe^{\mm,\X,\Y_2}_{\HH(S)} $$ 
for any $\Y_1,\Y_2 \subset S(\C)$ as above, where the left-hand side is $\Isom^{\otimes}_{\mathcal{A}(S)}(   \omega^{\Y_1}_{dR}, \omega^{\Y_2}_{dR})$.

The ring $\Pe^{\mm,\X,\Y}_{\HH(S)}$   has an increasing \emph{weight filtration}  \label{gloss: weightfiltration2}  
$W_{\bullet}$ which is inherited from the weight filtration on the category $\HH(S)$ (specifically, $W_n$ is generated by  matrix coefficients $(\ref{eqnfamilymatrixcoeff})$ such that the de Rham class satisfies
$ \omega \in \omega^Y_{dR} W_n \V_{dR}$), and  is preserved by $G^{dR,\Y}_{\HH(S)}$.  The morphisms $(\ref{constant})$ and $(\ref{eqn: specialisation})$ preserve the weight filtrations.  The same comments apply to the rings of de Rham periods  $\Pe^{\dR,\Y}_{\HH(S)}$. This notion, applied in the setting   \cite{Cosmic} gives a rigorous meaning to intuitive notions of `transcendental weight' of functions in the physics literature.

 The ring $\Pe^{\mm,\X,\Y}_{\HH(S)}$ also has a  right action by 
$\pi_1^{\mathrm{alg},B}(S,\X)$, and in  particular,  a right action by 
the topological fundamental group or \emph{monodromy action}: \label{gloss: monodromyaction}
\begin{equation}\label{monodromyonPe} 
 \Pe^{\mm,\X,\Y}_{\HH(S)} \times \pi_1^{\mathrm{top}}(S(\C),\X) \To \Pe^{\mm,\X,\Y}_{\HH(S)}\ . 
 \end{equation}
It commutes with the action of $G^{dR,\Y}_{\HH(S)}$ and also respects $W$.
 The monodromy action can be read off  matrix coefficients   $(\ref{eqnfamilymatrixcoeff})$ by its action on the Betti class $\sigma$, since   the $\Q$-vector space $(\V_B)_{\X}$ naturally carries  a right $\pi_1^{\mathrm{top}}(S(\C),\X) $-action. More generally, for any $\X_1,\X_2 \subset S(\C)$ simply connected we have an action of the topological fundamental groupoid or 
 \emph{continuation along paths}  
 \begin{equation}\label{continuationonPe}
 \Pe^{\mm,\X_1,\Y}_{\HH(S)} \times \pi_1^{\mathrm{top}}(S(\C),\X_1,\X_2) \To \Pe^{\mm,\X_2,\Y}_{\HH(S)}\ . 
 \end{equation} These actions  commute with the action of $G^{dR,\Y}_{\HH(S)}$ and respect the weight filtration, since the latter are defined entirely in terms of the de Rham class. \vspace{0.05in}

The following   structures on $\Pe^{\mm,\X,\Y}_{\HH(S)}$ are \emph{not preserved} by the action of $G^{dR,\Y}_{\HH(S)}$. 
 By the Tannaka theorem $\ref{thm:  Tannaka}$,  $\Pe^{\mm,\X,\Y}_{\HH(S)}$ is the $\omega^{\Y}_{dR}$-image of an ind-object in $\HH(S)$.  Denote its image under $\omega_{dR}$  by $\widetilde{\Pe}^{\mm,\X,\Y}_{\HH(S)}$. It  is an (infinite-dimensional) algebraic vector bundle  on $S$, or ind-object of $\mathcal{A}(S)$,  whose image under $u^*_{\Y}$, the restriction to $\Spec (\Or_{S,\Y})$, is  $\Pe^{\mm,\X,\Y}_{\HH(S)}$. Furthermore, it  is equipped with an increasing weight filtration $W$, decreasing Hodge filtration $F$, and an integrable connection
 $$\nabla : \widetilde{\Pe}^{\mm,\X,\Y}_{\HH(S)} \To \widetilde{\Pe}^{\mm,\X,\Y}_{\HH(S)} \otimes_{\Or_S} \Omega^1_S$$
 which satisfies Griffiths transversality. Restricting to $\Spec (\Or_{S,\Y})$, we deduce the existence of a \emph{Hodge filtration} \label{gloss: Hodgefilt2} 
 $ F^{\bullet}  \Pe^{\mm,\X,\Y}_{\HH(S)}$ and \emph{connection} \label{gloss: connection} 
 \begin{equation} \label{nablaonPe}
 \nabla: \Pe^{\mm,\X,\Y}_{\HH(S)} \To \Pe^{\mm,\X,\Y}_{\HH(S)} \otimes_{\Or_{S,\Y}} \Omega^1_{\Or_{S,Y}}\ ,
 \end{equation} 
 where $\Omega^1_{\Or_{S,Y}}$ is the ring of K\"ahler differentials on $\Or_{S,\Y}$. The connection $(\ref{nablaonPe})$ is integrable, respects $W$, and satisfies Griffiths transversality with respect to  $F$.  On matrix coefficients the connection $(\ref{nablaonPe})$  is given  by  
 $$\nabla [(\V_{B}, \V_{dR}, c), \sigma, \omega]^{\mm} = [(\V_{B}, \V_{dR}, c), \sigma, \nabla \omega]^{\mm}\ , $$ 
 where $\nabla : u_\Y^* (\V_{dR}) \rightarrow  u_\Y^* (\V_{dR}) \otimes_{\Or_{S,\Y}}  \Omega^1_{S,\Y} $ is the connection on $\V_{dR}$ restricted to $u_Y: \Spec(\Or_{S,\Y})\rightarrow S$. The space $F^n  \Pe^{\mm}_{\HH(S),\X,\Y}$ is generated by 
  matrix coefficients
 $ [(\V_{B}, \V_{dR}, c), \sigma, \omega]^{\mm}$ where $\omega \in \omega^Y_{dR}  F^n   \V_{dR} $. 
 There is an analogous connection on the ring of left de Rham periods $\Pe^{dR}_{\HH(S)}$, where the connection acts on   $ [(\V_{B}, \V_{dR}, c), v, \omega]^{\dR}$ through its action on $\omega$. 
     From  formula $(\ref{eqn: coaction})$, namely
 $$   \Delta {[}V, \sigma, \omega]^{\mm}  =  \sum_{i} [V, \sigma, e_i]^{\mm} \otimes  [V, e^{\vee}_i, v]^{\dR}  $$
        one can check that   the following diagram commutes
  $$
\begin{array}{ccc}
 \Pe_{\HH(S)}^{\mm,\X,\Y} & \overset{\Delta}{\To}   & \Pe^{\mm,\X,\Y}_{\HH(S)} \otimes_{\Or_{S,\Y}} \Pe^{\dR}_{\HH(S),\Y}   \\
  \downarrow &   &   \downarrow  \\
  \Pe_{\HH(S)}^{\mm,\X,\Y} \otimes_{\Or_{S,\Y}} \Omega^1_{\Or_{S,\Y}} &  \overset{\Delta\otimes \id}{\To}   &    \Pe_{\HH(S)}^{\mm,\X,\Y} \otimes_{\Or_{S,\Y}}  \Pe^{\dR,\Y}_{\HH(S)} \otimes_{\Or_{S,Y}} \Omega^1_{S,\Y}
\end{array}
$$ 
where the vertical map on the left is $\nabla$, and on the right is $\id \otimes \nabla$.  Hence
\begin{equation} 
(\Delta \otimes \id )\nabla = (\id \otimes \nabla) \Delta \ ,
\end{equation}
  which relates the Galois coaction to the connection. Since the  connection $(\ref{nablaonPe})$ only invokes the de Rham framing, it commutes with the monodromy action $(\ref{monodromyonPe})$. 
  
  \begin{rem}
  The previous remarks give  a proof of two formulae $(5.23)$ conjectured in \cite{Duhr} in the case of the multiple polylogarithms (iterated integrals on the moduli space
  of curves of genus $0$ with $n$ marked points), and generalise to all families of motivic periods (and in particular, motivic Feynman amplitudes).  
  \end{rem} 
  
  Finally, let us suppose that $\X, \Y$ are preserved by a subgroup $A $ of the group of automorphisms of $S$. Then the functoriality $(\ref{Pmfunctorial})$ gives rise to an action
  \begin{equation}  \label{AutosonPm}
   A \times \Pe^{\mm, \X, \Y }_{\HH(S)} \To  \Pe^{\mm, \X, \Y}_{\HH(S)}
   \end{equation}
  of $A$ on $\HH(S)$-periods, and similarly on $\Pe^{\dR}_{\HH(S), \Y}$ under the weaker assumption that only  $\Y$ is stable under $A$. 

 \subsection{The period homomorphism}
 Let $\X,\Y\subset S(\C)$  with $X$ simply-connected. The period is defined by pairing the Betti and de Rham classes of a matrix 
 coefficient $(\ref{eqnfamilymatrixcoeff})$ using the comparison $c$. In order to obtain a multi-valued function on the whole of  $S(\C)$, and not just on $X$,  these classes must be suitably extended as follows.
     Let $\pi: \widetilde{S}(\C)_{\X}\rightarrow S(\C)$ denote the universal covering space of  $S(\C)$ based at $\X$, and let   $M_{\X,\Y}(S(\C))$ denote the ring of meromorphic functions on $\widetilde{S}(\C)_{\X}$ which have no poles on $\pi^{-1}(\Y)$.   By this we mean that for every $f \in M_{\X,\Y}(S(\C))$,  and any  $x \in \widetilde{S}(\C)_{\X}$, there exists a $g$, an element in the fraction field of $S$,  such that 
 $f \times  \pi^{-1}(g) $ is analytic in some open neighbourhood of $x$. If $x\in \pi^{-1}(Y)$ then we can take $g=1$, and $f$ is already analytic in some neighbourhood of $x$.
 Elements of $M_{\X,\Y}(S(\C))$ can be thought of  as multivalued meromorphic functions on $S(\C)$ with a prescribed  branch on the set $\X$, and poles outside $\Y$. 
 
Suppose that $\Y$ is contained in the complex points of some open affine subset of $S$, as earlier. 
The \emph{period map}  \label{gloss: periodmap}  is then a homomorphism
$$\per : \Pe^{\mm,\X,\Y}_{\HH(S)} \To M_{\X,\Y}(S(\C))\ ,$$ 
and is  defined on matrix coefficients $[(\V_B, \V_{dR}, c), \sigma, v]^{\mm}$ as follows.  The local system $\pi^*(\V_B^{\vee})$ is trivial on the simply connected space $\widetilde{S}(\C)_{\X}$ and $\sigma$ extends to a unique global section
 $\sigma \in \Gamma( \widetilde{S}(\C)_{\X},  \pi^*(\V_B^{\vee}))$. Let $ x\in\widetilde{S}(\C)_{\X}$, and let   $N_x$ be a sufficiently small neighbourhood of $x$ such that the restriction of $\pi$ to $N_x$ is an isomorphism.  We obtain a local section 
 $\sigma_x \in \Gamma(\pi(N_x), \V_B^{\vee})$, defined by  $\sigma_x= (\pi|_{N_x}^{-1})^*\sigma$.

  On the other hand, by assumption on $\Y$, there exists an open affine $U\subset S$
  with $\Y \subset U(\C)$ such that $ v\in \Gamma(U, \V_{dR})$. Let  $W\subset S$ be an open affine such that $\pi(x) \in W(\C)$ and the restriction of $\V_{dR}$ to $W$ is trivial as a vector bundle. 
  Since $S$ is irreducible, $U \cap W \neq \emptyset$ and we have $v|_{U\cap W}\in \Gamma(U \cap W, \V_{dR}) = \Gamma(W,\V_{dR}) \otimes_{\Or_W} \Or_{U \cap W}$. It can have poles on $W\backslash U$. We can `clear its denominator', since $\V_{dR}$ is of finite type, there exists an  $\alpha \in \Or_W$ such that $\alpha  v \in \Gamma(W, \V_{dR})$.   By making $N_x$ smaller, we can assume that $W(\C)$ contains $\pi(N_x)$, and we can view $\alpha v$, by restriction and passing to the associated analytic vector bundle, as an element in $\Gamma(\pi(N_x), \V_{dR}^{an})$.  We have defined
  $$\sigma_x \in  \Gamma(\pi(N_x), \V_B^{\vee}) \qquad \hbox{ and } \qquad   \alpha  v \in \Gamma(\pi(N_x), \V_{dR}^{an})\ .$$
   The comparison map $c : \V^{an}_{dR} \rightarrow \V_B \otimes_{\Q} \Or_S^{an}$  yields an element
 $$\sigma_x( c ( \alpha v)) \in \Gamma( \pi(N_x), \Or_S^{an})\ $$
 which can be viewed as a locally analytic function on $N_x$. 
 The period  homomorphism is defined  on the following matrix coefficient by  
 $$ \per( [(\V_B, \V_{dR}, c), \sigma, \alpha v]^{\mm})  = \sigma_x( c (  \alpha v))  \ ,$$
 and the period of $ [(\V_B, \V_{dR}, c), \sigma, v]^{\mm}$ is obtained by dividing by the rational function $\alpha$. 
It is well-defined (does not depend on the representative for the matrix coefficient) because  morphisms in $\HH(S)$ respect the comparison $c$.  Note that it locally has poles along the zeros of $\alpha$. In the case when $\pi(x) \in Y \subset U(\C)$,  we may assume $W=U$ in the above and hence $\alpha=1$, and the period has no poles.  The period homomorphism
therefore takes values in $M_{\X,\Y}(S(\C))$ as claimed.

The period map satisfies the following properties, which follow from the definitions.
First of all, the period is functorial with respect to smooth morphisms, and in particular is compatible with the constant map $(\ref{constant})$. This means that the  following diagram commutes:
$$
\begin{array}{ccc}
 \Pe^{\mm}_{\HH} & \overset{(\ref{constant})}{\To}  &  \Pe^{\mm,\X,\Y}_{\HH(S)}  \\
 \downarrow_{\mathrm{per}}&   & \downarrow_{\mathrm{per}}  \\
  \C &  \subset   & M_{\X,\Y}(S(\C))  
\end{array}
$$
where the inclusion on the bottom line is the inclusion of constant functions. 
 The period is also compatible with monodromy; there is a commutative diagram
$$
\begin{array}{ccc}
 \Pe^{\mm,\X,\Y}_{\HH(S)}    \times \pi_1^{\mathrm{top}}(S(\C), \X) & \To  &  \Pe^{\mm,\X,\Y}_{\HH(S)}  \\
 \downarrow_{\per \times \id}&   & \downarrow_{\mathrm{per}}  \\
  M_{\X,\Y}(S(\C))   \times \pi_1^{\mathrm{top}}(S(\C), \X) &  \To  &     M_{\X,\Y}(S(\C))    
\end{array}
$$
where the action of the topological fundamental group on $M_{\X,\Y}(S(\C))$ is induced by the action of  the group of deck transformations on $\widetilde{S}(\C)_{\X}$. If one thinks of elements of $M_{\X,\Y}(S(\C))$ as multivalued functions on an open subset of $S(\C)$, this is just analytic continuation along loops.  More generally, given two simply-connected subsets $\X_1, \X_2$ of $S(\C)$, we have a compatibility of groupoid actions
$$
\begin{array}{ccc}
 \Pe^{\mm,\X_1,\Y}_{\HH(S)}    \times \pi_1^{\mathrm{top}}(S(\C), \X_1,\X_2) & \To  &  \Pe^{\mm,\X_2,\Y}_{\HH(S)}  \\
 \downarrow_{\per \times \id}&   & \downarrow_{\mathrm{per}}  \\
  M_{\X_1,\Y}(S(\C))   \times \pi_1^{\mathrm{top}}(S(\C), \X_1,\X_2) &  \To  &     M_{\X_2,\Y}(S(\C))    
\end{array}
$$
where the map along the bottom is defined by analytic continuation.

 Now let $x \in \mathrm{Der}_{\Q}(\Or_{S,\Y})$ be a $\Q$-linear derivation of $\Or_{S,\Y}$.
 It defines a map $\Omega^1_{S,Y} \rightarrow \Or_{S,\Y}$. Let $\partial_x = (\id \otimes x) \nabla$, where $\nabla$ is $(\ref{nablaonPe})$.  We have a commutative diagram 
 $$
\begin{array}{ccc}
 \Pe^{\mm,\X,\Y}_{\HH(S)}  &  \overset{\partial_x} \To  &  \Pe^{\mm,\X,\Y}_{\HH(S)}  \\
 \downarrow_{\per }&   & \downarrow_{\mathrm{per}}  \\
  M_{\X,\Y}(S(\C))   &  \To  &     M_{\X,\Y}(S(\C))    
\end{array}
$$
where the map along the bottom is differentiation of locally analytic functions along the  vector field  defined  by $x$ on some open subset of $S(\C)$ containing $Y$. Thus the connexion on the ring of periods corresponds to differentiation of functions.

The period map is functorial. Suppose we are in the situation described in the lines preceeding $(\ref{Pmfunctorial})$. Then there is a commutative diagram
$$
\begin{array}{ccc}
 \Pe^{\mm,\X',\Y'}_{\HH(S')}  &  \overset{f_*} \To  &  \Pe^{\mm,\X,\Y}_{\HH(S)}  \\
 \downarrow_{\per }&   & \downarrow_{\mathrm{per}}  \\
  M_{\X',\Y'}(S'(\C))   &  \overset{f_*}{\To}  &     M_{\X,\Y}(S(\C))    
\end{array}
$$
where the map along the bottom is composition  $\phi \mapsto \phi \circ f$.

 Suppose that $t\in S(\Q)$ is a rational point  and the image of $t$ in $S(\C)$ lies in $X \cap Y$. Then the period map is well-defined at $t$, and  we have a commutative diagram 
 $$
\begin{array}{ccc}
 \Pe^{\mm,\X,\Y}_{\HH(S)} & \overset{\mathrm{ev}_t}{\To}  &  \Pe^{\mm}_{\HH}  \\
 \downarrow_{\per_t}&   & \downarrow_{\per}  \\
  \C &  =  & \C   
\end{array}
$$
where $\per_t$ is evaluation  of elements of $M_{\X,\Y}(S(\C))$ at $t$.
Thus $(\ref{eqn: specialisation})$  corresponds to taking the value of a function at a point. In this manner many classical notions for multivalued functions have 
analogues on the ring of  motivic periods. Nonetheless, the ring of motivic periods has extra features such as the weight and Galois group which are invisible on functions.

 \subsection{Complex conjugation}
 Consider the category $\HH(\overline{S})$ consisting of triples $(\V_{B}, \V_{dR}, \overline{c})$ defined in an identical manner as above, except that 
 $$\overline{c}: \V_{dR} \otimes_{\Or_S} \Or_{\overline{S}^{\an}} \overset{\sim}{\To} \V_B \otimes_{\Q} \Or_{\overline{S}^{\an}} $$
 is antiholomorphic (and respects $W$, etc).  The real Frobenius defines an equivalence $F_{\infty} : \HH(S) \rightarrow \HH(\overline{S})$ which maps $(\V_{B}, \V_{dR}, c)$ to $(\sigma^* \V_B, \V_{dR}, \overline{c})$. We can form a ring of periods 
 $\Pe^{\mm, \X, \Y}_{\HH(\overline{S})}$ as before,  and we have an isomorphism
 $$F_{\infty} : \Pe^{\mm, \X, \Y}_{\HH(S)} \overset{\sim}{\To} \Pe^{\mm, \overline{\X}, \Y}_{\HH(\overline{S})} \ .$$
Composing with the map $\sigma^* : \Pe^{\mm ,\overline{\X}, \Y}_{\HH(\overline{S})}  \rightarrow \Pe^{\mm,\X, \Y}_{\HH(\overline{S})}  $, which sends
 $(\V_B, \V_{dR} , \overline{c} ) $ to $(\sigma^* \V_B, \V_{dR}, \overline{c} \sigma^*)$, gives an isomorphism 
 $ \sigma^* F_{\infty}  : \Pe^{\mm, \X, \Y}_{\HH(S)} \overset{\sim}{\rightarrow} \Pe^{\mm ,\X, \Y}_{\HH(\overline{S})} .$
 The period map on $ \Pe^{\mm, \X, \Y}_{\HH(\overline{S})}$ takes values in the ring
 $\overline{M}_{\X,\Y}(S(\C))$ of antiholomorphic functions on $S(\C)$ with prescribed branch on $\X$.
  The following diagram commutes
 $$
 \begin{array}{ccc}
  F_{\infty}\sigma^*:  \Pe^{\mm, \X, \Y}_{\HH(S)}  & \overset{\sim}{\To}    & \Pe^{\mm, \X, \Y}_{\HH(\overline{S})}    \\
  \downarrow_{\per} &   &   \downarrow_{\per}   \\
   M_{\X,\Y}(S(\C)) &  \overset{\sim}{\To}   &  \overline{M}_{\X,\Y}(S(\C))
\end{array}
$$ 
where the map along the bottom is complex conjugation $f \mapsto \overline{f}$.

\section{Further remarks}\label{sectfamilyfurther}
There are many constructions that   can be made involving   families of periods.  I shall only mention the minimum required for applications to 
\cite{Cosmic}, and single-valued functions and symbols which have  independent applications to physics. 

 \subsection{Some jargon} \label{sectjargon}  Most of the definitions of earlier paragraphs generalise in an evident way to the case of families of periods. I will not repeat all of them here except to mention that 
 an element $\xi \in \Pe^{\mm}_{\HH(S),\X,\Y}$ generates a representation of $G^{dR}_{\HH(S),\Y}$ which, by the Tannaka theorem,  defines a (minimal) object $M(\xi)$ of $\HH(S)$ (see \S\ref{sectminimalobject}).  The \emph{Hodge numbers} and \emph{Hodge polynomial} of $\xi$ 
are defined as the Hodge numbers and polynomial of the fiber of $M(\xi)$ at any point $t\in \X$. Define the \emph{local system associated to $\xi$} to be $M(\xi)_B$. It is equivalent to the \emph{monodromy representation} of $\xi$ which is the vector space
$\omega_{B,\X} (M(\xi)) \in \mathrm{Vec}_{\Q}$
together with its left $\pi_1(S(\C), \X)$-action. Likewise, define the \emph{vector bundle with connexion} associated to $M(\xi)$ to be $M(\xi)_{dR}$, equipped with its integrable connexion $\nabla$. The vector bundle and local system associated to $\xi$ are equivalent under the Riemann-Hilbert correspondence after tensoring with $\C$.

 \subsection{Variant: families of periods with single-valued branch}
For applications such as \cite{Cosmic}, we are often faced with the situation where we have a family of periods on some open subvariety $S$ of a given variety $Z$, without knowing   what $S$ is. This is in fact the typical situation; $S$ will be the complement of a  discriminant locus which is complicated and   not computable. We often have some further information, namely that the periods are well-defined on some connected open subset $U\cap S(\C) \subset S(\C)$ for the analytic topology.  This motivates the following 
definition.

Let $Z$ be a smooth, geometrically connected algebraic variety over $\Q$, and let $U \subset Z(\C)$ be a connected open analytic subset.
 For any geometrically connected $S\subset Z$,  consider the ring of periods 
 $\Pe^{\mm}_{\HH(S), s, \emptyset}$
 for any $s\in S(\C)\cap U$. For any two points $s_1,s_2 \in U \cap S(\C)$, analytic continuation along paths gives 
 $$\Pe^{\mm}_{\HH(S), s_1, \emptyset} \times \pi_1^{\mathrm{top}}(U \cap S(\C), s_1, s_2) \overset{\sim}{\To} \Pe^{\mm}_{\HH(S), s_2, \emptyset}$$
 and in particular, since $U\cap S(\C)$ is connected, a canonical isomorphism
 $$ \big(\Pe^{\mm}_{\HH(S), s_1, \emptyset} \big)^{\pi_1(U\cap S(\C), s_1)} = \big(\Pe^{\mm}_{\HH(S), s_2, \emptyset} \big)^{\pi_1(U\cap S(\C), s_2)}\ . $$
  Thus, by moving the base-point $s$ if necessary, we can define
  $$\Pe^{\mm,U}_{\HH(Z)} = \varinjlim_{S} \big(\Pe^{\mm}_{\HH(S), s, \emptyset}\big)^{\pi_1(U\cap S(\C), s)} 
$$ 
 where the limit is over all such $S\subset Z$, since any two such opens $S_1 , S_2 \subset Z$ have a non-empty intersection $U\cap S_1 \cap S_2(\C)$.
  The periods of elements of $\Pe^{\mm,U}_{\HH(Z)}$ restrict to  single-valued meromorphic functions  on $U$. 
 The ring $\Pe^{\mm,U}_{\HH(Z)}$ has similar properties to the rings of periods discussed earlier, except for the  monodromy action.

\subsection{Single-valued versions}
We can define single-valued versions of families of motivic periods in a similar way to \S\ref{sectSVconstant}, except that there are some slight differences.
Let $M= (\V_{B}, \V_{dR}, c)$ by an object of $\HH(S)$.  As in \S\ref{sectSVconstant}, there is a universal comparison homomorphism
$$c_M^{\mm}\ :\  \omega_Y(\V_{dR}) \To \omega_X(\V_B) \otimes_{\Q} \Pe^{\mm}_{\HH(S), \X, \Y}\ .$$
Applying $F_{\infty}\sigma^*$ to the right-hand factor gives a homomorphism
$$\overline{c_M}^{\mm}:= F_{\infty} \sigma^* c_M^{\mm}\ :\  \omega_Y(\V_{dR}) \To \omega_X(\V_B) \otimes_{\Q} \Pe^{\mm}_{\HH(\overline{S}), \X, \Y}\ .$$
Define a ring 
$$ \Pe = \Pe^{\mm}_{\HH(\overline{S}), \X, \Y}  \otimes_{\Q} \Pe^{\mm}_{\HH(S), \X, \Y}\  .$$
Embed $\Pe^{\mm}_{\HH(S), \X, \Y}$ into $\Pe$ via $x \mapsto 1 \otimes x$, and $\Pe^{\mm}_{\HH(\overline{S}), \X, \Y}$ via $x\mapsto x\otimes 1$.
Thus
$$c_M^{\mm} \quad ,\quad  \overline{c_M}^{\mm} \quad \in \quad \Hom( \omega_Y(\V_{dR}), \omega_X(\V_B)) \otimes_{\Q} \Pe$$
Finally, define 
$$\s^{\mm}_M = (\overline{c_M}^{\mm})^{-1} c_M^{\mm}   \quad \in  \quad \mathrm{End}( \omega_Y(\V_{dR})) \otimes_{\Q} \Pe \ .$$ We leave it to the reader to  replace the above argument  with universal arguments on torsors, exactly as in \S\ref{sectSVconstant}, to define a canonical element
$$ \s^{\mm} \in G^{dR}_{\HH(S), \Y} (\Pe) $$
or equivalently, a  homomorphism 
\begin{equation} \label{singlevaluedhomforfunctions} 
 \s^{\mm}: \Pe^{\dR}_{\HH(S), \Y} \To \Pe^{\mm}_{\HH(\overline{S}), \X, \Y}  \otimes_{\Q} \Pe^{\mm}_{\HH(S), \X, \Y}\ ,
 \end{equation} 
where the left (resp. inverse right) action of $G^{dR}_{\HH(S), \Y} $ on $ \Pe^{\dR}_{\HH(S), \Y}$ corresponds to the  left  action of $G^{dR}_{\HH(S), \Y} $ on $\Pe^{\mm}_{\HH(S), \X, \Y}$ (resp. $ \Pe^{\mm}_{\HH(\overline{S}), \X, \Y}$).
It follows from the definition that the single-valued homomorphism 
$(\ref{singlevaluedhomforfunctions})$ is compatible with the connexion:
\begin{equation}
(\s^{\mm} \otimes \id )\nabla ( \xi) = ( \id  \otimes \nabla) \s^{\mm}(\xi) \ .
\end{equation} 
This means that, after taking the period homomorphism, the single-valued map respects the \emph{holomorphic} (and only the holomorphic) differential.

\begin{rem} When $S= \Spec(\Q)$ is a point, 
$\Pe^{\mm}_{\HH(\overline{S}), \X, \Y} \cong   \Pe^{\mm}_{\HH(S), \X, \Y}  \cong \Pe^{\mm}_{\HH}$
and the earlier definition  $(\ref{single-valuedmapdefn})$ is obtained as the composition 
$$  \Pe^{\dR}_{\HH} \overset{\s^{\mm}} \To   \Pe^{\mm}_{\HH}  \otimes_{\Q} \Pe^{\mm}_{\HH} \To \Pe^{\mm}_{\HH}  $$
where the second map is multiplication.
\end{rem} 

\begin{ex} \label{exdilog} Consider the dilogarithm motivic period on $S = \Pro^1 \backslash \{0,1,\infty\}$, defined in $(\ref{defnmultiplepolylog})$ below. Here $X\subset S(\C)$ is the open interval $(0,1)$.   Its universal period matrix over a point $x\in X\subset S(\C)$ is 
$$c^{\mm}_M = 
\left(
\begin{array}{ccc}
 1  & \Li_1^{\mm}(x)   &   \Li_2^{\mm}(x)  \\
 0 & \Lef^{\mm}      &    \Lef^{\mm} \log^{\mm}(x)      \\
  0 & 0    &    (\Lef^{\mm})^2
\end{array}
\right) \ .$$
Let $\gamma_0$ (resp. $\gamma_1$) denote a small path around $0$ (resp. $1$)  based at $X$.  Under the monodromy homomorphism they act by  left   multiplication by 
$$ \rho(\gamma_0) = \left(
\begin{array}{ccc}
 1  & 1   &  0  \\
 0 & 1      &    0      \\
  0 & 0    &   1
\end{array}
\right) \quad \hbox{ and }  \quad \rho(\gamma_1) = \left(
\begin{array}{ccc}
 1  & 0   &  0  \\
 0 & 1      &    1      \\
  0 & 0    &   1
\end{array}
\right)\ .$$
Denoting the image of $\Li_k^{\mm}(x)$ under $F_{\infty} \sigma_*$ by $\overline{\Li}_k^{\mm}(x)$ and similarly for $\log^{\mm}(x)$, the matrix  $\overline{c_M}^{\mm}$ is  given by 
$$\overline{c_M}^{\mm}  = 
\left(
\begin{array}{ccc}
 1  & \overline{\Li_1^{\mm}}(x)   &   \overline{\Li_2^{\mm}}(x)  \\
 0 & -\Lef^{\mm}      &    -  \Lef^{\mm} \overline{\log^{\mm}}(x)      \\
  0 & 0    &    (\Lef^{\mm})^2
\end{array}
\right) \ .$$
By computing $(\overline{c_M}^{\mm} )^{-1} c_M^{\mm} $ we find that $\s^{\mm}(\Lef^{\dR}) = (-1)$ (yet again),
$$ \s^{\mm}(\log^{\dR}(x)) = \log^{\mm}(x) +  \overline{\log^{\mm}}(x)\ ,$$
and similarly   for $\Li_1^{\dR}(x) = -\log^{\dR}(1-x)$. The top-right corner gives
$$  \s^{\mm}(\Li_2^{\dR}(x)) = \Li_2^{\mm}(x) -    \overline{\Li_2^{\mm}}(x)  + ( \log^{\mm}(x) +  \overline{\log^{\mm}}(x))\overline{\Li_1^{\mm}}(x)  \quad \in \Pe $$
whose period is $2 i$ times the Bloch-Wigner dilogarithm. Equivalent calculations for the associated period matrices were first carried out  in \cite{Be-De} for the classical polylogarithms. 
For multiple polylogarithms, the computations are made much simpler using the language of  non-commutative formal power series \cite{SVMP}. 
\end{ex}

\section{Symbols} \label{sectSymbols} 
We briefly indicate how a certain class of motivic periods give rise to invariants involving differential forms.
There are several possible variations on this theme which all specialise (in the mixed Tate case) to the  notion of `symbol of differential forms' as currently
used in the physics literature.

\subsection{Some pitfalls} \label{sectpitfall} The symbol is commonly understood by physicists to be  a tensor product of differential forms obtained by differentiating a family of period  integrals with respect to a parameter.  Consider the following examples:
\begin{enumerate}
\item The classical polylogarithm $\Li_k(x) = \sum_{n\geq 1} {x^n \over n^k}$, where $k\geq 1$,   satisfies the following differential  equation for all $k\geq 2$:
$$ d \, \Li_k(x) =   \Li_{k-1}(x)  \, {dx \over x}\ .$$
 Furthermore $d \Li_1(x) = {dx \over 1-x}$. The recipe in the physics literature for constructing the symbol is recursive by repeated differentiation. For example,  $\mathrm{symbol}(\Li_1(x)) = {dx \over 1-x}$ and $\mathrm{symbol} ( \Li_k(x)) = \mathrm{symbol} (\Li_{k-1}(x) )\otimes {dx \over x}$, e.g.,
 $$\mathrm{symbol}( \Li_2(x)) =  {dx \over 1-x} \otimes {dx \over x}\ .$$
 Occasionally, the arguments $d\log f$ in the right-hand side are represented by their arguments $f$, and the symbol of $\Li_2(x)$ is written $1-x \otimes x$.  This notion is already very useful for capturing functional relations between polylogarithms and is ubiquitous in the literature.  Note that the symbol captures the information that $\Li_2(x)$ is an iterated integral of two one-forms on the punctured projective line, but some information is lost: one requires a path of integration to reconstruct the function $\Li_2(x)$ from its symbol. 
  The symbol is sometimes used to infer notions such as `transcendental weight' and monodromy data about functions.  
 
\item Simply differentiating  functions is too naive, as  one sees by considering examples such as $ f(x) \Li_2(x) + g(x) \Li_1(x) \log(x)$,  where $f, g$ are rational functions.  Repeated differentiation leads to an infinite sequence of more complicated functions.   To make sense of the definition $(1)$, one must find differential operators whose application successively decreases some quantity called the `length' (in the examples in $(1)$, this happens to coincide with   the `transcendental weight', although this will not necessarily be true in more general settings). This will  depend on choices, and shows that the symbol should in fact be defined as a tensor product \emph{modulo certain equivalence relations}. These are described in \S \ref{sectReducedBar}. 
 This recursive structure will be  encoded by the  notion of a \emph{unipotent  connection}.

  \item  The symbol  of a constant family of periods is necessarily zero since it is defined in terms of differentiation.   However,  the notion of symbol  in the physics literature has somehow morphed into a version in which the arguments are formally allowed to be  constants.  This is the basis for the definition of `motivic amplitudes' in \cite{Cluster}. 
One finds  equations such as 
  $$\mathrm{symbol} ( \Li_2(3)  )  =   -2 \otimes 3  , $$
  which are obtained by specialising the earlier example to $x=3$, and  
  possibly based on  Goncharov's notion of symbol (remark \ref{remGonchsymbol}). This notion does not go very far to capture constants, since  most are zero under this map:   for example, the Goncharov symbol of the quantity $\zeta(3) \Li_1(x)$ is zero.

  \item A further problem with the   formalism mentioned in $(3)$, which is only defined in terms of de Rham classes and with no mention of Betti classes, is that the de Rham analogue of $\zeta(2)$, or $2 \pi i$,  is irretrievably zero, and so attempting to take the period leads to contradictions. This problem is fixed by replacing `coproduct' with `coaction', which is part of the structure of our definition of families of $\HH$-periods. 
  
   \item There have been attempts  to incorporate  constants into a common `symbolic' framework such as  \cite{Duhr}, which contains both  the coaction for motivic multiple zeta values, and the symbols of polylogarithms. The main properties of this framework were mostly conjectural and shown to be   equivalent to complicated combinatorial identities. 
   We shall show how these problems can be  easily overcome in our setting
   using  the notion of `symbol based at a point $t$', which has  all the properties   conjectured in  \cite{Duhr}.

     \end{enumerate}

Our notion of families of motivic periods fixes these problems and    subsumes  all  of the above notions. It also generalises these concepts to situations which are non-polylogarithmic. More precisely,  the  notions of symbol presently found in the physics literature are derived from the  coaction on families of motivic periods  and discarding more or less information. One must bear in mind, however, that one side of the motivic coproduct involves de Rham periods, not $\HH$-periods, and the former do not have canonical periods.  They do, however,  possess single-valued periods. 

\subsection{Abstract definition of the symbol} 
The symbol will be defined for a certain class of motivic periods.
Recall that an algebraic vector bundle with connection $(\V_{dR}, \nabla)$  on $S$ is \emph{unipotent}  of length $n$ if there exists a filtration 
 \begin{equation} \label{Vfilt} 
 0=\V_{-1} \subset \V_0 \subset \ldots \subset \V_n=\V
 \end{equation}
by algebraic sub-bundles, with   $\nabla: \V_k \rightarrow \V_k \otimes_{\Or_{S}} \Omega^1_{S}$, such that each graded quotient
$\V_k/ \V_{k-1}$ is isomorphic to a direct sum of trivial vector bundles $(\Or_{S}, d)$.  
Such a vector bundle automatically has regular singularities at infinity. The category of unipotent algebraic vector bundles with connection forms a full Tannakian subcategory $\mathcal{A}^{un}(S)$ of $\mathcal{A}(S)$.

A local system $V$ on $S$ is \emph{unipotent} if it admits a finite increasing filtration 
 \begin{equation} \label{Vfilt} 
 0=V_{-1} \subset V_0 \subset \ldots \subset V_n=V
 \end{equation}
of local sub-systems, such that each graded quotient $V_k/V_{k-1}$ is trivial, i.e., constant. Equivalently, for any basepoint $s \in S(\C)$, the  representation $\omega_s(V)$ of $\pi_1(S(\C), s)$ associated to $V$ is unipotent, i.e., admits a finite increasing filtration such that the associated graded quotients are trivial representations. The category of unipotent local systems forms a full Tannakian subcategory $\mathcal{L}^{un}(S)$ of $\mathcal{L}(S)$. 
Since the category of unipotent vector bundles with connection (resp. local systems) behaves well with respect to extensions of scalars, the Riemann-Hilbert correspondence  induces an equivalence
$$\mathcal{A}^{un}(S) \otimes \C  \sim  \mathcal{L}^{un}(S) \otimes \C\ .$$

\begin{defn}   A family of motivic periods $\xi \in \Pe^{\mm,\X,\Y}_{\HH(S)}$ is \emph{differentially unipotent}, or has \emph{unipotent monodromy},  \label{gloss: diffunip} \label{gloss: unipmonod} if the associated vector bundle with connection (\S\ref{sectjargon}) $M(\xi)_{dR}$  is unipotent. 
Equivalently, $\xi$ is differentially unipotent if the associated  local system on $M(\xi)_B$  is unipotent. 
 This means that $\xi$ can be represented as a matrix coefficient $(\ref{eqnfamilymatrixcoeff})$ where  one of  $\V_B, \V_{dR}$ (and hence both)  are unipotent. 
 
 We define a family of de Rham motivic periods $\xi \in \Pe^{\dR,\Y}_{\HH(S)}$ to be \emph{differentially unipotent} in an identical manner. It follows immediately from the definition that the motivic coaction $(\ref{H(S)coaction})$ preserves the quality of  unipotency. 
 \end{defn}

Note that the filtration involved in the definition of differential unipotency will not in general be the weight filtration.

The above categories give rise to Tannaka group schemes over $\Q$:
\begin{eqnarray} \pi_1^{un} (S(\C) ,s )  &= &  \mathrm{Aut}^{\otimes}_{\Lo^{un}(S)}(\omega_s) \qquad \hbox{ for } s \in S(\C) \ ,  \nonumber \\ 
 \pi_1^{dR} (S ,s ) &  = &  \mathrm{Aut}^{\otimes}_{\Ao^{un}(S)}(\omega_s)  \qquad \hbox{ for } s \in S(\Q) \ .   \nonumber 
 \end{eqnarray}
The first group is the unipotent or Mal\v{c}ev completion of the topological fundamental group, the second is the de Rham fundamental group. 
They are naturally quotients of the Betti algebraic and de Rham algebraic fundamental groups. 

A family of motivic or de Rham periods is differentially  unipotent if and only if the  natural  actions by the 
Betti algebraic fundamental group, de Rham algebraic fundamental group, or topological fundamental group factor through their unipotent quotients. 
It follows that for a differentially  unipotent de Rham period, its image under the dual of  the natural map $(\ref{restrict2algdR})$:
\begin{eqnarray} 
 \Pe^{\dR,\Y}_{\HH(S)} = \Or( G^{dR}_{\HH(S), \Y})  & \To &  \Or(  \pi_1^{dR, \mathrm{alg}} (S, \omega^{\Y}_{dR}) )\ , \nonumber \\
{ [}(\V_B, \V_{dR}, c), v, \omega]^{\dR} & \mapsto&   (\alpha \mapsto v(\alpha(\omega))) \nonumber \ ,
\end{eqnarray}  
where $\alpha \in  \pi_1^{dR, \mathrm{alg}} (S, \omega^{\Y}_{dR})$, 
actually lands in the unipotent subspace 
$$ \Or(  \pi_1^{dR} (S, \omega^{\Y}_{dR} )) \quad  \subset  \quad \Or(  \pi_1^{dR, \mathrm{alg}} (S, \omega^{\Y}_{dR})) \ . $$ 
Next restrict to the affine open  $\Spec(\Or_{S,\Y}) $ of $S$, which gives a map on affine rings
$$ \Or(  \pi_1^{dR} (S, \omega^{\Y}_{dR})) \To \Or(  \pi_1^{dR} (\Spec \Or_{S,\Y}, \omega^{\Y}_{dR}))\ .$$
Finally it follows either from Chen's $\pi_1$-de Rham theorem (over the complex numbers), or from the universal properties of the reduced bar construction $\BB$ (whose zero'th cohomology is the universal unipotent extension of $\Or_{S,\Y}$),  that the affine ring of the de Rham fundamental group on $\Spec \Or_{S,\Y}$  is  described  explicitly:
$$ \Or(  \pi_1^{dR} (\Spec \Or_{S,\Y}, \omega^{\Y}_{dR})) \cong H^0 (\BB(\Omega_{S,\Y})) \ . $$
We shall recall the relevant definitions below.

 \begin{defn} We shall define the \emph{symbol} \label{gloss: symbol}   of  a differentially   unipotent family of  de Rham periods $\xi \in \Pe^{\dR, \Y}_{\HH(S)}$ to be its image: \begin{equation} \label{symboldefinition} 
 \smb (\xi)   \in H^0(\BB(\Omega_{S,\Y})) \ .
  \end{equation}
  There is a natural  \emph{length filtration} on $H^0(\BB(\Omega_{S,\Y}))$. The  \emph{length}  \label{gloss: length}  of $\xi$ is bounded above by the length $n$ in  the filtration $(\ref{Vfilt})$.
  \end{defn} 
 This a  generalisation of the notion of  symbol as used by physicists.
Our abstract  definition of the symbol could be generalised further still. For example, one could also consider relative  unipotent completion instead of  unipotent completion. 

\subsection{Computing the symbol} We now explain how to compute the symbol in the spirit of the  recursive differentiation procedure described in  \S\ref{sectpitfall} $(1)$. 
Consider a   differentially  unipotent  de Rham period   $\xi = [(\V_B, \V_{dR}, c),  f, \omega]^{\dR} $ in $  \Pe^{\dR,\Y}_{\HH(S)}$,
 where $f \in \omega_{\Y}(\V_{dR})^{\vee}$ and $ \omega \in \omega_{\Y}(\V_{dR})$, and let $\V  $ denote the pull-back of $\V_{dR}$ to $\Spec \Or_{S,\Y}$.  We shall  assume  that  $H_{dR}^0 (\Or_{S,\Y})=\Q$.
 Since $\xi$ is  differentially  unipotent, we can assume by  equivalence of matrix coefficients that $\V_{dR}$ and hence $\V$ is unipotent.   There  is a filtration 
 \begin{equation} \label{Vfilt2} 
 0=\V_{-1} \subset \V_0 \subset \ldots \subset \V_n=\V
 \end{equation}
by algebraic sub-bundles  such that  $\nabla: \V_k \rightarrow \V_k \otimes_{\Or_{S,\Y}} \Omega^1_{S,\Y}$, and with respect to which $\nabla$ is unipotent. It splits because $H^1(\Spec \Or_{S,\Y}, \Or_{S,\Y}) =0 $ since $\Spec \Or_{S,\Y}$ is affine and $\Or_{S,\Y}$ coherent. 
  Choose a splitting of this  filtration: 
  \begin{equation} \label{splitchoose} 
  \V \cong \gr\,  \V\ .
  \end{equation}
  The associated graded $(\gr\,  \V , \gr\, \nabla)$ is  a direct sum of trivial vector bundles  $(\Or_{S,\Y}, d)$, and 
  $\gr\, \V \cong \Or_{S,\Y}\otimes_{\Q} H^0(\gr\, \V, \gr \, \nabla)$ has a  $\Q$-structure given by flat sections. 
   With respect to this choice of splitting, the connection map
  $$N= \nabla - d \quad \in \quad \mathrm{Hom}_{\Or_{S,\Y}}(     \V, \V \otimes_{\Q} \Omega^1_{S,\Y})$$
  satisfies the integrability condition $$d N +  N \wedge  N=0\ ,$$  which is equivalent to $\nabla^2=0$,  and  it also satisfies $N^m=0$ for some $m$ by the assumption of unipotency. It is an $\Or_{S, \Y}$-linear operator,
  where $\V$ and $\V \otimes_{\Q} \Omega^1_{S,\Y}$ are left $\Or_{S,\Y}$-modules. 
  For computations, let us write $V_k = \Gamma( \Or_{S,\Y}, \V_k)$. Then $V_k/V_{k-1}$ is a free $\Or_{S,\Y}$-module and it follows by induction
  that the $V_k$ are free $\Or_{S,\Y}$-modules too. We can choose a basis of $V_n$ which is adapted to the filtration $V_k$, and  is flat on the graded quotients $\gr\, V$, i.e., a basis element $e$ of $V_k$ satisfies $\nabla(e) \subset V_{k-1}$.  Write the connection as $\nabla = d+ N$ in this basis. Thus  $N$ is represented as an  upper-triangular matrix  of one-forms. 
    Consider the element
  $$\smb_N (\xi) = \sum_{k \geq 0 } \langle f, N^k \omega\rangle   \quad \in \quad \Or_{S,\Y}\otimes_{\Q}  T^c(\Omega^1_{S,\Y}) \ ,$$
  where we recall that the tensor coalgebra $T^c(\Omega^1_{S,\Y}) = \bigoplus_{k\geq 0} ( \Omega^1_{S,\Y})^{\otimes k}$, and tensors are over $\Q$. 
  It depends on the choice of splitting $(\ref{splitchoose})$ (resp.  choice of basis of $V$).

One can think of the symbol in the following way. View $N$ as an $n\times n$ matrix with coefficients in $T^c(\Omega^1_{S,\U})$, since its entries can be considered as tensors of length one.   Consider the following matrix 
\begin{equation}  \label{seriesinN} 1 + N+ N^2 + N^3 +  \ldots \qquad \in M_{n\times n}( T^c(\Omega^1_{S,\T}))\ , \end{equation} 
where the multiplication of matrix entries is given by the (non-commutative) concatenation product in $T^c(\Omega^1_{S,\U})$.
 The series  is finite by the nilpotence of $N$. The vector $\omega$ and covector $f$ define an entry of this matrix, which is exactly $\smb_N(\xi)$.

  \subsubsection{Reduced bar construction}\label{sectReducedBar}
   Define an internal differential 
   $$d_I :  T^c(\Omega^{\bullet} _{S,\Y})   \rightarrow   T^c(\Omega^{\bullet}_{S,\Y})    $$
   by
\begin{multline} d_I {[}\omega_1| \ldots | \omega_n ]  =  \sum_{i=1}^n (-1)^i [ j \omega_1| \ldots |j\omega_{i-1} | d\omega_i | \omega_{i+1} |  \ldots | \omega_n]  \nonumber \\  +  \sum_{i=1}^{n-1} (-1)^{i+1} [ j \omega_1| \ldots | j \omega_{i-1} | j\omega_i \wedge \omega_{i+1} | \omega_{i+2}| \ldots | \omega_n] \ . \nonumber 
\end{multline}
where $j$ acts on $\Omega^n_{S,\Y}$ by $(-1)^n$.  Define a  grading on $T^c(\Omega^{\bullet}_{S,\Y})$ by
$$ \deg \, {[}\omega_1| \ldots | \omega_n ]  = \sum_{i=1}^n \deg(\omega_i)-1 \ .$$ Consider the linear map $d_C: T^c(\Omega^{\bullet} _{S,\Y})   \rightarrow  \Or_S \otimes_{\Q} T^c(\Omega^{\bullet}_{S,\Y})$ defined by
$$d_C   {[} \omega_1 | \ldots | \omega_n ] =  - \varepsilon(\omega_1) [ \omega_2| \ldots | \omega_n] +  (-1)^{\nu} \varepsilon(\omega_n) [\omega_1 | \ldots | \omega_{n-1}] $$
where $\varepsilon: \Omega^{\bullet}_{S,\Y} \rightarrow \Or_{S,\Y}$ is projection onto degree $0$ and   $\nu$ is given by  $(\deg(\omega_n)-1)\deg [\omega_1|\ldots |\omega_{n-1}]$. One verifies that 
$$d=\id \otimes d_I+\id \otimes d_C$$ satisfies $d^2=0$. See, for example, the presentation in \cite{HainZucker} (3.4).  
Note that the signs simplify drastically when all $\omega_i$ are of degree one, which is the case we are mainly interested in.
Consider the smallest subspace $\mathcal{R}$ in  $\Or_{S}\otimes_{\Q} T^c(\Omega_{S,\Y})$ generated by   
$$\mathcal{R} : \qquad \qquad [\omega_1 | \ldots |\omega_i | f | \omega_{i+1} | \ldots | \omega_n]$$
where $f \in \Or_{S,\Y}$, and stable under the differential $d$.   The quotient  of $\Or_{S,\Y}\otimes_{\Q} T^c(\Omega_{S,\Y})$ by $\mathcal{R}$ is a complex which we denote by 
 $\BB(\Omega^{\bullet}_{S,\Y})$. 
    It is a  close relative of Chen's reduced circular bar complex\footnote{In the usual formulation, due to Chen \cite{Chen}, one considers the tensor algebra $T^c(\Omega^{\geq 1}_{S,\Y})$ and quotients out by      a certain family of relations. The reader may like to check that Chen's relations are boundaries in the complex we have defined here and are therefore incorporated automatically.} 
on $\Omega_{S,\Y}$.

\begin{example}  \label{exofrelation} If $f\in \Or_{S,\Y}$ and $\omega \in \Omega^1_{S,\Y}$ is closed we have:
\begin{eqnarray}  d [ f | \omega] & = & - [df | \omega]  + [f\omega] - f[\omega]  \nonumber \\ 
d [ \omega | f] & =&   - [\omega | df] - [f\omega] + f [\omega] \ . \nonumber
\end{eqnarray} 
The expressions on the right-hand side are therefore in $\mathcal{R}$. 
\end{example}

\subsubsection{The symbol}
The integrability of $N$  implies   that the element 
  $\smb_N(\xi)$ is  \emph{integrable}, i.e., is in the kernel of  the differential $d$  in  $\BB(\Omega^{\bullet}_{S,\Y})$. Furthermore, it is of degree zero.
Its cohomology class
 $$
 \smb (\xi)   \in H^0(\BB(\Omega_{S,\Y})) \ , 
 $$
  is exactly the symbol as defined in $(\ref{symboldefinition})$. It is true, but  not  obvious from the computational procedure defined above,  that the cohomology   class of $\smb(\xi)$ is well-defined, i.e., independent of the choice of $N$. 
The integrability of $\smb(\xi)$ can be seen directly by writing the matrix $(\ref{seriesinN})$ in the form 
  $$[1] + [N] + [N| N] + \ldots \   . $$
The equation $dN + N\wedge N =0$ implies that it lies in  the kernel of the differential $d$ applied formally to the previous expression in the obvious way.

\begin{rem}
The bar complex is equipped with a shuffle product (with signs) which is compatible with $d$, and this induces a commutative algebra structure on its cohomology. It follows from the definition $(\ref{symboldefinition})$ that   the symbol is  a homomorphism:
$$\smb(\xi_1 \xi_2) = \smb(\xi_1) \sha \smb(\xi_2)\  .$$
The shuffle product restricted to  $H^0(\BB(\Omega_{S,\Y}))$ has no signs.
\end{rem}

\begin{example}  \label{example symbol} Suppose that $\gr\,  \V$ has length two ($n=2$), so 
$$0=\V_{-1}  \subset \V_0 \subset \V_1 \subset \V_2 = \V$$  and that each graded quotient $\V_k / \V_{k-1}$ for $k=0,1,2$ is a rank one $\Or_{S,\Y}$-module. Choose a basis
 $e_0,e_1,e_2$ of $V = \Gamma(\Or_{S,\Y}, \V)$ such that $\nabla e_0=0$, $\nabla e_1 =   e_0 \otimes \omega_1$ and $\nabla e_2 = e_1 \otimes  \omega_2 + e_0  \otimes \omega_{12}$. 
In this basis,  the matrix $N$ is 
$$N= 
\left(
\begin{array}{ccc}
0  &    \omega_1 & \omega_{12}   \\
 0 & 0  &  \omega_2  \\
  0& 0  &0   
\end{array}
\right)\ .
$$ 
For this to define a connection (or equivalently, $dN + N\wedge N=0$), we must assume that
 $\omega_1,\omega_2$ are closed, and $d\omega_{12} + \omega_1 \wedge \omega_2 =0$. 
If $\xi = [\V, e_0^{\vee}, e_2]^{\dR}$, then 
\begin{eqnarray} \smb_N(\xi)  & =  & \langle e_0^{\vee}, e_2\rangle + \langle e_0^{\vee}, N e_2 \rangle + \langle e_0^{\vee}, N^2 e_2\rangle \nonumber \\
& = & \quad  0 \quad +  \qquad  \omega_{12}   \quad  +   \langle e_0^{\vee}, N e_1 \rangle \otimes   \omega_2 +   \langle e_0^{\vee}, N e_0 \rangle   \otimes \omega_{12} \nonumber \\
& = &  0 +  \omega_{12}  + \omega_1 \otimes \omega_2 + 0 \ .\nonumber 
\end{eqnarray} 
In bar notation, this is denoted by: 
$$\smb_N( \xi )= [\omega_1| \omega_2] + [\omega_{12}]\ .$$
 Now change basis to $e_0',e_1',e_2'$ where 
$e_0'=e_0$, $e_1'= e_1$ and $e_2'= e_2+f e_1 $,  
where $f \in \Or_{S,\Y}$. In the new basis, the matrix $N$ is replaced by 
$$N'= 
\left(
\begin{array}{ccc}
0  &    \omega_1 & \omega_{12}  +f\omega_1 \\
 0 & 0  &  \omega_2 +df   \\
  0& 0  &0   
\end{array}
\right)\ ,
$$ 
and  since $e_2 = e_2' - f e_1$ one checks  by a similar computation that
$$\smb_{N'}(\xi) = [\omega_1 | \omega_2 + df] +[\omega_{12} + f\omega_1] - f[ \omega_1] \ .$$
The difference between the two elements
$$\smb_{N'}(\xi) - \smb_N(\xi) =   [\omega_1| df] + [f \omega_1 ] - f [\omega_1] $$
which is exactly a boundary $- d ([\omega_1 | f])$, by example \ref{exofrelation}. 
\end{example} 

\subsubsection{Variants}
The above construction  was defined for de Rham periods but can be embellished  in any number of ways.  
Let $\xi \in \Pe^{\mm}_{\HH(S), \X,\Y}$ be differentially unipotent. Then using the coaction $(\ref{H(S)coaction})$ we can define
$$(\id \otimes \smb)  \Delta(\xi)  \quad  \in \quad    \Pe^{\mm, \X,\Y}_{\HH(S)}\otimes_{\Or_{S,\Y}} H^0(\BB(\Omega_{S,\Y})) \ .$$
 A further  possibility is to introduce a base point as follows. Suppose that 
$t\in S(\Q)$, whose image in $S(\C)$ lies in  $ X\cap Y$ so that the evaluation $(\ref{eqn: specialisation})$ is defined. Let $\xi \in \Pe^{\mm, \X,\Y}_{\HH(S)}$
be differentially unipotent, and define the symbol `based at $t$'  by  \label{gloss: symbolatt}
$$\smb_t(\xi) = (\ev_t \otimes \smb) \Delta (\xi)  \qquad \in \qquad  \Pe^{\mm}_{\HH} \otimes_{\Q} H^0(\BB(\Omega_{S,\Y}))\ .$$
 This notion captures constants and satisfies similar properties to the symbol defined above. See for \cite{Duhr} for applications of such a construction.

\begin{rem}
Suppose that $\xi$ is a period of a mixed Tate variation. By this we mean that 
we can write $\xi = [ (\V_B, \V_{dR},c), \sigma, \omega]^{\mm}$, where $\gr_{n}^W (\V_B, \V_{dR},c)$
is zero if $n$ is odd, and  isomorphic to a direct sum of constant Tate elements $\Q(-k)_{/S}$ if $n=2k$ is even (images of the pull-back of the Tate objects
$\Q(-k)$ in $\HH$ along the structural map $\pi: S \rightarrow \Spec (\Q)$).  The connection underlying $\Q(-k)_{/S}$ is isomorphic to $(\Or_{S}, d)$ and is trivial.  It follows that the  element $\xi$ is automatically  differentially  unipotent with respect to the weight filtration 
$\V_n = W_{2n} \V_{dR}$. 
If, furthermore, $\xi$ is effective then we can apply  a version of the projection map
of \S$\ref{sectProjection}$ to associate to $\xi$ a de Rham period $\xi^{\dR}$,  which is necessarily unipotent, and take its symbol $\smb(\xi^{\dR})$.

Thus  we have shown that de Rham, and effective mixed Tate periods, always have symbols.  All the examples used in physics seem to be of this  special type.
\end{rem}

\subsection{Cohomological symbol} The bar complex is somewhat cumbersome. We can define a coarser version of  a symbol  of length $n$ 
by passing to the associated length-graded of the bar complex.

\begin{defn} Let $\xi \in \Pe^{\dR,\Y}_{\HH(S)}$ be a differentially  unipotent de Rham period of length $\leq n$.  Recall that this means the filtration $(\ref{Vfilt})$  satisfies $\V_n=\V$.   Define its  \emph{cohomological symbol} \label{gloss: cohomologicalsymbol}  in length $n$
to be its class
$$\csmb_n(\xi)  = [\smb(\xi)] \quad \in \quad \gr^{\ell}_n  H^0(\BB(\Omega_{S,\Y}))\ ,$$
where $\ell$ denotes the length filtration. 
 \end{defn}

 An Eilenberg-Moore spectral sequence implies  that the associated graded for the length filtration 
$ \gr^{\ell}  H^0(\BB(\Omega_{S,\Y})) \cong H^0 (\BB (H (\Omega_{S, \Y})))$
 is the bar complex
on the cohomology of $\Omega_{S, \Y}$, equipped with the trivial differential.   Therefore
$$ \csmb_n (\xi)  \in  \Or_{S,\Y} \otimes_{\Q} H^1(  \Omega_{S,\Y} ) ^{\otimes n} \ , $$
and  lies in the kernel of the map
\begin{eqnarray} \label{cohommap}
H^1(  \Omega_{S,\Y} ) ^{\otimes n}   &\To&  \bigoplus_{k}  H^1(  \Omega_{S,\Y} ) ^{\otimes k-1} \otimes_{\Q} H^2( \Omega_{S,\Y}) \otimes_{\Q} H^1(  \Omega_{S,\Y} ) ^{\otimes n-k-1}  \nonumber  \\ \qquad \qquad 
{[}\omega_1 | \ldots | \omega_n ] & \mapsto & \sum_k [ \omega_1 | \ldots  | \omega_{k-1} | \omega_{k} \wedge \omega_{k+1}  | \omega_{k+2} |  \ldots | \omega_n] \ .
\end{eqnarray} 

\begin{example} The cohomological symbol  is the  length $n$ part of the symbol, after replacing forms with their cohomology classes. In our  example $\ref{example symbol}$
it  gives 
$$\csmb_2(\xi) = \big[ [\omega_1] \big| [\omega_2]\big]\ . $$
This lies in the kernel of $(\ref{cohommap})$ since $[\omega_1] \wedge [\omega_2] = [\omega_1 \wedge \omega_2 ] = [-d \omega_{12}]=0$.
\end{example}

This invariant can be computed directly and more simply in the following way.
With the previous notations, consider the operator 
$$ \overline{N} = \gr_{\bullet} \,  (\nabla- d)  : \gr_{\bullet} \,  \V \To \gr_{\bullet-1} \V \otimes_{\Q} \Omega^1_{S,\Y}\ .$$
  Iterating it defines an operator
$$   (\overline{N})^{\otimes n} : \gr_n \V  \To \gr_0 \V  \otimes_{\Q} ( \Omega^1_{S,\Y})^{\otimes n}\ .$$
The vector $\omega$  defines a section in the associated graded via the map $\V \rightarrow \gr_n \V$. Similarly, 
since $\V_{-1}=0$, we have $\gr_0 \V =  \V_0 \subset \V$, and we can consider the image of the covector $f$ along the dual  map $\V^{\vee} \rightarrow \gr_0 \V^{\vee}$.  Now consider
\begin{equation} \label{csmbdef} \langle f,  (\overline{N})^{\otimes n} \omega\rangle \quad \in \quad  \Or_{S,\Y} \otimes_{\Q} ( \Omega^1_{S,\Y})^{\otimes n}\ .
\end{equation}
Because $dN+N \wedge N =0$, it follows that $d \overline{N}=0$ and $\overline{N} \wedge \overline{N} =0$.

\begin{ex}  Let $S = \Pro^n  \backslash D$ where $D = \cup_{i=0}^m D_i$ is a  union of $m+1\geq 1$ distinct  hyperplanes over $\Q$ and $Y = S(\C)$.  Let $f_i=0$ be an equation of $D_i$, where $f_i \in \Q[x_0,\ldots, x_n]$ is homogeneous  of degree one. 
A basis for $H_{dR}^1(S)$ is given by the cohomology classes of forms
$$ \omega_i  =  d  \log \Big({f_i\over f_0} \Big) \qquad \hbox{ for }  1 \leq i \leq m\ .  $$
A (cohomological) symbol is simply a linear combination of  tensor products of  $\omega_i$
which satisfies the integrability condition $(\ref{cohommap})$. Since this  case is mixed Tate, the length filtration on the bar construction  coincides with (one half of) the weight filtration, and is canonically split by the Hodge filtration. It follows either from this, or from formality of the cohomology of a hypersurface complement \cite{BrENS} \S3.2, that 
$H^0(\BB(\Omega^{\bullet}_{S})) \cong H^0(\BB(H^{\bullet}(\Omega_S))$ and so there is no significant difference between  symbols $\smb$ and their cohomological versions $\csmb_n$. They are integrable words in the one-forms $\omega_i$ (resp. their cohomology classes $[\omega_i]$).\footnote{Furthermore,  $H^1(\Pro^n;\Or)=0$, so the canonical extension $\overline{\V}$ of a unipotent vector bundle $\V$ is trivial, and we can take as de Rham fiber functor the global sections functor  $\V \mapsto \Gamma(\Pro^n, \overline{\V})\in \mathrm{Vec}_{\Q}$.}
In the case $n=1$, the integrability condition $(\ref{cohommap})$ is trivially satisfied. 
 Example:  
 $$\smb(\Li_{n}^{\dR}(x)) = [d\log (1-x)  |   d \log x | \ldots | d \log x]\ .$$
 This setting  covers much of the recent work of physicists on symbols. Furthermore, using the motivic fundamental group of such a space \cite{DeGo} 4.12,  one can define a notion of motivic iterated integrals. 
  \end{ex} 

\begin{rem} The theory of iterated integrals enables us to construct a map in the opposite direction and associate a family of motivic periods to  a symbol together with some extra data.  This discussion would take us too far afield.  A full treatment  should also incorporate the mixed Hodge structure on the reduced bar construction \cite{Ha} and a discussion of admissible variations and tangential base points.  
\end{rem} 

It is important to remark that  symbols  \emph{do not} have periods in their own right:  to define periods one requires a  path of integration, or at the very least a base-point (if one only wishes to define the associated single-valued periods).

\section{Some  geometric examples} \label{sect:  geometric examples}

 Let $S$ be as in \S\ref{sectFamilies} and let $\pi: S \rightarrow \Spec(\Q)$ be the structural map. Denote the Tate variation on $S$ by $ \Q(-n)_{/S}$. It is  the object   of $\HH(S)$ which is defined by  $\pi^* \Q(-n)$. Concretely, it is the triple $(\V_B, \V_{dR}, c)$ where $\V_{dR} = (\Or_{S}, d)$ is the trivial vector bundle with trivial  connection;    $\V_B$ is the constant local system $\Q$; and $c :\C= (\V_{dR}^{an})^{\nabla} \overset{\sim}{\rightarrow} \V_B \otimes \C =\C$ is multiplication by $(2 \pi i)^{n}$.  It has weight $2n$, and the Hodge filtration satisfies $F^{n} \V_{dR}= \V_{dR}$, $F^{n+1} \V_{dR}=0$.

\subsection{Mixed Tate motives over number fields} \label{sectMToverF}
In the  following example, let  $S= \Spec(F)$ for $F$ a number field. Since $S(\C)= \Hom(F,\C)$, it  is not geometrically connected   and therefore does not immediately fit into the framework of the previous paragraphs. However, one can define a category $H(S)$ of realisations without difficulty. An element of $H(S)$ consists of: $V_{dR} \in \mathrm{Vec}_{F}$ a vector space with the zero connexion; $\V_B$ a collection of vector spaces $V_{\sigma} \in \mathrm{Vec}_\Q$ for every $\sigma \in S(\C)$; and a  comparison $c= (c_{\sigma})_{\sigma}$, where  $c_{\sigma}: V_{dR} \otimes_{F,\sigma} \C \cong V_{\sigma} \otimes_{\Q} \C$ for every $\sigma\in S(\C)$, together with filtrations $W, F$ defined as before. There are  Frobenius isomorphisms $F_{\infty} : V_{\sigma} \cong V_{\overline{\sigma}}$   which  are compatible with the comparison isomorphisms via the action of complex conjugation. 
For example, the  object $\Q(0)_{/S}$ is  the triple $(\{V_\sigma\}, V_{dR}, c)$ where $V_{dR} = F$, and $V_{\sigma}= \Q $ for each $\sigma$, and  the isomorphisms $c_{\sigma}$ are the canonical ones. Taking  $\alpha \in F$, and $\tau: F \hookrightarrow \C$, we can view  algebraic numbers as $H(S)$-periods
$$ \alpha^{\mm,\tau} = [ \Q(0)_{/S} , \tau, \alpha]^{\mm} \in \Pe^{\mm, \tau, \Spec \, F}_{H(S)}\ .$$
Its period $\per (\alpha^{\mm,\tau})=\tau(\alpha) \in \C$.  Note that the motivic Galois action is trivial on $ \alpha^{\mm,\tau}$, since $\alpha^{\mm}$ is viewed as a family of periods over $F$ (but observe that  if $F$ is Galois, the Galois action on algebraic numbers could be retrieved by considering  the action of the automorphism group  of $\Spec(F)$).

Now  consider the category  $\MT(\Or)$ 
 of mixed Tate motives over a ring of integers $\Or$ in $F$ defined in \cite{DeGo}. The de Rham fiber functor is
 $\omega_{dR}: \MT(\Or) \rightarrow \mathrm{Vec}_{F}$, and there is a Betti fiber functor $\omega_{\sigma}: \MT(\Or) \rightarrow \mathrm{Vec}_{\Q}$ for 
 every $\sigma \in S(\C)$. Hence there is a functor $\MT(\Or) \rightarrow H(S)$, which is known to be fully faithful \cite{DeGo}. In this manner, we can view
 the periods of $\MT(\Or)$ as families of periods over $\Spec(F)$.\footnote{In order to capture better  the idea of a family of periods ramified over certain primes, 
 then the current set-up  in which we only consider Betti and de Rham information is inadequate. One could proceed  along the lines of \cite{DeP1}, \S1.18.}

 However, in this, mixed Tate, situation, the de Rham functor is in fact obtained from  a canonical fiber functor $\omega: \MT(\Or) \rightarrow \mathrm{Vec}_{\Q}$ 
 by extension of scalars  $\omega_{dR} = \omega \otimes_{\Q} F$. 
 This leads to a slightly different point of view. Define, for every $\sigma \in S(\C)$, a ring of motivic periods $\Pe^{\mm_{\sigma}}_{\MT(\Or)}$ over $\Q$ where $\mm_{\sigma}$ is the pair of fiber functors $(\omega_{\sigma}, \omega)$   in the manner of \S \ref{sectTannakianCase}. It is spanned by matrix coefficients $[M, x, v]^{\mm}$ where $M \in \MT(\Or)$, $ x\in M_{\sigma}^{\vee}$ and $v \in \omega(M)$. Every automorphism $\alpha \in \mathrm{Gal}(\overline{\Q}/\Q)$ defines an isomorphism $\Pe^{\mm_{\sigma}}_{\MT(\Or)} \cong \Pe^{\mm_{\sigma\alpha}}_{\MT(\Or)}$ via its action on the Betti class. 
 
 Following an identical argument to theorem \S\ref{thmdecompisom},  we deduce the following
 \begin{thm} The decomposition map gives a canonical isomorphism
 \begin{equation} \label{PhiforMT} 
 \Phi: \gr^C \Pe^{\mm_{\sigma},  +}_{\MT(\Or)} \overset{\sim}{\To} \Q[\Lef^{\mm}] \otimes_{\Q} T^c(\bigoplus_n K_{2n-1}(\Or) \otimes_{\Z} \Q(-n)_{dR})
 \end{equation} 
 in an identical manner to \S \ref{exampleMTcase}, which describes the structure of its motivic periods. 
  There is a de Rham version of the previous isomorphism:
 \begin{equation} \label{dRPhiforMT} \gr^C \Pe^{\dR}_{\MT(\Or)} \overset{\sim}{\To} \Q[\Lef^{\dR}, (\Lef^{\dR})^{-1}] \otimes_{\Q} T^c(\bigoplus_n K_{2n-1}(\Or) \otimes_{\Z} \Q(-n)_{dR})
 \end{equation}
 where $\dR$ means  $(\omega, \omega)$, for $\omega$ the canonical fiber functor. The projection map $\Pe^{\mm_{\sigma},  +}_{\MT(\Or)} \rightarrow
 \Pe^{\dR}_{\MT(\Or)}$ corresponds to the homomorphism  $\Lef^{\mm} \rightarrow 0$.
 \end{thm}

 \begin{rem} \label{remGonchsymbol}
  Goncharov considered  the image  (of what, in our language would be a $U_{\MT(\Or)}^{dR}$-period)  in  the de Rham version of $(\ref{dRPhiforMT})$ in the quotient 
 $$T^c (   K_1(F) \otimes_{\Z} \Q(-1)) \cong \bigoplus_{n\geq 0} (K_1(F) \otimes_{\Z} \Q)^{\otimes n} = \bigoplus_{n\geq 0} (F^{\times})^{\otimes n}\ ,$$
 (see discussion preceding lemma 3.7 in \cite{GoMTM}).   One can ignore the coradical grading in this case, since in this particular quotient  it is equivalent to  the weight-grading (this is false in general).  Thus Goncharov's notion of symbol is the homomorphism 
 $$ \Or(U_{\MT(\Or)}^{dR} ) \To \bigoplus_{n\geq 0} (F^{\times})^{\otimes n}\ . $$
 This map has a huge kernel and   loses  most of the    information about periods. For example, if $\Or= \Z$, then 
 this homomorphism is identically zero and all (unipotent de Rham) multiple zeta values map to zero.

 A version of this notion of symbol for variations, defined in \cite{GoSymbol}, \S1.3, is used in the physics literature, and  the `motivic amplitude' considered in \cite{Cluster} 
 is defined by extrapolation as an element of $(F^{\times})^{\otimes n}$ for $F$ a certain field. For the reasons above, this notion  loses information about periods and  does not apply   in the non-mixed Tate case.  It is not to be confused with the notion  of  motivic periods defined  here.
 \end{rem}

 \subsection{A family of examples} \label{sectExfamily} The following family of examples is  sufficient for the purposes of \cite{Cosmic}. Let  $D\subset X$ be a family of simple normal crossing divisors  relative 
 to  a smooth morphism $\pi: X\rightarrow S$, where $S$ smooth over $\Q$ and geometrically connected.  Furthermore, we assume that $\pi$ is topologically trivial on the underlying
 analytic varieties  (it is a locally trivial fibration of stratified varieties, according to \cite{Go-Ma}).  
 Let $j: X\backslash D \hookrightarrow X$ be the inclusion.  Define an object
 $H^n(X, D)_{/S}$ in the category $\HH(S)$
 as follows. Its Betti realisation is 
 $$H^n_B(X,D)_{/S} = R^n \pi_* j_{!} \Q$$
 where $\Q$ is the constant sheaf on $(X\backslash D)(\C)$.  Since $\pi$ is topologically trivial, this is a local system, and  its fibres at a point $s\in S(\C)$ are
 $H^n_B(X_s,D_s)$, where  $X_s, D_s$ denote the fibres of $X,D$.
  For de Rham, denote the irreducible components of $D$  by $D_i$,  for $i\in I$,  where $I$ is an ordered set, and write $D_J = \cap_{j\in J} D_i$
 for any $\emptyset \neq J\subset I$. Consider the double complex of sheaves of relative differential forms on $X$
 $$\Omega^{\bullet}_{D_{\bullet}/S} : \qquad\qquad \Omega^{\bullet}_{X/S} \To \bigoplus_{|J|=1} \Omega^{\bullet}_{D_J/S} \To \ldots \To \bigoplus_{|J|=n} \Omega^{\bullet}_{D_J/S}$$
 where the horizontal maps are determined by the usual rule: if $i_j$ denotes the inclusion of $D_{J\backslash \{i_j\} } \hookrightarrow D_{J }$,
 where $i_j$ is the $k\mathrm{th}$ element of $J$ then $i_j^*$ occurs with the sign $(-1)^k$.  The  $\Omega^{\bullet}_{D_J/S}$ denote the  direct images of the corresponding sheaves on $D_J$, and vanish outside $D_J$.

 Define the de Rham realisation by 
 $$H^n_{dR}(X,D)_{/S} = \mathbb{R}^n \pi_* ( \Omega^{\bullet}_{D_{\bullet}/S})\ .$$
It is the sheaf associated to the presheaf whose sections over an affine open $U\subset S$ are  the hypercohomology of $\Omega^{\bullet}_{D_{\bullet}/S}  (\pi^{-1}(U))$. It is a locally free sheaf of $\Or_S$-modules and its fibres at the point $s$ are  the relative de Rham cohomology groups:
 $$(H^n_{dR}(X,D)_{/S} )(s) =  H^n_{dR}(X_s,D_s)\ .$$
 It admits an integrable connection $\nabla$ by a relative version of \cite{Katz-Oda}. 
 To check the comparison isomorphism, denote by $\Q_J$ the constant sheaf $\Q$ on $D_J(\C)$, extended by zero to the whole of $X(\C)$.  The complex of sheaves
 $$\Q_{D_{\bullet}/S} : \qquad \qquad  \Q \To  \bigoplus_{|J|=1} \Q_{J} \To \bigoplus_{|J|=2} \Q_{J} \To \ldots \To \bigoplus_{|J|=n} \Q_{J}$$
where the sign conventions are exactly as defined for the complex $\Omega^{\bullet}_{D_{\bullet}/S}$, defines a resolution of $j_!\Q$.  The analytification 
$\Omega^{\bullet,an}_{D_{\bullet}/S}$ 
of $\Omega^{\bullet}_{D_{\bullet}/S} $
is a resolution of $\Q_{D_{\bullet}/S}\otimes \C$ over $S^{an}(\C)$. Using the  triviality of $\pi$, and arguing  as in \cite{De4}, Proposition 2.28, there is  a natural isomorphism
$$ c^{-1}: H^n_B(X,D)_{/S} {\otimes_{\Q}} \Or_S^{an} \overset{\sim}{\To} (H^n_{dR}(X,D)_{/S})^{an}\ .$$
  It is known  that $H^n_B(X,D)_{/S}$, equipped with its  weight filtration and Hodge filtration from  $c H^n_{dR}(X,D)_{/S}$, is a variation of mixed Hodge
 structure. It is effective: the  Hodge numbers on every fiber satisfy $h_{p,q}=0$ if $p$ or $q$ are $<0$.

\subsection{Face maps}\label{sect:  facemapsgeneral}
 With the  above  notations, let $D_I =\cap_{i\in I} D_i$  denote an intersection of  ireducible components of $D$ of codimension $k$, 
 and let $D^I = \bigcup_{j \notin I} D_j$ denote the union of all remaining irreducible components. The pair $D_I \supset D_I \cap D^I$ satisfies the conditions of the previous
 paragraph.

 There are natural morphisms, that we shall call  \emph{face maps},  in the category $\HH(S)$ 
 \begin{equation}  \label{eqn: facemapsgeneral} 
 H^{n-k}( D_I, D_I \cap D^I)_{/S} \To H^{n}( X, D)_{/S} \ .
 \end{equation} 
 For the de Rham (respectively Betti) realisation, this is given by the inclusion of complexes
 $\Omega^{\bullet}_{D^I_{\bullet}/S} \rightarrow \Omega^{\bullet}_{D_{\bullet}/S}$ 
 (respectively   $\Q_{D^I_{\bullet}/S}\rightarrow \Q_{D_{\bullet}/S}$).

On the other hand, let $0\leq k \leq n-1$  and let $D^{(k)} = \bigcup_{|I|=n-k}  D_I$ denote the $k$-dimensional skeleton of $D$. 
Then $H^k(D^{(k)})_{/S}$ defines an object of $\HH(S)$ given by truncating the complexes 
$  \Omega^{\bullet}_{D_{\bullet}/S}$  and $\Q_{D_{\bullet}/S}$  on the left so that the non-zero components are  $|J| \geq n-k$.
The inclusion of these complexes similarly defines
\begin{equation}  \label{skelmap}  H^k(D^{(k)})_{/S} \To  H^n(X, D)_{/S}\ . 
\end{equation} 
The case $k=n-1$ is the boundary map in the relative cohomology sequence 
$$\cdots\To H^{n-1}(X)_{/S} \To H^{n-1}(D^{(n-1)})_{/S} \To H^n(X,D)_{/S}\To  H^n(X)_{/S}\To  \cdots$$

 The face maps $(\ref{eqn: facemapsgeneral})$ factor through  $(\ref{skelmap})$.
 \subsection{Weight filtration} 
 By strictness, apply the weight functor to the previous long exact cohomology sequence to obtain an exact sequence 
 $$\cdots  \rightarrow W_k H^{n-1}(X)_{/S}  \rightarrow W_k H^{n-1}(D^{(n-1)})_{/S} \rightarrow W_k H^n(X,D)_{/S} \rightarrow W_k H^n(X)_{/S} \rightarrow \cdots$$
Since $H^n(X)$ has smooth fibers, it has weights concentrated in degrees between $n$ and $2n$  (by   \cite{De2} 8.2.4). It follows that
$$W_k H^{n-1}(D^{(n-1)})_{/S} \To W_k H^n(X,D)_{/S}$$
is surjective if $k={n-1}$, and an isomorphism if $k< n-1$. 
 The following proposition generalises the previous fact, and is presumably well-known. We include a quick proof for completeness.  It requires the following lemma.
 
 \begin{lem} \label{lemsurjcohom}
  Let $\phi: C \rightarrow C'$ be a morphism of cochain  complexes such that $\phi_i: C_i \rightarrow C'_{i}$ is surjective, and $\phi_j : C_j \rightarrow C'_j$
  are isomorphisms for all $j>i$. Then the induced maps on cohomology have the same properties: $ H^i(C) \rightarrow H^i(C')$ is surjective, and $ H^j(C) \overset{\sim}{\rightarrow} H^j(C')$ is an isomorphism  for all $j>i$. 
 \end{lem} 
 \begin{proof} Exercise.
 \end{proof}
 \begin{prop} Let $m<n$.  The  map
  $$W_k H^m(D^{(m)})_{/S} \To W_k H^n(X, D)_{/S}$$
 is surjective if $m=k$ and is an isomorphism if $k < m$. 
\end{prop}
 
  \begin{proof}  Since the comparison isomorphism respects the weight filtration, it suffices to verify the statement in the Betti realisation. For this it is enough
  to check the statement on every fiber: for every $t \in S(\C)$, 
  $$   W_k H_B^{m} ( D^{(m)})_t \To   W_k H_B^n(X, D)_t  $$
 is an isomorphism for $k<m$ and surjective for $k=m$. 
 To alleviate the notation, write $B_I= (D_I)_t(\C)$  and $Y = X_t(\C)$. Consider 
 \begin{equation} \label{eqn: inproofHmBtoHYB}
 H^{m} (B^{(m)})  \To   H^n(Y, B)  \ .
 \end{equation}
   There  are relative cohomology spectral sequences
  \begin{equation} \label{relcomss}
  E_1^{p,q} (Y)= \bigoplus_{|J|=p} H^q(B_J)   \quad \Longrightarrow \quad  H^{p+q}(Y,B)
  \end{equation} 
  and
  $$E_1^{p,q}(B^{(m)}) = \bigoplus_{|J|=p+n-m } H^q(B_J)   \quad \Longrightarrow \quad  H^{p+q}(B^{(m)})$$
  The morphism $(\ref{eqn: inproofHmBtoHYB})$ induces a  map of spectral sequences
  $$ E_1^{p,q}(B^{(m)}) \To E_1^{p+n-m,q} (Y)$$
  which is the identity  on each summand $H^q(B_J)$.  
    Let $j \leq k$ and  apply the functor $\gr^W_j$. It is exact, giving a morphism  of spectral sequences
  \begin{equation} \label{Cpnmq} 
  \gr^W_{j} E_r^{p,q}(B^{(m)}) \To \gr^W_j E_r^{p+n-m,q} (Y) \ .
  \end{equation}  
   Since $B_J$ is smooth,   $H^q(B_J)$ has weights in the interval
  $[q,2q]$ by \cite{De2} 8.2.4, and therefore  both sides of $(\ref{Cpnmq})$ vanish  for all $q >  j$.
  The entries of $(\ref{Cpnmq})$ for $r=1$ are identical in the range   $p\geq 0$.  By running the spectral sequence, and applying the previous lemma, 
  one verifies by induction on $r$ that $(\ref{Cpnmq})$ is an isomorphism  for $p\geq r-1$ or $p+ q \geq m+1$ and surjective for other values of $p\geq 0$.
      \end{proof}
  
  The spectral sequence $(\ref{relcomss})$  implies the 
      \begin{cor} Let $0 \leq k \leq n$. Then $\gr^W_k H^n(X,D)_{/S}$ is  isomorphic to a subquotient of $\bigoplus_{|I| \geq  n-k}\gr^W_k \ H^{n-|I|} (D_{I})_{/S}$.            \end{cor}

 Putting $k=0$, $m=1$ in the previous proposition gives the following corollary.
  \begin{cor} We have
 \begin{equation}  \label{eqn: Weight0HXD}
 W_0 H^n(X,D)_{/S} \cong  \Q(0)^{\oplus m}_{/S}
 \end{equation} 
 where $m = \dim_{\Q}  \widetilde{H}^{n-1}(D_t(\C))$  for any $t\in S(\C)$. 
 In particular, the motivic periods of $H^n(X,D)_{/S}$ of weight zero are constant and rational. 
      \end{cor}
 \begin{proof}  Note that if $|I|=n$ then  $D_{I} \cong \Spec\, S$ and $H^0(D_{I})_{/S} = \Q(0)_{/S}$, so $(\ref{eqn: Weight0HXD})$ holds for some $m$. To determine $m$,
pass to the Betti realisation at the fiber $s$. With the notations of the previous proposition, we have 
$$\gr^W_0 H^n(Y,B) \cong  E_2^{n,0}(Y) =   \mathrm{coker}\,\, \big( \bigoplus_{|I|=n-1} H^0(B_{I})\To\bigoplus_{|I|=n} H^0(B_{I})\big)  $$
Since the $B_I$ are connected, the dimension of this cokernel is the dimension of the reduced cohomology $\dim \widetilde{H}^{n-1}(B)$. 
 \end{proof}
  
 We considered earlier the case $S= \Spec \Q$, and $D$  normal crossing, rather than simple normal crossing. Then a similar argument proves that the weight zero part of $H^n(X,D)_{/S}$
is the (realisation of a) constant Artin motive. The action of $\mathrm{Gal}(\overline{\Q}/\Q)$ upon its Betti realisation  is induced  by the Galois action on the points  $\bigcup_{|I|=n} D_I(\C)$.

 Under some further assumptions, the face maps provide information about  the mixed Hodge structure on $H^n(X,D)_{/S}$ in low weights.

  \begin{prop} \label{prop: weightsandfacemaps}  Suppose that $H^k((D_{I})_s)=0$ for all $k> n- |I|$ (for example, if the
  strata $D_I$ have  affine fibres of dimension $n-|I|$).

  Let $m<n$. Then for any $k\leq m$,  the sum of  face maps
  $$ \bigoplus_{|I|=n-m}  W_k H^{m} (D_I,  D_I \cap D^I)_{/S} \To   W_k H^n(X, D)_{/S}\ .$$
 is surjective.  \end{prop}
  \begin{proof} A similar spectral sequence argument as in the previous proposition.
  The assumption implies that $E_1^{p,q}$ vanishes above the diagonal, i.e.,  for all $p+q >n$. 
   Now consider the map of spectral sequences induced by the sum of face maps. Take their fibers as in the previous proposition and apply $\gr^W_j$ for $j\leq k$.  On $E^{p,q}_1$ terms we obtain the natural map
  \begin{equation} \label{natssE1grw}     \gr^W_j     \bigoplus_{|I|=n-m}   \bigoplus_{J \supset I , |J|=p} H^q(B_J) \To    \gr^W_j      \bigoplus_{ |J|=p} H^q(B_J)   \ . \end{equation}
 This is an isomorphism above the diagonal ($p+q>n$) since both sides are identically zero in that region by assumption. 
 Furthermore, both sides vanish of $(\ref{natssE1grw})$  in the region $q> j$ (since the $B_J$ are smooth and $H^q(B_J)$ has weights in the interval $[q,2q]$) and therefore in particular in the region $q> m$. It follows that $(\ref{natssE1grw})$ is 
 also surjective along the diagonal $p+q=n$.  
   By induction on $r$ and   lemma
 \ref{lemsurjcohom}, the induced map is surjective on the diagonal for all $r$. 
    \end{proof} 

\subsection{The periods}
In the  situation of \S\ref{sectExfamily}, let  $s\in S(\C)$ and let $\sigma_s\subset X_s(\C)$ be a topological $n$-chain whose boundary is contained in $D_s(\C)$.  It defines a relative homology class
$[\sigma_s] \in H^n_B (X_s(\C), D_s(\C))^{\vee} = (H^n_B (X,D)_{/S})^{\vee}_s.$
 By local triviality, there exists a
small neighbourhood $N\subset S(\C)$ of $s$, and isomorphism
$$(X(\C), D(\C)) \cap \pi^{-1}(N) \cong N \times (X_s(\C), D_s(\C))\ .$$
Via this isomorphism, the chain 
  $\sigma_s$ uniquely extends to  a family of topological $n$-chains $\sigma_t \subset X_t(\C)$ whose boundaries are contained in $D_t(\C)$ for all $t\in N$. 

For simplicity, let us consider the  particular case when  we are given a global form $\omega \in \Omega^n_{\pi^{-1}(U)/U}$ for some Zariski open $U \subset S$
and suppose that the fibres of $X$ are of dimension $n$. Since the restrictions of $\omega$ to components of $D$ vanish for  reasons of dimension,  it defines a relative class $[\omega ] \in  \Gamma(U, H^n_{dR} (X, D))$, and
$$ \xi = [ H^n(X,D)_{|S}, [\sigma_s] , [\omega] ]^{\mm} \quad \in \quad \Pe^{\mm, \{s\}, \Y}_{\HH(S)}$$
for any  $Y\subset U(\C)$. For any $t\in U(\C)$, let $\omega_t$ denote the restriction of the form $\omega$ to a fiber $t$.  The period is then
$$ \per(\xi)(t) =  \int_{\sigma_t} \omega_t $$
for all $t$ in the open set  $U(\C) \cap N$, and is extended by analytic continuation to  a meromorphic function on the universal covering of $S(\C)$ based at $s$.

 \subsection{Example: iterated integrals on the projective line minus 3 points}
 Let $S= \Pro^1 \backslash \{0,1,\infty\}$ throughout this section. 
 
 \subsubsection{Motivic logarithm} 
 Let $Z= S \times \GG_m$ and $\pi: Z \rightarrow S$ the projection 
 onto the first factor. If $x$ is the coordinate on $S$, and $y$ the coordinate on $\GG_m$, let $D= \{y=1 \} \cup \{y=x\}$. 
 Let $\X\subset S(\C)$ denote the real interval $(0,1)$, and for all $x\in \X$, let   $\sigma_x \subset \GG_m(\C)$ denote the straight path  from $1$ to $x$ in the fiber over $x$. Its  image is $\{x\}\times [1,x]$.  Define the motivic logarithm by 
 $$\log^{\mm}(x) = [ H^1(Z,D)_{|S}, [\sigma_x], [ \textstyle{dy \over y}]]^{\mm} \qquad \in \qquad \Pe^{\mm}_{\HH(S), \X,  \Y}$$ 
 where $\Y = S(\C)$. Note that $\Or_{S,\Y}= \Or_S = \Q[x, x^{-1}, (1-x)^{-1}]$ (this was example $(3)$ in \S \ref{sectFiberfunctors}). 
  Its period is the logarithm as expected:
  $$\per(\log^{\mm}(x)) = \log (x) = \int_{\sigma_x} {dy \over y} \qquad \hbox{ for } x \in \X \ . $$ 
  The long exact cohomology sequence reduces to 
 $$ 0 \To \Q(0)_{|S} \To H^1(Z,D)_{|S} \To \Q(-1)_{|S}\To 0 \ , $$
 and with respect to the de Rham  basis $[ {dy\over y}]$, $[{dy \over x-1}]$ and Betti basis $[\sigma_x], [\gamma]$, where $\gamma$ is a small loop around $0$ in the fiber, 
  the period matrix is given by the identical  formula to $(\ref{logmotivicperiodmatrix})$,  where $\Lef^{\mm}$ is now viewed a constant family of periods over $S$ and $\alpha=x$. 
The coaction satisfies $\Delta \log^{\mm}(x) = \log^{\mm}(x) \otimes \Lef^{\dR} + 1 \otimes \log^{\dR}(x)$,  and the Galois group is $\GG_a \rtimes \GG_m$, and is in fact defined over $\Q$ in this case.   
 Recall that $\log^{\dR}(x)$ does not have a period, but has a single-valued period $2 \log |x|$.
 
 \subsubsection{Motivic fundamental groupoid} Now consider the  ind-object  of $\HH(S)$ 
 $$ \Or(\pi_1^{\HH}(S,  \tone_0, \bullet  ) ): = ( \Or(\pi_1^{un}(S, \tone_0, \bullet  )), \Or(\pi_1^{dR}(S ))\otimes_{
 \Q} \Or_S, \mathrm{comp})  $$
 where $\tone_0$ is the tangential base-point at $0$ with unit length.  The first entry (Betti local system) is the affine ring of the unipotent completion of the torsor of paths beginning at $\tone_0$  on $S(\C)$ 
 and defines a local system on $S(\C)$: its fiber at  a point $x\in S(\C)$ is $  \Or(\pi_1^{un}(S, \tone_0, x))$ with the action of $\pi^{\mathrm{top}}(S(\C), x)$.  The second entry does not in fact depend on basepoints and  is the affine ring of the (unipotent) de Rham fundamental group. It is a shuffle algebra on two generators
 $$\Or(\pi_1^{dR}(S)) \cong  T^c (\Q e_0 \oplus \Q e_1 ) $$
  where $e_0,e_1$ correspond to the one-forms ${dx \over x}$ and ${dx \over 1-x}$. It defines a  trivial vector bundle on $S$, and is equipped with the Kniznhik-Zamolodchikov 
  connection
   $$\nabla e_{i_1} \ldots e_{i_n}  = \begin{cases}
e_{i_1}\ldots e_{i_{n-1}} \otimes {dx \over x} \qquad  \,\,\, \hbox{ if } i_n=0 \\   
   e_{i_1}\ldots e_{i_{n-1}} \otimes {dx \over 1-x} \qquad \hbox{ if } i_n=1 \ .  
   \end{cases} $$
   It is clearly unipotent with respect to the weight filtration. The weight filtration on $\Or(\pi_1^{dR}(S))$ is in fact a grading in this case (it is split by the Hodge filtration since it is mixed Tate), and  the grading assigns degree two to 
    $e_0,e_1$.

      We take $\X= (0,1)$ and $\Y=S(\C)$ as before. 
 Define the motivic multiple polylogarithm  to be the family of motivic periods: 
 \begin{equation} \label{defnmultiplepolylog}
 \Li^{\mm}_w(x) =   [ \Or(\pi_1^{\HH}(S,  \tone_0, \bullet )) , \sigma_x, w]^{\mm} \quad \in \quad \Pe^{\mm, \X, \Y}_{\HH(S)} 
 \end{equation} 
 where $w $ is a word in $e_0,e_1$ and $\sigma_x$ is the straight line path from $\tone_0$ to $x$.  The path $\sigma_x$ is viewed as an element of $\Or(\pi_1^{un}(S, \tone_0, x ))^{\vee}$
 via the natural map 
 \begin{equation} \label{naturalpi1mapend} 
 \pi^{\mathrm{top}}_1(S(\C), \tone_0, x) \To \pi_1^{un}(S, \tone_0, x)(\Q)\ .
 \end{equation}
 The period of $\Li^{\mm}_w(x)$ is  $\Li_w(x)$, which is the iterated integral $\int_{\sigma_x} w$, and we have $\Li^{\mm}_{e_0}(x) = \log^{\mm}(x)$.  More generally we write $\Li^{\mm}_n(x)$ for $\Li^{\mm}_{e_1 e_0^{n-1}}(x)$, and call it the motivic classical polylogarithm. 
  The connection satisfies    $$\nabla \Li^{\mm}_{we_s } (x) =(-1)^s \Li^{\mm}_w(x) \otimes {dx \over x - s} \qquad \hbox{ where } s \in \{0,1\}\ , \    w \in \{e_0,e_1\}^{\times}     $$
 and $\Li^{\mm}_{\emptyset}(x)$ is the constant motivic period $1$.  
  The de Rham motivic multiple polylogarithms  are defined by 
  $$
 \Li^{\dR}_w(x) =   [ \Or(\pi_1^{\HH}(S,  \tone_0, \bullet )) , \varepsilon, w]^{\mm} \quad \in \quad \Pe^{\dR, \Y}_{\HH(S)} 
 $$
  where $\varepsilon: \Or(\pi_1^{dR}(S)) \rightarrow \Q$ is the augmentation map (it sends every non-trivial word $w$ to zero). The de Rham versions  are the images
 of $\Li^{\mm}_w(x)$ under the projection map \S\ref{sectProjection}.    Our definition of the symbol  satisfies, as expected, 
  $$\smb (\Li^{\dR}_{w}(x))  = w \qquad \in \qquad T^c (\Omega^1(S)) \cong H^0(\BB(\Omega^1(S)))\ .$$
The  single-valued versions of $\Li^{\dR}_w(x)$ are obtained in an identical way to \cite{SVMP}, by simply writing superscript $\mm$ everywhere  (with a possible sign difference of $(-1)^{|w|}$).

Recall that for every word $w$ we defined the (image in the ring of $\HH$-periods) of the shuffle-regularised motivic multiple zeta values: 
$$\zetam(w) =  [ \Or(\pi_1^{\HH}(S,  \tone_0, -\tone_1)) , \varepsilon, w]^{\mm} \quad \in \quad \Pe^{\mm}_{\HH}\ ,$$ 
where $-\tone_1$ is the tangent vector $-1$ at the point $1$. Denote their pull-backs to $\Pe^{\mm,X,\Y}_{\HH(S)}$ via the structural map $\pi: S \rightarrow \mathrm{Spec}\, \Q$ by the same symbol. Let 
$$\Pe^{\mm, +}_{\mathcal{M}_{0,4}}  \qquad \subset \qquad \Pe^{\mm, X, \Y}_{\HH(S)}$$
 denote the $\Or_S$-module  generated by the $\zetam(w), \Li^{\mm}_w(x)$ and $\Lef^{\mm}$.  Then $\Pe^{\mm,+}_{\mathcal{M}_{0,4}}$ is a ring of motivic periods which is stable under the monodromy action of the fundamental group of $S(\C)$ at the base-point $\tone_0$. 
 More generally, one can consider in a similar manner the  motivic periods given by iterated integrals on the moduli spaces of curves $\mathcal{M}_{0,n}$ whose periods are multiple polylogarithms in several variables. 
For example,  the space of functions $H$ used in the analytic bootstrap for the 6-point function in the planar limit of $N=4$ SYM theory  described in the introduction of \cite{Cosmic} is contained in the space  $\Pe^{\mm,+}_{\mathcal{M}_{0,6}}$.

\begin{rem}
These examples can be expressed geometrically in the spirit of \S\ref{sectExfamily}  using Beilinson's construction for the unipotent fundamental group, and  indeed  defines a variation of mixed Tate motives  in the sense of \cite{DeGo}, 4.12. \end{rem}

If $t$ is a rational point on $S$ (or even a tangential base point), we can define the evaluation $\ev_t$ at the point $t$.  In this 
situation, the de Rham coaction $(\ref{H(S)coaction})$   commutes with evaluation at $t$:
$$\Delta  \ev_t( \Li^{\mm}) =  ( \ev_t \otimes \ev_t)  \Delta \Li^{\mm}(x)\ . $$
The follows from the triviality of  the de Rham vector bundle $\Or(\pi_1^{dR}(S))$  and 
 equation $(\ref{eqn: coaction})$. The unipotent coaction can be computed by transposing a formula due to Goncharov \cite{Go} to  this setting. By way of  example, the de Rham coaction on the   motivic dilogarithm satisfies
 $$\Delta \Li_2^{\mm}(x) =\Li^{\mm}_2(x)  \otimes  (\Lef^{\dR})^2  +  \Li_{1}^{\mm}(x) \otimes \Lef^{\dR} \log^{\dR}(x)   + 1 \otimes \Li^{\dR}_2(x) $$
 and was discussed in further  detail in example $\ref{exdilog}.$  The unipotent coaction is obtained by replacing $\dR$ in the right-hand terms by $\uu$, and noting that $\Lef^{\uu}=1$.

    In conclusion,  the notion of families of  motivic periods  provides a  framework for motivic multiple polylogarithms which  includes constants and is sufficient for  many applications to high-energy physics. It  satisfies  the properties conjectured in \cite{Duhr}.

 \section{Glossary of non-standard terms} \label{sectGlossary}

 \vspace{0.05in}
\noindent  Types of  periods:
\vspace{0.02in}

      \emph{$\HH$-periods} : page       
       \hfill{  \emph{$\HH$-de Rham} : page \pageref{gloss: mixedTate}, \S \ref{gloss: mixedTate}}

      \emph{motivic} : page \pageref{gloss: motperiod}, \S \ref{gloss: motperiod}  
   \hfill{    \emph{effective} : page   \pageref{gloss: HdeRham}, \S  \ref{gloss: HdeRham} }

      \emph{mixed Tate} : page \pageref{gloss: mixedTate}, \S \ref{gloss: mixedTate}
  \hfill{    \emph{single-valued}: page  \pageref{gloss: singlevalued1}, \S \ref{gloss: singlevalued1}   }

     \emph{semi-simple}: page  \pageref{gloss: semisimple}, \S \ref{gloss: semisimple}
   \hfill{ \emph{unipotent}: page  \pageref{gloss: unipotentperiods}, \S  \ref{gloss: unipotentperiods}}
     
       \emph{primitive, stable}: page   \pageref{gloss: stable}, \S   \ref{gloss: stable}

 \vspace{0.05in}
 \noindent Invariants  of periods and auxiliary constructions:
\vspace{0.02in}

      \emph{conjugates}: page \pageref{gloss: conjugates}, \S \ref{gloss: conjugates}  
  \hfill{  \emph{rank}: page \pageref{gloss: rank}, \S \ref{gloss: rank} }

     \emph{Hodge numbers}: page   \pageref{gloss: Hodgenumbers}   , \S \ref{gloss: Hodgenumbers} 
   \hfill{   (\emph{polynomial}: page  \pageref{gloss: Hodgepoly}, \S   \ref{gloss: Hodgepoly}   
   ) }

      \emph{period matrix}: page  \pageref{gloss: periodmatrix}, \S \ref{gloss: periodmatrix} 
   \hfill{   (\emph{single-valued}: page \pageref{gloss: singlevaluedperiodmatrix1}, \S \ref{gloss: singlevaluedperiodmatrix1}) }
     
     \emph{determinant}: page \pageref{gloss: determinant}, \S \ref{gloss: determinant} 
 \hfill{    \emph{unipotency degree}: page  \pageref{gloss: unipdegree}, \S \ref{gloss: unipdegree}               } 
    
      \emph{Transcendence dimension}: page \pageref{gloss: transdim}, \S \ref{gloss: transdim} 
  \hfill{  \emph{Galois group}: page \pageref{gloss: GaloisGroup} , \S \ref{gloss: GaloisGroup} }

   \emph{Decomposition into primitives}: page  \pageref{gloss: decompprim}, \S\ref{gloss: decompprim}; page \pageref{gloss: decompprim2}, \S \ref{gloss: decompprim2} 
   
    \vspace{0.05in}
    \noindent Operations on periods
        \vspace{0.02in}
        
     \emph{antipode}: page  \pageref{gloss: U-antipode}, \S \ref{gloss: U-antipode}
   \hfill{
   \emph{projection map}: page \pageref{sectProjection}, \S \ref{sectProjection}  }

   \vspace{0.05in}
\noindent  Families of periods   
      \vspace{0.02in}
     
        \emph{monodromy homomorphism/action}: page \pageref{gloss: monodromyhomorphism}, \S \ref{gloss: monodromyhomorphism}
   , $(\ref{monodromyonPe})$, page  \pageref{gloss: monodromyaction}, \S  \ref{gloss: monodromyaction}

     \emph{constant map}: page \pageref{gloss: constantmap}, \S \ref{gloss: constantmap} 
    \hfill{          \emph{evaluation map}: page \pageref{gloss: evaluationmap}, \S \ref{gloss: evaluationmap} }

     \emph{Weight/Hodge  filtration}: page  \pageref{gloss: weightfiltration2}, \S \ref{gloss: weightfiltration2} 
  \hfill{    \emph{connection}: page \pageref{gloss: connection}, \S \ref{gloss: connection} }

     \emph{period homomorphism}: page  \pageref{gloss: periodmap}, \S \ref{gloss: periodmap}
     
 \vspace{0.05in}
\noindent Symbols
 \vspace{0.02in}

       \emph{symbol}: page \pageref{gloss: symbol}, \S\ref{gloss: symbol} 
    \hfill{    \emph{symbol at a point}: page \pageref{gloss: symbolatt}, \S\ref{gloss: symbolatt} }

     \emph{cohomological symbol}: page  \pageref{gloss: cohomologicalsymbol}, \S \ref{gloss: cohomologicalsymbol}
 \hfill{    \emph{length}: page \pageref{gloss: length}, \S\ref{gloss: length} } 
      
  \emph{differentially unipotent/unipotent monondromy}:  page \pageref{gloss: diffunip}, \ref{gloss: diffunip}

\end{document}